\documentclass[12pt]{amsart}
\usepackage[utf8]{inputenc}
\usepackage{mathrsfs} 
\usepackage{amssymb}
\usepackage{bm}
\usepackage{graphicx}
\usepackage[centertags]{amsmath}%\href{}{}
\usepackage{amsfonts}
\usepackage{amsthm}
\usepackage{enumerate}
\usepackage{tocvsec2}
\usepackage[dvipsnames]{xcolor}
\usepackage[margin=.75in]{geometry}

\newtheorem{theorem}{Theorem}[section]
\newtheorem*{theorem*}{Theorem}
\newtheorem*{xrem}{Remark}
\newtheorem{corollary}[theorem]{Corollary}
\newtheorem{lemma}[theorem]{Lemma}
\newtheorem{rem}[theorem]{Remark}

\newtheorem{proposition}[theorem]{Proposition}
 
\newtheorem{problem}{Problem}[section]

\newtheorem*{question*}{Question}

\newtheorem*{cl}{Claim}
\theoremstyle{definition}
\newtheorem{definition}[theorem]{Definition}

\newcommand{\rr}{\mathbb{R}}
\newcommand{\nn}{\mathbb{N}}
\newcommand{\ee}{\varepsilon}
\newcommand{\supp}{\mathrm{supp}}

\newcommand{\tim}{\widehat{\otimes}}

\usepackage{tikz}
\usetikzlibrary{patterns}
\begin{document}

\title[Isomorphic classification of projective tensor products of $C(K)$ spaces]{Isomorphic classification of projective tensor products of spaces of continuous functions}

\author{R. M. Causey}
\email{rmcausey1701@gmail.com}

\author{E. M. Galego}
\address{University of S\~ao Paulo, Department of Mathematics, IME, Rua do Mat\~ao 1010,  S\~ao Paulo, Brazil}
\curraddr{Department of Mathematics and Statistics,}
\email{eloi@ime.usp.br}

\author{C. Samuel}
\address{Aix Marseille Universit\'e, CNRS, I2M, Marseille, France}
%\curraddr{Department of Mathematics and Statistics,}
\email{christian.samuel@univ-amu.fr}


\begin{abstract}
 The isomorphic classification is an important question in Banach spaces theory. Guided by a classic result of Bessaga and Pe\l czy\'{n}ski on spaces of continuous functions,  we solve in this paper the problem of obtaining a complete  isomorphic classification of projective tensor products of spaces of scalar continuous functions on infinite countable compact metric spaces. More precisely, we prove that for  any countably infinite  compact metric  spaces $K$ and $L$, the space  $C(K)\widehat{\otimes}_\pi C(L)$ is isomorphic to exactly one of the spaces $C(\omega^{\omega^\xi})\widehat{\otimes}_\pi C(\omega^{\omega^\zeta})$, $0\leqslant \zeta\leqslant \xi<\omega_1$. 

Since the classic Szlenk index proves to be an ineffective tool to solve this problem, our main task throughout the work was to first introduce  auxiliary functionals  $\mathbf{s}_\xi$ and then special functionals $\mathbf{g}_{\xi, 1+\zeta}$ on the spaces $C(K)\widehat{\otimes}_\pi C(L)$   and from there obtain upper and lower bounds  for $\mathbf{g}_{\xi, 1+\zeta}$     with the help of a  Grothendieck's constant. In this way the functionals $\mathbf{g}_{\xi, 1+\zeta}$ became capable of differentiating   the classes of isomorphisms  between spaces  $C(\omega^{\omega^\xi})\widehat{\otimes}_\pi C(\omega^{\omega^\zeta})$, $0\leqslant \zeta\leqslant \xi<\omega_1$.

\end{abstract}

\maketitle

\tableofcontents

\section{Introduction}

We shall use the standard notations and terminology
of Banach space theory (see e.g. \cite{JL}). For $K$ a compact Hausdorff space, we
denote by $C(K)$ the Banach space of all continuous scalar valued functions
defined on $K$ and endowed with the supremum norm. If  $\alpha$ is an ordinal number, then $[0, \alpha]$ denotes the interval $\{\eta: 0 \leq \eta \leq \alpha \}$  endowed with the
order topology.  The space $C([0, \alpha])$ will be denoted by $C(\alpha)$. $\omega$ denotes the first infinite ordinal and $\omega_1$ the first uncountable ordinal.

In 1920 Mazurkiewicz and Sierpi\'{n}ski \cite{MS} took a big step towards obtaining an isomorphic classification of the spaces $C(K)$  when $K$ are countable  compact metric  spaces. In fact they proved that for every countable  metric compact  space there exists an ordinal $\alpha< \omega_1$ such that $K$ is homeomorphic to $[0,\alpha]$.  Forty years later, Bessaga and Pe\l czy\'{n}ski \cite{BP}, \cite[p. 1560]{Rosenthal} established this classification by proving that if $K$ is an infinite countable  metric compact  space then $C(K)$ is isomorphic to $C(\omega^{\omega^\xi})$ for exactly one  ordinal $\xi<\omega_1$. 

The  isomorphic classification of separable spaces  $C(K)$  was completed in 1966 by Milutin when he proved that if  $K$ is an uncountable  compact metric  space then $C(K)$ is isomorphic to $C([0,1])$, where $[0 ,1]$ is the interval of real numbers with the usual topology \cite{M}, \cite[p. 379]{Semadeni}.

On the other hand, with the development of the injective tensor product  $X\widehat{\otimes}_\epsilon Y$        and  projective tensor product   $X\widehat{\otimes}_\pi Y$  of Banach spaces $X$ and $Y$, mainly through Grothendieck's works \cite{Grothendieck}, \cite{Gro} the study of the geometry of these new Banach spaces  has been  widely considered and has had an enormous impact on the theory of Banach spaces (see, e.g., the survey paper \cite{Pi}).  In particular,  if $K$ and $L$ are compact Hausdorff spaces,   many things are already known about  the spaces $C(K)\widehat{\otimes}_\epsilon C(L)$, but due to the somewhat intractable nature of projective tensor products
there are still many things to know about the geometric structure of the spaces   $C(K)\widehat{\otimes}_\pi C(L)$ even in the simplest case where $C(K)$ and $C(L)$ are isomorphic to $C(\omega)$  \cite{Unexpected}.

In this line of research, we only point out that the injective tensor products between spaces  $C(\omega^{\omega^\xi})$, $\xi<\omega_1$, do not present anything new in the sense that these products are still isomorphic to some spaces  $C(\omega^{\omega^\eta})$ with $\eta< \omega_1$. In fact, it is well known that $C(\omega^{\omega^\xi}) \widehat{\otimes}_\epsilon C(\omega^{\omega^{\zeta}})$ is isomorphic to    $C(\omega^{\omega^{max \{\xi, \zeta \}}}),$ see for instance \cite[Lemma 2.4]{ADG}. Consequently, by the Bessaga and Pe\l czy\'{n}ski theorem, the isomorphic classification of spaces  $C(\omega^{\omega^\xi}) \widehat{\otimes}_\epsilon C(\omega^{\omega^{\zeta}})$ is already  known.

In opposition to the injective tensor products between spaces $C(\omega^{\omega^\xi})$, $\xi<\omega_1$,  the projective tensor products between  spaces  $C(\omega^{\omega^\xi})$, $\xi<\omega_1$, are never isomorphic to another $C(K)$ space. Indeed,  by   \cite[Corollary 2.3]{Unexpected} each of these spaces  contains uniformly complemented  copies of $\ell_2^n$ and it cannot therefore be even a ${\mathcal L}_{\infty}$-space \cite[p. 1598]{Rosenthal}.  

  Our main aim in the  present work is to provide a complete isomorphic classification of the spaces  $C(\omega^{\omega^\xi}) \widehat{\otimes}_\pi C(\omega^{\omega^{\zeta}})$,  $\zeta,\xi < \omega_1$.  More precisely,  our major theorem is the following one.
\begin{theorem} \label{mm} If $K$ and $ L$ are infinite countable compact metric  spaces, then $C(K)\tim_\pi C(L)$ is isomorphic to exactly one of the  spaces $C(\omega^{\omega^\xi})\tim_\pi C(\omega^{\omega^\zeta})$, $\zeta\leqslant \xi<\omega_1$. 
\end{theorem}
Theorem \ref{mm} is an immediate consequence of Corollary \ref{coro} which in turn is an extension of  Bessaga and Pe\l czy\'{n}ski theorem for projective tensor products of  spaces $C(K)$.

\

Of course,  Theorem \ref{mm} suggests several questions about the projective tensor products between spaces $C(K)$,  we will highlight only one of them related to space $C([0,1])$. First of all, observe that fixed    $\xi, \zeta, \eta< \omega_1$,  since   $C(\omega^{\omega^\xi}) \widehat{\otimes}_\pi C(\omega^{\omega^{\zeta}})$ is $c_0$-saturated \cite[Theorem 1.1]{GS} and the spaces  $C(\omega^{\omega^\eta}) \widehat{\otimes}_\pi C([0, 1])$   and $C([0,1]) \widehat{\otimes}_\pi C([0, 1])$    contain complemented copies of $\ell_2$ \cite[Theorem 1.2]{Unexpected}, we see that $C(\omega^{\omega^\xi}) \widehat{\otimes}_\pi C(\omega^{\omega^{\zeta}})$ is   not isomorphic to either $C(\omega^{\omega^\eta}) \widehat{\otimes}_\pi C([0, 1])$ or $C([0,1]) \widehat{\otimes}_\pi C([0, 1]).$ Moreover, $C(\omega^{\omega^\xi}) \widehat{\otimes}_\pi C([0, 1])$ is not  isomorphic to $C([0,1]) \widehat{\otimes}_\pi C([0, 1])$, because the first space has the reciprocal Dunford-Pettis  property  and the last one doesn't     \cite[Theorem 8]{Emma}. So, the following   intriguing question naturally arises.
\begin{problem} \label{PPPPP}  Fix $\xi, \zeta <\omega_1$. Suppose that $C(\omega^{\omega^\xi}) \widehat{\otimes}_\pi C([0, 1])$ is isomorphic to $C(\omega^{\omega^\zeta}) \widehat{\otimes}_\pi C([0, 1])$. Is it true  that $\xi=\zeta$?   
\end{problem}

  Note that after the Theorem \ref{mm} and through  Bessaga and Pe\l czy\'{n}ski's theorem and Milutin's theorem, a  solution positive  to  Problems \ref{PPPPP}  would imply a complete isomorphic classification of  the spaces $C(K) \widehat{\otimes}_\pi C(L)$  for  separable  spaces $C(K)$ and $C(L)$.

\

As far as the proof of Theorem \ref{mm} is concerned, the typical way to show that the  spaces $C(\omega^{\omega^\xi})$, $\xi<\omega_1$, are mutually non-isomorphic, is to use the Szlenk index, which is an isomorphic invariant. 
The Szlenk index can be thought of as the complexity required to obtain $\ell_1^+$ combinations of the branches of weakly null trees in the ball of a given space.   We recall the definition of the Szlenk index in section \ref{szlenk index}.

The third named author showed in \cite{Samuel} that the Szlenk index of $C(\omega^{\omega^\xi})$ is $\omega^{\xi+1}$, which immediately yields that $C(\omega^{\omega^\xi})$ cannot be isomorphic to $C(\omega^{\omega^\zeta})$ when $\zeta, \xi<\omega_1$ are distinct. However, it was recently shown in \cite{CGS1} that for any  ordinals $\zeta\leqslant \xi$, the Szlenk index of $C(\omega^{\omega^\xi})\tim_\pi C(\omega^{\omega^\zeta})$ is equal to the Szlenk index of $C(\omega^{\omega^\xi})$. Therefore the Szlenk index is an insufficiently granular tool to provide an isomorphic classification of the projective tensor products.

To solve the problem we are considering,  we first  introduce a family of properties $\mathbf{G}_{\xi,\zeta}$,  parameterized by two ordinals, $\xi,\zeta<\omega_1$. We define these properties in section \ref{Properties G}.  In our properties $\mathbf{G}_{\xi,\zeta}$, and some associated auxiliary properties $\mathbf{S}_\xi$,  the coarser ordinal $\xi$ will quantify the complexity required  of $\ell_1^+$ combinations of the branches of weakly null trees to guarantee that when we obtain an appropriately complex sequence of $\ell_1^+$ combinations of a branch of a weakly null tree in the ball of our space, that sequence will be $2$-weakly summing with $2$-weakly summing norm depending only on the space.  

Then we will define for each Banach space $X$ a (possibly infinite) quantity $\mathbf{g}_{\xi,\zeta}(X)$, and $X$ will have the $\mathbf{G}_{\xi,\zeta}$ property provided that $\mathbf{g}_{\xi,\zeta}(X)<\infty$.  Moreover, if $X$ has a separable dual, then $\mathbf{g}_{\xi,\zeta}(X)=0$ for some (equivalently, every) $\zeta<\omega_1$) if and only if the Szlenk index of $X$ denoted by $Sz(X)$ is less than or equal to $\omega^\xi$,  and if $\mathbf{g}_{\xi,\zeta}(X)<\infty$ for some $\zeta<\omega_1$, then $Sz(X) \leqslant \omega^{\xi+1}$.  Therefore for a fixed $\xi<\omega_1$ and a Banach space $X$ with separable dual, $X$ can only have one of the properties $\mathbf{G}_{\xi,\zeta}$, $\zeta<\omega_1$ in a non-trivial way (that is, $\mathbf{g}_{\xi,\zeta}(X)\in (0,\infty
)$) if and only if $Sz(X)=\omega^{\xi+1}$.   We establish an upper bound of $\mathbf{g}_{\xi,\zeta}(C(\omega^{\omega^\xi})\tim_\pi C(\omega^{\omega^\zeta}))$ in section \ref{positive} and a lower bound of   $\mathbf{g}_{\xi,\zeta}(C(\omega^{\omega^\xi})\tim_\pi C(\omega^{\omega^\zeta}))$                 in section \ref{sharpness}.
In this way, the $\xi$ can be thought of as quantifying some coarser quantification, and the ordinal $\zeta$ provides some finer gradations within the coarser classes.  As with the Szlenk index, these properties are isomorphic invariants. With these tools in hand, we will be ready to present the proof  of Theorem \ref{mm}  in section \ref{last}.

In the next seven   sections we will recall the main definitions, notations and preliminary results about several concepts that will play a central role in this work, namely: $p$-weakly summable sequences with $p=1$ and $p=2$,  Cantor-Bendixson index and Cantor scheme on compact Hausdorff spaces,   trees,  Schreier families,  Szlenk index and probability measures on  Schreier families.

\section{Tensor products and $p$-weakly summable sequences}\label{projective product}

In the sequel $X,Y,U,V...$ denote Banach spaces over the field $\mathbb{K}$, which is either $\mathbb{R}$ or $\mathbb{C}$.  $B_X,...$ denotes the closed unit ball of $X$ and $S_X, ...$ the unit sphere. By subspace, we shall mean closed, linear subspace, and by operator, we shall mean, bounded, linear operator. If  $x\in X$ and $x^*\in X^*$ we denote $x^*(x)$ by $ \langle x^*,x \rangle.$ 
$\mathcal{L}(X,Y)$ denotes the space of operators from $X$ to $Y$ and 
$\mathcal{B}(X,Y)$ denotes the space of bounded,   bilinear forms on $X\times Y.$ 
The space $\mathcal{B}(X,Y)$ is isometrically isomorphic to the space $\mathcal{L}(X,Y^*).$

Fix $u=\sum_{i=1}^mx_i\otimes y_i\in X\otimes Y.$ The injective norm $\| \cdot\|_\varepsilon$ is defined on the tensor product $X\otimes Y$ by 
$$\| u\|_\varepsilon=\sup\left\{\left\|\sum_{i=1}^mx^*(x_i)y^*(y_i)\right\| : x^*\in B_{X^*}\text{ and }y^*\in B_{Y^*}\right\}.$$
The projective norm $\|\cdot\|_\pi$ is defined on the tensor product $X\otimes Y$ by
$$\| u\|_\pi=\sup\left\{\left\|\sum_{i=1}^mb(x_i,y_i)\right\| : b\in B_{\mathcal{B}(X,Y)}\right\}=\sup\left\{\left\|\sum_{i=1}^mT(x_i)(y_i)\right\| : T\in B_{\mathcal{L}(X,Y^*)}\right\}.$$
We denote by $X\widehat{\otimes}_\varepsilon Y$ the completion of  $X\otimes Y$ with respect to $\|\cdot\|_\varepsilon$ and by $X\widehat{\otimes}_\pi Y$ the completion of  $X\otimes Y$ with respect to $\| \cdot\|_\pi.$

We recall that for a Banach space $E$, $1\leqslant p<\infty$,  a sequence $(e_i)_{i=1}^\infty$ in $E$ is said to be \emph{weakly} $p$-\emph{summing} provided that \[\|(e_i)_{i=1}^\infty\|_p^w := \sup\Bigl\{\Bigl(\sum_{i=1}^\infty |\langle e^*, e_i\rangle|^p\Bigr)^{1/p}: e^*\in B_{E^*}\Bigr\}\] is finite. In this case,  $\|(e_i)_{i=1}^\infty\|_p^w$ denotes the \emph{weakly} $p$-\emph{summing} norm of $(e_i)_{i=1}^\infty$.  We note that if $1/p+1/q=1$, \[\|(e_i)_{i=1}^\infty \|_p^w = \sup\Bigl\{\Bigl\|\sum_{i=1}^\infty a_i e_i\Bigr\|: \|(a_i)_{i=1}^\infty\|_{\ell_q}=1\Bigr\}.\]

The following is a consequence of Grothendieck's theorem. In what follows, $k_G$ denotes Grothendieck's constant. If $X$ and $Y$ are Banach spaces we denote by $\Vert\, .\,\Vert_{1,\pi}^w$ (resp.$\Vert\, .\,\Vert_{2,\pi}^w$) the weakly 1-summing (resp.  weakly 2-summing) norm of a sequence of $X\widehat{\otimes}_\pi Y$ and by $\Vert\, .\,\Vert_{2,\varepsilon}^w$ the weakly 2-summing norm of a sequence of $X\widehat{\otimes}_\varepsilon Y$. 

\begin{proposition}\cite[Lemma $1.1$]{Unexpected} Let $K,L$ be compact, Hausdorff spaces.   If  $(f_n)_{n=1}^\infty $    is a sequence in $B_{C(K)}$  such that $\|(f_n)_{n=1}^\infty\|_2^w \leqslant 1$, then for any  sequence $(g_n)_{n=1}^\infty$ in $B_{C(L)}$, $\|(f_n\otimes g_n)_{n=1}^\infty\|_{2,\pi}^w \leqslant k_G.$   
\label{groth}
\end{proposition}
The following facts regarding the tensorization of weakly $p$-summing norms behave under the formation of projective and injective tensor products. 
\begin{proposition} Let $X,Y$ be Banach spaces and fix a sequence $(x_n)_{n=1}^\infty$ in $X$ and a sequence  $(y_n)_{n=1}^\infty$ in   $Y$. \begin{enumerate}[(i)]\item $\|(x_n\otimes y_n)_{n=1}^\infty\|_{1,\pi}^w \leqslant \|(x_n)_{n=1}^\infty\|_1^w \|(y_n)_{n=1}^\infty\|_1^w$. \item $\|(x_n\otimes y_n)_{n=1}^\infty\|_{2,\ee}^w \leqslant \|(x_n)_{n=1}^\infty\|_2^w \|(y_n)_{n=1}^\infty\|_\infty$. \end{enumerate} Here, $\|\cdot\|_{1,\pi}^w$ denotes the weakly $1$-summing norm of the sequence of tensors considered as members of $X\tim_\pi Y$ and $\|\cdot\|_{2,\ee}^w$ denotes the weakly $2$-summing norm of the sequence of tensors considered as members of $X\tim_\ee Y$. 
\label{haig}
\end{proposition}
\begin{proof}$(i)$ Let $(\ee_i)_{i=1}^\infty$ be a Rademacher sequence.  Fix $m\in\nn$ and scalars $(\delta_i)_{i=1}^m$ such that $|\delta_i|=1$ for all $1\leqslant i\leqslant m$, 
\[\sum_{i=1}^m \delta_i x_i\otimes y_i = \mathbb{E}\left(\Bigl(\sum_{i=1}^m \delta_i \ee_i x_i\Bigr)\otimes \Bigl(\sum_{i=1}^m \ee_i y_i\Bigr)\right),\] where $\mathbb{E}$ is the expectation with respect to $(\ee_i)_{i=1}^\infty$, 
so 
$$ \Bigl\|\sum_{i=1}^m \delta_i x_i \otimes y_i\Bigr\|_\pi  \leqslant \mathbb{E}\left(\Bigl\|\sum_{i=1}^m \delta_i \ee_i x_i\Bigr\|\Bigl\|\sum_{i=1}^m \ee_i y_i\Bigr\|\right) \leqslant \|(x_i)_{i=1}^\infty \|_1^w \|(y_i)_{i=1}^\infty\|_1^w. $$

$(ii)$ Fix $x^*\in B_{X^*}$ and $y^*\in B_{Y^*}$.   Fix $m\in\nn$ and scalars $(a_i)_{i=1}^m$ such that $1=\sum_{i=1}^m |a_i|^2$. Then  \begin{align*} \Bigl|\Bigl\langle x^*\otimes y^*, \sum_{i=1}^m a_ix_i\otimes y_i\Bigr\rangle \Bigr| &  \leqslant \Bigl\|\sum_{i=1}^m a_i \langle y^*, y_i\rangle x_i\Bigr\| \leqslant \|(x_i)_{i=1}^\infty\|_2^w \|(a_i \langle y^*, y_i\rangle)_{i=1}^m\|_{\ell_2} \\ & \leqslant \|(x_i)_{i=1}^\infty\|_2^w \|y^*\|\|(y_i)_{i=1}^\infty\|_\infty \|(a_i)_{i=1}^m\|_{\ell_2} \\ & \leqslant \|(x_i)_{i=1}^\infty\|_2^w \|(y_i)_{i=1}^\infty\|_\infty. \end{align*} 

\end{proof}

\section{On  Cantor-Bendixson index of compact Hausdorff spaces}

For a compact, Hausdorff space $K$, we let $d(K)$ denote the \emph{Cantor-Bendixson derivative of} $K$, which consists of the points of $K$ which are not isolated in $K$. We let $d(\varnothing)=\varnothing$.     Note that either $d(K)$ is empty, or $d(K)$ is also compact, Hausdorff.  We define $d^{(0)}(K)=K,$ $d^{(\eta+1)}(K)=d(d^{(\eta)}(K)),$ and if $\eta$ is a limit ordinal, $d^{(\eta)}(K)=\bigcap_{\nu<\eta}d^{(\nu)}(K).$

   By the Baire category theorem, every non-empty subset of a countable, compact, Hausdorff space $K$ has an isolated point. From this and a cardinality argument, for each countable, compact, Hausdorff space, there exists a countable ordinal $\eta$ such that $d^{(\eta)}(K)\neq \varnothing$ and $d^{(\eta+1)}(K)=\varnothing$. Moreover, $d^{(\eta)}(K)$ is finite.  The ordinal $\eta+1$ is called the \emph{Cantor-Bendixson index} of $K$. We denote the Cantor-Bendixson index of $K$ by $CB(K)$, and since $CB(K)$ is a successor, we let $CB(K)-1$ denote the immediate predecessor of $CB(K)$.    We let \[C_0(K)=\{f\in C(K): f|_{d^{(CB(K)-1)}(K)}\equiv 0\}.\]  It is easy to see that $d^{(\omega^\xi)}([0,\omega^{\omega^\xi}])=\{\omega^\xi\}$ \cite[Theorem 8.6.6]{Semadeni}, so this notation is consistent with the usual notation of $C_0(\omega^{\omega^\xi})$.

 By classical facts about countable, compact, Hausdorff spaces, two countable, compact, Hausdorff spaces $K,L$ are homeomorphic if and only if $CB(K)=CB(L)$ and $K^{CB(K)-1}$ and $L^{CB(K)-1}$ have the same cardinality.  Moreover, any homeomorphism between $K$ and $L$ must map $K^{CB(K)-1}$ to $L^{CB(L)-1}$.   In this case, $C_0(K)$ is isometrically isomorphic to $C_0(L)$. Indeed, if $\phi:L\to K$ is a homeomorphism, then the restriction of the map $\Phi:C(K)\to C(L)$,  $\Phi f=f\circ \phi$, to $C_0(K)$ is an isometric isomorphism between $C_0(K)$ and $C_0(L)$.

\section{On Cantor scheme on compact Hausdorff spaces} \label{seccantor}

Let $K$ be a compact Hausdorff space.   We let $\mathcal{M}(K)=C(K)^*$, the space  of Radon measures on $K$. For each $\varpi\in K$, we denote by $\delta_\varpi$ the Dirac evaluation functional given by $\delta_\varpi(f)=f(\varpi)$.  

Let $\Delta_m=\{(\ee_n)_{n=1}^m: \ee_n\in \{\pm 1\}\}$.  We let $\Delta_0=\{\varnothing\},$  $\Delta_{\leqslant m}=\cup_{n=0}^m \Delta_n $,  and $\Delta_{<m}=\cup_{n=0}^{m-1} \Delta_n.$ 

  A \emph{Cantor scheme on} $K$ will refer to a family of non-empty subsets $D=(A_d)_{d\in \Delta_{\leqslant m}}$ of $K$ such that for each $d\in \Delta_{<m}$, $A_d\supset A_{d\smallfrown (-1)}\cup A_{d\smallfrown (1)}$ and $A_{d\smallfrown(-1)}\cap A_{d\smallfrown (1)}=\varnothing$.   Note that we do not have any requirement that $A_\varnothing$ be equal to $K$.  Given a Cantor scheme $D$ on $K$, a $D$-\emph{selector} is a function $\psi:\Delta_m\to K$ such that for each $d\in \Delta_m$, $\psi(d)\in A_d$.    

Given a Cantor scheme $D=(A_d)_{d\in \Delta_{\leqslant m}}$ on $K$, we say that a sequence $(f_i)_{i=1}^m$ in $ C(K)$ is \emph{compatible with} the Cantor scheme $D$ provided that for any $1\leqslant i \leqslant m$ and any $d\in \Delta_{i-1}$, $f_i|_{A_{d\smallfrown (\ee)}}\equiv \ee$ for $\ee\in \{\pm 1\}$.  Note that if $(f_i)_{i=1}^m$ is compatible with $D$, then for any $D$-selector $\psi$ and $1\leqslant i\leqslant m$, $f_i(\psi(\ee_1, \ldots, \ee_m))=\ee_i$, since $\psi(\ee_1, \ldots, \ee_m)\in A_{(\ee_1, \ldots, \ee_i)}$ and $f_i|_{A_{(\ee_1, \ldots, \ee_i)}}\equiv \ee_i$.  

  Given a Cantor scheme $D=(A_d)_{d\in \Delta_{\leqslant m}}$ on $K$ and a $D$-selector $\psi$, we define the associated Rademacher sequence $(\mu^\psi_i)_{i=1}^m$ by letting 
  \[\mu^\psi_i=\frac{1}{2^m}\sum_{(\varepsilon_1,\ldots,\varepsilon_m)\in\Delta_m}\ee_i\delta_{\psi(\ee_1, \ldots, \ee_m)}.\]

  If $\mu$ denotes the uniform probability measure on $\{\psi(d): d\in \Delta_m\}$, then $\mu_i^\psi\ll \mu$ for each $1\leqslant i\leqslant m$.   Then if $r_i=\frac{d\mu_i^\psi}{d\mu}$ is the Radon-Nikodym derivative of $\mu_i^{\psi}$ with respect to $\mu$, $(\mu_i^\psi)_{i=1}^m$ is isometrically equivalent to $(r_i)_{i=1}^m \subset L_1(\mu)$.  Moreover, $(r_i)_{i=1}^m$ is an orthonormal system in $L_2(\mu)$, which means it is weakly $2$-summing with $2$-weakly summing norm equal to $1$ when considered as a sequence in $L_2(\mu)$. Since $\mu$ is a probability measure, the formal inclusion from $L_2(\mu)$ to $L_1(\mu)$ is norm $1$, which means that $(r_i)_{i=1}^m$ is weakly $2$-summing in $L_1(\mu)$ with $2$-weakly summing norm in $L_1(\mu)$ not exceeding $1$. It is clear that the $2$-weakly summing norm of $(r_i)_{i=1}^m$, when considered as a sequence in $L_1(\mu)$ is exactly $1$, since $\|r_i\|_{L_1(\mu)}=1$ for each $1\leqslant i\leqslant m$.    Therefore $(\mu_i^\psi)_{i=1}^m$ satisfies $\|(\mu_i^\psi)_{i=1}^m\|_2^w \leqslant 1$.     

\begin{rem}\upshape Note that if $D=(A_d)_{d\in \Delta_{\leqslant m}}$ is a Cantor scheme on $K$, $(f_i)_{i=1}^m$ is compatible with $D$, and $\psi$ is a $D$-selector, then $(f_i)_{i=1}^m$ and $(\mu^\psi_i)_{i=1}^m$ are biorthogonal. Indeed, $$\langle \mu^\psi_j, f_i\rangle =\frac{1}{2^m} \sum_{(\ee_1, \ldots, \ee_m)\in \Delta_m}\varepsilon_jf_i(\psi(\ee_1, \ldots, \ee_m)) = \frac{1}{2^m}\sum_{(\ee_1, \ldots, \ee_m)\in \Delta_m} \ee_i \ee_j =\delta_{i,j}.$$

\label{bio}
\end{rem}

We now arrive at the following consequence for witnessing our eventual lower estimates.  The following proposition and its proof deal with both projective and injective tensor norms. Therefore we distinguish between projective and injective tensor norms within the following proposition, and thereafter return to considering only projective tensor norms. 

\begin{proposition} Let $K$ be compact, Hausdorff  space and let $X$ be a Banach space.  Let $D=(A_d)_{d\in \Delta_{\leqslant l}}$ be a Cantor scheme on $K$ and let $(f_i)_{i=1}^l$ be compatible with $D$.    Fix $1\leqslant s_1<\ldots <s_m\leqslant l$ and suppose that  $(x_j)_{j=1}^n\subset X$, $(x^*_j)_{j=1}^m\subset B_{X^*}$, and  $0=r_0<\ldots < r_m=n$  are such that $\langle x^*_i, x_j\rangle = 1$ if $j\in (r_{i-1}, r_i]$, and $\langle x^*_i, x_j\rangle=0$ otherwise.  Then for any sequences $(a_i)_{i=1}^m$, $(b_j)_{j=1}^n$ of positive numbers such that $\sum_{i=1}^m a_i^2=1$ and $\sum_{j=r_{i-1}+1}^{r_i}b_j=1$ for each $1\leqslant i\leqslant m$, \[\Bigl\|\sum_{i=1}^m\sum_{j=1}^n a_ib_j f_{s_i}\otimes x_j\Bigr\|_\pi\geqslant 1.\] 

\label{lower}
\end{proposition}

\begin{proof} Let $\psi$ be any $D$-selector and let $(\mu_i)_{i=1}^l = (\mu_i^\psi)_{i=1}^l$ be the associated Rademacher sequence.  Since  $$\|(\mu_{s_k})_{k=1}^m\|_2^w\leqslant  \|(\mu_i)_{i=1}^l\|_2^w \leqslant 1,$$

by Proposition \ref{haig},  \[\|(\mu_{s_k}\otimes x^*_k)_{k=1}^m\|_{2,\ee}^w \leqslant \|(\mu_{s_k})_{k=1}^m\|_2^w \|(x^*_k)_{k=1}^m\|_\infty \leqslant 1.\]     Here, $\mu_{s_i}\otimes x^*_i$ is treated as a member of $\mathcal{M}(K)\tim_\ee X^*\subset (C(K)\tim_\pi X)^*$.   Therefore \[\Bigl\|\sum_{i=1}^m a_i\mu_{s_i}\otimes x^*_i\Bigr\|_\ee \leqslant 1.\]    

Note that by  Remark \ref{bio}  \begin{align*}\Bigl\|\sum_{i=1}^m\sum_{j=1}^n a_ib_j f_{m_i}\otimes x_j\Bigr\|_\pi & \geqslant \Bigl\langle \sum_{i=1}^m a_i\mu_{s_i}\otimes x^*_i, \sum_{i=1}^m\sum_{j=1}^n a_ib_j f_{s_i}\otimes x_j\Bigr\rangle \\ & = \sum_{i=1}^m\sum_{j=1}^n \sum_{k=1}^m a_ia_kb_j\langle \mu_{s_k}, f_{s_i}\rangle\langle x^*_k,x_j\rangle \\ &  = \sum_{i=1}^ma_i^2\sum_{j=r_{i-1}+1}^{r_i} b_j =1. \end{align*} 

\end{proof}

\section{On  trees}

Given a set $\Lambda$, we let $\Lambda^{<\omega}=\cup_{n=0}^\infty \Lambda^n$ denote the set  of all finite sequences of elements in $\Lambda$.  We note that $\Lambda^{<\omega}$ also includes the empty sequence, which we denote by $\varnothing$.    We let $\Lambda^{\omega}$ the set of all infinite sequences in $\Lambda.$

 For $t\in \Lambda^{<\omega}$, we let $|t|$ denote the length of $t,$ in particular $\vert\varnothing\vert=0.$  For $t\in \Lambda^{<\omega}$ and $0\leqslant i\leqslant |t|$, we let $t|_i$ denote the initial segment of $t$ having length $i$.  
 We note that $s\preceq t$ if $\vert s\vert \leqslant \vert t\vert$ and $s=t\vert_{\vert s\vert}$. $\preceq$ is called the initial segment order.

A non-empty subset $T$ of $\Lambda^
{<\omega}$ is called a \emph{tree on $\Lambda$} (or just a \emph{tree}) provided that if $s\preceq t$ and $t\in T$, then $s\in T$.

We denote by $[\nn]^{<\omega}$ the set of finite subsets of $\nn$. We denote by $[\nn]$ the set of infinite subsets of $\nn$.  Similarly, for $M\in[\nn]$, we let $[M]^{<\omega}$ and $[M]$ denote the sets of finite and infinite subsets of $M$, respectively. By an abuse of notation, we identify subsets of $\nn$ with finite or infinite strictly increasing   sequences. We use set and sequential notation interchangeably.  More precisely, we identify a subset $E$ of $\nn$ with the sequence obtained by listing the members of $E$ in increasing order.  With this identification, $[\nn]^{<\omega}$ can be treated as a subset of $\nn^{<\omega}$.  For a finite subset $E$ of $\nn$, the notation $|E|$ for the cardinality of $E$ is equal to the length of $E$ when treated as a sequence, so the notation is unambiguous.  For $E\subset \nn$ and $1\leqslant n\leqslant |E|$, we let $E(n)$ denote the $n^{th}$ smallest member of $E$. Similarly, for an infinite subset $M$ of $\nn$ and $n\in\nn$, we let $M(n)$ denote the $n^{th}$ smallest member of $M$.   For $E,F\in[\nn]^{<\omega},$ we let  $E<F$ denote the relation that either $E=\varnothing$, $F=\varnothing$, or $\max E<\min F$ and we let $E\smallfrown F$ denote the concatenation of $E$ with $F$.

In the sequel we will use collection indexed by a tree $\mathcal{F}$ on $\nn.$ This collection  will always viewed as a tree. Specifically, the tree
$$\{(x_{(m_1)},x_{(m_1,m_2)},\ldots,x_{(m_1,m_2,\ldots,m_k)}) : k\geq0, (m_1,m_2,\ldots,m_k)\in\mathcal{F}\}
$$ is associated with the collection $\mathbf{u}=(u_E)_{E\in \mathcal{F}\setminus\{\varnothing\}}$  in $U$ \cite[p. 75]{OSZ}.

We say that a family $\mathbf{u}=(u_E)_{E\in [\nn]^{<\infty}}$ in  $U$ is \emph{weakly null} provided that for each $E\in [\nn]^{<\infty}, $ the sequence $(u_{E\smallfrown (k)})_{E<k}$ is weakly null.

For $(m_i)_{i=1}^p, (n_i)_{i=1}^p\in[\nn]^{<\omega}$, we say $(n_i)_{i=1}^p$ is a \emph{spread} of $(m_i)_{i=1}^p$ provided that $m_i\leqslant n_i$ for all $1\leqslant i\leqslant p$.  We say $T\subset [\nn]^{<\omega}$ is \emph{spreading} provided $T$ contains all spreads of its members.   
  We say a set $T\subset [\nn]^{<\omega}$ is \emph{hereditary} provided that if $E\subset F$ and $F\in T$, then $E\in T$. Using our identification of subsets of $\nn$ with sequences in $\nn$, each hereditary subset of $[\nn]^{<\omega}$ can be treated as a tree on $\nn$.

 We endow the power set $2^\nn$ of $\nn$ with the Cantor topology, which is the topology making the identification $2^\nn\ni E\leftrightarrow 1_E\in \{0,1\}^\nn$ a homeomorphism, where $\{0,1\}^\nn$ is endowed with the product of the discrete topology.

 We say a subset $T$ of $2^\nn$ is \emph{compact} if it is compact in the Cantor topology. 
We say a subset $T$ of $[\nn]^{<\omega}$ is \emph{regular} provided it is compact, spreading, and hereditary.

    For a tree $T$ on $\Lambda,$ we let $MAX(T)$ denote the set of all maximal members of $T$ with respect to the initial segment ordering.  For a tree $T$, we define the \emph{derived tree} of $T$, denoted by $T^{(1)}$, by $T^{(1)}=T\setminus MAX(T).$  We define the transfinite derived trees by  $T^{(0)}=T,$ $T^{(\xi+1)}=(T^\xi)^{(1)},$ and if $\xi$ is a limit ordinal, $T^{(\xi)} = \bigcap_{\zeta<\xi}T^{(\zeta)}.$   We say $T$ is \emph{well-founded} provided there exists an ordinal $\xi$ such that $T^{(\xi)}=\varnothing$. In this case, the \emph{rank of} $T$ is defined to be the minimum $\xi$ such that $T^{(\xi)}=\varnothing$. We say $T$ is \emph{ill-founded} if it is not well-founded.

\section{On Schreier families $\mathcal{S}_\xi$ and   $\mathcal{S}_\zeta[\mathcal{S}_\xi].$}\label{defschreier}
	
We next recall the definition of the Schreier families \cite{AA}, which are regular families of subsets of $[\nn]^{<\omega}$.  We let $\mathcal{S}_0=\{\varnothing\}\cup \{(n):n\in\nn\}.$  Assuming that $\mathcal{S}_\xi$ has been defined for some $\xi<\omega_1$, we let \[\mathcal{S}_{\xi+1}=\{\varnothing\}\cup \Bigl\{\bigcup_{n=1}^m E_n: E_1<\ldots <E_m, \varnothing\neq E_n\in \mathcal{S}_\xi, m\leqslant E_1\Bigr\},\] and if $\xi$ is a limit ordinal, fix a sequence $\xi_n\uparrow\xi$ and let \[\mathcal{S}_\xi =\{\varnothing\} \cup \{E: \exists n\leqslant E \text{ and } E\in \mathcal{S}_{\xi_n+1}\}.\]  We refer to this sequence as the \emph{characteristic sequence in} $\xi$.  We note that the sequence $(\xi_n)_{n=1}^\infty$ can be chosen to satisfy the property that $\mathcal{S}_{\xi_n+1}\subset \mathcal{S}_{\xi_{n+1}}$ for all $n\in\nn$, which was shown in \cite{Concerning}.   We assume these sequences are chosen in this way, from which it follows that \[\mathcal{S}_\xi =\{\varnothing\} \cup \Bigl\{E: \varnothing\neq E\in \mathcal{S}_{\xi_{\min E}+1}\}.\]

We also recall a common method for constructing new families from old. If $\mathcal{F}, \mathcal{G}\subset [\nn]^{<\omega}$,  \[\mathcal{F}[\mathcal{G}]=\{\varnothing\}\cup \Bigl\{\bigcup_{n=1}^m E_n: E_1<\ldots <E_m, \varnothing\neq E_n\in \mathcal{G}, (\min E_n)_{n=1}^m\in  \mathcal{F}\Bigr\}.\]

Our main interest with this construction will be $\mathcal{S}_\zeta[\mathcal{S}_\xi]$ for $\zeta, \xi<\omega_1$, in which case  \[\mathcal{S}_\zeta[\mathcal{S}_\xi]=\{\varnothing\}\cup\Bigl\{\bigcup_{n=1}^m E_n: E_1<\ldots <E_m, \varnothing\neq E_n\in \mathcal{S}_\xi, (\min E_n)_{n=1}^m\in \mathcal{S}_\zeta\Bigr\}.\]   

 We note that $\mathcal{S}_{\xi+1}=\mathcal{S}_1[\mathcal{S}_\xi]$ for all $\xi<\omega_1.$ For $0\leq\xi,\zeta<\omega_1,$ $S_\xi$ and $\mathcal{S}_\zeta[\mathcal{S}_\xi]$
are trees on $\nn.$

 We let $MAX(\mathcal{S}_\xi)$   denote the set of all maximal elements of $\mathcal{S}_\xi$  with respect to the initial segment order. We note that since $\mathcal{S}_\xi$ is spreading and hereditary, $MAX(\mathcal{S}_\xi)$ is also the set of maximal elements of $\mathcal{S}_\xi$ with respect to inclusion. Similarly, we denote by $MAX(\mathcal{S}_\zeta[\mathcal{S}_\xi])$ the set of all maximal elements of $\mathcal{S}_\zeta[\mathcal{S}_\xi]$ with respect to the initial segment ordering (equivalently, with respect to inclusion).

We isolate the following important properties of the Schreier families. Each of these properties can be found in \cite{Concerning}.

\begin{lemma} Fix $\xi,\zeta<\omega_1$. \begin{enumerate}[(i)]\item  $\mathcal{S}_\xi$ is a regular family with a  Cantor-Bendixson index $ I_{CB}(\mathcal{S}_\xi)= \omega^\xi+1$. More generally,  $\mathcal{S}_\zeta[\mathcal{S}_\xi]$ is a regular family with a Cantor-Bendixson index $I_{CB}(\mathcal{S}_\zeta[\mathcal{S}_\xi])=\omega^{\xi+\zeta}+1.$ 
\item If $\xi\leqslant \zeta$,  there exists $k=k(\xi,\zeta)\in\nn$ such that if $E\in \mathcal{S}_\xi$ and $k\leqslant E$, then $E\in \mathcal{S}_\zeta$. 

\item For $E\in \mathcal{S}_\xi$, either $E\in MAX(\mathcal{S}_\xi)$ or for all $n\in\nn$ and  $E<n,$ $E\smallfrown(n)\in \mathcal{S}_\xi.$ 

\item For $E,F\in \mathcal{S}_\xi$, if $E\in MAX(\mathcal{S}_\xi)$ and $E \preceq F $, then $E=F$.  \end{enumerate} 
\label{schreier}
\end{lemma}
\begin{rem} \label{Cantor}  
 Note that $C_0(\omega^{\omega^\xi})$ is isometrically isomorphic to $C_0(\mathcal{S}_\xi)=\{f\in C(\mathcal{S}_\xi): f(\varnothing)=0\}.$ We will use this in the  proof of  Theorem \ref{negative one}. 
\end{rem}

   Next  we state that  every infinite subset $M$ of $\nn$ admits a unique decomposition by successive members of a Schreier family.

\begin{lemma}\label{decomposition}Fix $\xi,\zeta<\omega_1.$ For each infinite subset $M$ of $\nn$, there exist a unique sequence $(F_k)_{k=1}^\infty$ of successive members in $MAX(\mathcal{S}_\xi) $ $(F_1<F_2<\ldots)$ and a unique sequence $(E_k)_{k=1}^\infty$ of successive members in $MAX(\mathcal{S}_\zeta[\mathcal{S}_\xi])$ $(E_1<E_2<\ldots)$  such that $$M=\bigcup_{k=1}^\infty E_k=\bigcup_{k=1}^\infty F_k.$$
\end{lemma}

\begin{proof} In order to see that each $M$ admits such a $\mathcal{S}_\zeta[\mathcal{S}_\xi]$ decomposition, it is sufficient to note that there exists a (necessarily unique) initial segment of $M$ which is a maximal member of $\mathcal{S}_\zeta[\mathcal{S}_\xi]$, which we denote by   $M|_{\xi,\zeta,1}$.   We then note if $$E_1=M|_{\xi,\zeta,1}, \ \  E_2= (M\setminus E_1)|_{\xi,\zeta,1},  \  \ldots,   \ \ E_{n+1}=(M\setminus \cup_{k=1}^n E_k)|_{\xi,\zeta,1},$$ then $M=\cup_{k=1}^\infty E_k$ is the unique decomposition of $M$ into successive, maximal member of $\mathcal{S}_\zeta[\mathcal{S}_\xi]$.

Let us explain how to deduce the existence of $M|_{\xi,\zeta,1}$.    We let \[A=\{m\in\nn: (M(1), \ldots, M(m))\in \mathcal{S}_\zeta[\mathcal{S}_\xi]\}.\]   Since $\mathcal{S}_\zeta[\mathcal{S}_\xi]$ contains all singletons, $1\in A\neq \varnothing$.    Since $\mathcal{S}_\zeta[\mathcal{S}_\xi]$ is hereditary, $A$ is an interval. Since $\mathcal{S}_\zeta[\mathcal{S}_\xi]$ is compact,   the set $A$ is finite.   Indeed, if $A$ were infinite, then it would be the case that $A=\nn$, which means $$F_m:=(M(1), M(2), \ldots, M(m))\in \mathcal{S}_\zeta[\mathcal{S}_\xi],$$ for all $m\in\nn$.    
But  the sequence  $(F_m)_{m=1}^\infty$ converges to $M$ and so $M\in\mathcal{S}_\zeta[\mathcal{S}_\xi]$,  contradicting the fact that $\mathcal{S}_\zeta[\mathcal{S}_\xi]$ consists only of finite sets.  Therefore the set $A$ is finite and non-empty so    it has a maximum, call it $m$.

By definition of $A$ and maximality of $m$, $$E_1:=(M(1), \ldots, M(m))\in \mathcal{S}_\zeta[\mathcal{S}_\xi] \ \ \hbox{and}  \ \ E_1\smallfrown (M(m+1)) =(M(1), \ldots, M(m+1))\notin \mathcal{S}_\zeta[\mathcal{S}_\xi].$$  
It remains to show that $E_1$ is a maximal member of $\mathcal{S}_\zeta[\mathcal{S}_\xi]$.  Here we recall that maximality with respect to inclusion and initial segment ordering are equivalent. If it were not maximal, then it would be a proper initial segment of some other member of $ \mathcal{S}_\zeta[\mathcal{S}_\xi]$, and since $\mathcal{S}_\zeta[\mathcal{S}_\xi]$ is hereditary, $E_1\smallfrown (p)\in \mathcal{S}_\zeta[\mathcal{S}_\xi]$ for some $E_1<p$.   By by Lemma \ref{schreier}$(iii)$, this would imply that $$E_1\smallfrown(M(m+1))=(M(1), \ldots, M(m+1))\in \mathcal{S}_\zeta[\mathcal{S}_\xi],$$ contradicting the maximality of $m$.   This shows that $E_1$ is maximal in $\mathcal{S}_{\zeta}[\mathcal{S}_\xi]$.

Note also that non-empty elements of $\mathcal{S}_\zeta[\mathcal{S}_\xi]$  admit a canonical representation. For $E\in [\nn]^{<\infty}$  we can write $E=\cup_{n=1}^m F_n$ with $F_n\in \mathcal{S}_\xi$  and $F_n\in MAX(\mathcal{S}_\xi)$ for $1\leqslant n<m$. If $E\in \mathcal{S}_\zeta[\mathcal{S}_\xi]$  then $(\min E_n)_{n=1}^m\in \mathcal{S}_\zeta$.  The existence of this representation is established similarly to the previous paragraphs. If $F\in[\nn]^{<\omega}$, we let $E_1$ be the maximal initial segment of $F$ which is in $\mathcal{S}_\zeta[\mathcal{S}_\xi]$. Then either $E_1=F$, in which case $F\in \mathcal{S}_\zeta[\mathcal{S}_\xi]$ may be non-maximal in $\mathcal{S}_\zeta[\mathcal{S}_\xi]$, or $E_1\prec F$ and $E_1$ is maximal in $\mathcal{S}_\zeta[\mathcal{S}_\xi]$, with the same proof as in the last paragraphs.  In the latter case, we let $E_2$ denote the maximal initial segment of $F\setminus E_1$ which is a member of $\mathcal{S}_\zeta[\mathcal{S}_\xi]$, etc.   We arrive at the advertised decomposition where each set $E_k$, except possibly the last, is maximal in $\mathcal{S}_\zeta[\mathcal{S}_\xi]$. 
\end{proof}

\begin{rem} \label{notation} The sequence $(F_k)_{k=1}^\infty$ given by Lemma \ref{decomposition}  is called the  $\mathcal{S}_\xi$-{\textit{decomposition of}} $M$ and the sequence $(E_k)_{k=1}^\infty$ is called the  $\mathcal{S}_\zeta[\mathcal{S}_\xi]$-{\textit{decomposition of}} $M.$ We let 
$$M\vert_{\xi,0}=\varnothing, \ \  M\vert_{\xi,n}=\bigcup_{k=1}^nF_k, \ \   M\vert_{\xi,\zeta,0}=\varnothing \ \ \hbox{and} \ \  M\vert_{\xi,\zeta,n}=\bigcup_{k=1}^nE_k.$$
We notice that the case of the $\mathcal{S}_\xi$-decomposition of an infinite subset  $M\subset \nn$ is proven in \cite[Corollary 3.2]{Gasp}. 
\end{rem}

\begin{xrem} It follows easily from Lemma \ref{schreier} that for $E\in [\nn]^{<\omega},$ we have $E\in MAX(\mathcal{S}_\zeta[\mathcal{S}_\xi])$ if, and only if, there exists a unique finite sequence $(F_r)_{r=1}^s$ of $MAX(\mathcal{S}_\xi)$ which satisfies $E=\bigcup_{1\leq r\leq s}F_r,$
$F_1<...<F_s$ and $(\min F_r)_{r=1}^s\in MAX(\mathcal{S}_\zeta).$ 
 
Observe that the convolution sets $\mathcal{S}_\zeta[\mathcal{S}_\xi]$ satisfy the following identity for all countable $\zeta, \xi$:  \begin{align*}\mathcal{S}_{\zeta+1}[\mathcal{S}_\xi] & = (\mathcal{S}_1[\mathcal{S}_\zeta])[\mathcal{S}_\xi]= \mathcal{S}_1[\mathcal{S}_\zeta[\mathcal{S}_\xi]] \\ & = \{\varnothing\}\cup \Bigl\{\bigcup_{n=1}^m E_n: E_1<\ldots <E_m, \varnothing\neq E_n\in \mathcal{S}_\zeta[\mathcal{S}_\xi], m\leqslant \min E_1\Bigr\}.\end{align*}

Finally,  note that members of $\mathcal{S}_1[\mathcal{S}_\zeta[\mathcal{S}_\xi]]$ admit a canonical decomposition. 
\end{xrem}

\section{On   Szlenk index $Sz(X)$}\label{szlenk index}

The Szlenk index is the fragmentation index of the dual ball $B_{X^*}$
of a  Banach space $X$ with respect to the the weak*-topology and the metric induced by the norm. The concept of fragmentation has been introduced in \cite{JR}.

For  $\varepsilon>0$ and  $K\subset X^*$ we let $s_\varepsilon(K)$ denote the set of $x^*\in K$ such that for any weak*-neighborhood $V$ of $x^*,$ $$\text{diam}(V\cap K)=\sup_{u^*,v^*\in K}\Vert u^*-v^*\Vert>\varepsilon.$$

We define $s_\varepsilon^{(0)}(K)=K$, 
$s_\varepsilon^{(\alpha+1)}(K)=s_\varepsilon(s_\varepsilon^{(\alpha)}(K))$, and if $\alpha$ is a limit ordinal
$s_\varepsilon^{(\alpha)}(K)=\bigcap_{\beta<\alpha}s_\varepsilon^{(\beta)}(K).$

We let $Sz(K,\varepsilon)$ denote the minimum $\alpha$ such that $s_\varepsilon^{(\alpha)}(K)=\varnothing,$ if such a $\alpha$ exists, and otherwise we write  $Sz(K,\varepsilon)=\omega_1.$ We let $Sz(K)=\sup_{\varepsilon>0}Sz(K,\varepsilon)$ if $Sz(K,\varepsilon)$ is an ordinal for all $\varepsilon>0,$ and we write $Sz(K)=\omega_1$ if $Sz(K,\varepsilon) =\omega_1$ for some $\varepsilon>0.$ We denote by $Sz(X,\varepsilon)= Sz(B_{X^*},\varepsilon)$ the \emph{$\varepsilon$-Szlenk index of $X$} 
 and by $Sz(X)=Sz(B_{X ^*})$ the
\emph{Szlenk index of $X$}.

The Szlenk index can be characterized as the  complexity which is required to obtain arbitrarily small convex (or $\ell_1^+$) combinations of the branches of weakly null trees \cite[Definition 2.7.]{AJO} in the unit ball of an Asplund space .

Precisely, for $0<\varepsilon<1,$ let
$$\mathcal{F}_\varepsilon=\{(x_i)_{i=1}^n:n\in\nn, x_1,...,x_n\in S_X \text{ and for each } x\in \text{co}\{x_1,...,x_n\},\ \Vert x\Vert\geqslant\varepsilon\}
$$

The weak derivative $\mathcal{F}_\varepsilon^{(1)}$ of $\mathcal{F}_\varepsilon$ consists of all the elements $(x_i)_{i=1}^n\in \mathcal{F}_\varepsilon$ for which there exists a weakly null sequence $(y_k)_{ k\geqslant1}$ of elements of $S_X$ such that $(x_1,...,x_n,y_k)\in  
\mathcal{F}_\varepsilon$ for each $k\geqslant1.$

We define by induction higher derivatives  $\mathcal{F}_\varepsilon^{(\alpha)}$  for $\alpha<\omega_1$,  by \  $\mathcal{F}_\varepsilon^{(0)}=\mathcal{F}_\varepsilon$, 
$\mathcal{F}_\varepsilon^{(\alpha)}
=(\mathcal{F}_\varepsilon^{(\beta)})^{(1)}\text{ if }\alpha=\beta+1$,  and 
$\mathcal{F}_\varepsilon^{(\alpha)}
=\bigcap_{\beta<\alpha}\mathcal{F}_\varepsilon^{(\beta)}$ if $\alpha$ is a limit ordinal.

We now define the $\ell_1^+$-$index$ of $X.$

$$
I_w(\mathcal{F}_\varepsilon)=\begin{cases}\min\{\alpha<\omega_1:\mathcal{F}_\varepsilon^{(\alpha)}=\emptyset\}\text{ if there exists $\alpha<\omega_1$ with $\mathcal{F}_\varepsilon^{(\alpha)}=\emptyset$}\\
\omega_1 \text{ othewise}
\end{cases}
$$
and $I_w(X)=\sup_{\varepsilon>0}I_w(\mathcal{F}_\varepsilon).$
The ordinal $I_w(X)$ is called the $\ell_1^+$-$index$ of $X.$

We recall the  main properties of the index that we will use later.

\begin{theorem}\label{Sz=Iw} Let $X$ be an infinite dimensional separable Banach space  not containing $\ell_1.$ Then,

i) $I_w(X)=Sz(X), $ \cite[Theorem 4.2.]{AJO}.

ii) If $X$ has a separable dual, then $I_w(X)=Sz(X)=\omega^\alpha$ for some $\alpha<\omega_1$ \cite[Theorem 3.2.]{AJO}.

iii) $Sz(X)< \omega_1$ if, and only if $X^*$ is separable \cite[Theorem 1]{Lancien}.
\end{theorem}

The following proposition is an immediate consequence of \cite[Proposition 5]{OSZ}. 

\begin{proposition} \label{OdellSZ}  Suppose that $I_w(X)> \omega^{\alpha+1}$  for a Banach space $X$ and $\alpha<\omega_1$. Then there exist $\varepsilon>0$ and a collection  $\mathbf{u}=(u_E)_{ E\in \mathcal{S}_{\alpha+1}\setminus\{\varnothing\}}$ in $B_X$ such that:
\begin{enumerate}[(a)]\item for any $E\in \mathcal{S}_{\alpha+1}\setminus MAX(\mathcal{S}_{\alpha+1})$, the sequence $(u_{E\smallfrown(n)})_{E<n}$ is  weakly null,  
\item for each $E\in \mathcal{S}_{\alpha+1},$  $\varepsilon\leqslant\inf\{\Vert x\Vert : x\in \mathrm{co}(u_F: \varnothing \prec F\preceq E)
\}$
    \end{enumerate}

\end{proposition}

Finally, we will also need the following result on the Szlenk index of the projective tensor product of spaces $C(K)$ spaces obtained in \cite[Theorem 1.1]{CGS1}.

\begin{theorem} \label{JFA} Let $K$, $L$ be compact Hausdorff topological spaces. Then
$$Sz(C(K) \tim_\pi C(L))= \hbox{max} \{Sz(C(K), Sz(C(L)\}.$$
\end{theorem}

\section{On families of probability measures on $\mathcal{S}_\xi$} \label{sufficient}

 In this section, we define the notions of probability block, $\xi$-sufficient, and $\xi$-regulatory probability blocks. The key property here is to make precise the idea that if a function on a tree has a sufficiently large set of averages with respect to special convex combinations of specified complexity, then it has a sufficiently large set on which it is pointwise large. 

Let us begin with a simple analogy. Let $S$ be an arbitrary infinite set and let $f:S\to [0,1]$ be an arbitrary function. Then for any $\ee>0$, $|\{x\in S:f(x)>\varepsilon\}|=m_\ee<\infty$, it must follow that
$$\inf \left\{\frac{1}{|T|}\sum_{x\in T}f(x):T\subset S, 0<|T|<\infty\right\}=0. 
$$

By contraposition, if \[\inf\left \{\frac{1}{|T|}\sum_{x\in T}f(x):T\subset S, 0<|T|<\infty\right\}=\ee>0,\] then for any $0<\ee_1<\ee$ and $m\in\nn$, there must exist $T\subset S$ with $|T|=m$ and $f(x)\geqslant \ee_1$ for all $x\in T$. 

As another example, let $T$ be a well-founded tree such that $\sup_{t\in T}|t|=\infty$ and let $f:\Pi\to [0,1]$ be a function, where $\Pi=\{(s,t)\in T\times MAX(T):s\preceq t\}$.   Suppose that \[\frac{1}{|t|}\sum_{s\preceq t} f(s,t)=\ee>0.\]   Then for any $0<\ee_1<\ee$ and $m\in\nn$, there exist $t\in T$ and a subset $R$ of $\{s\in T:s\preceq t\}$ such that $|R|=m$ and $f(s,t)\geqslant \ee_1$ for all $s\in R$.  From this it follows that for any tree $S$ with $\text{rank}(S)=\omega$ and $0<\ee_1<\ee$, there exists a monotone map $d:S\to T$ and a map $e:MAX(S)\to MAX(T)$ such that \begin{enumerate}[(i)]\item $d(s)\preceq e(s)$ for any $s\in MAX(S)$, and \item $f(d(s),e(t))\geqslant \ee_1$ for all $s\in S$ and $t\in MAX(S)$ with $s\preceq t$. \end{enumerate}

Next, let us consider the complications of trying to extend this result to trees of rank $\omega^2$.   Let $X_1$ denote Schreier's space, which is the completion of $c_{00}$ with respect the norm of which is given by \[ \Bigl\|\sum_{n=1}^\infty a_ne_n\Bigr\|=\sup_{E\in \mathcal{S}_1} \sum_{n\in E}|a_n|=\sup_{E\in\mathcal{S}_1} \sum_{n\in E} \Bigl|\Bigl\langle e^*_n,\sum_{j=1}^\infty a_je_j\Bigr\rangle\Bigr|.\] Let $T=\mathcal{S}_2\setminus\{\varnothing\}$. For each $F\in MAX(T)$, let $x^*_F\in S_{X^*_1}$ be a norming functional for $\sum_{n\in F}e_n$ and define $f(E,F)=\text{Re\ }x^*_F(e_{\max E})$ for $E\in T$ with $E\preceq F$. Then for any $F\in MAX(T)$, \[\frac{1}{|F|}\sum_{E\preceq F}f(E,F) =\Bigl\|\frac{1}{|F|}\sum_{n\in F}e_n\|\geqslant \frac{1}{|F|}\sum_{F\cap (|F|/2,\infty)} 1 \geqslant 1/2.\] However, suppose that $S$ is a tree with rank $\omega^2$, $d:S\to T$ is a monotone map, and $e:MAX(S)\to MAX(T)$ is a map such that $d(s)\preceq e(s)$ for any $s\in MAX(S)$.   Then for any $n\in\nn$, we would be able to find $F\in MAX(S)$ and a subset $G$ of  $\{\max d(E):E\preceq F\}$ such that $G=\cup_{i=1}^n G_i$ for some $n\leqslant G_1<\ldots <G_i$, $G_1\neq \varnothing$, $|G_{i+1}|>2\max G_i$ for all $1\leqslant i<n$.   For $E\preceq F$, if $\max d(E)\in G$, let $w_E=\frac{1}{n|G_i|}$, where $\max d(E)\in G_i$, and otherwise let $w_E=\varnothing$.   Then 

\begin{align*} \sum_{E\preceq F} w_E f(d(E),e(F)) & = \sum_{E\preceq F} w_E\text{Re\ }x^*_{e(F)}(e_{\max d(E)})  = \text{Re }\Bigl\langle x^*_{e(F)},\frac{1}{n}\sum_{i=1}^n \frac{1}{|G_i|}\sum_{j\in G_i}e_j\Bigr\rangle \\ & \leqslant \Bigl\|\frac{1}{n}\sum_{i=1}^n \frac{1}{|G_i|}\sum_{j\in G_i}e_j\Bigr\|\leqslant 2/n.\end{align*} For the last inequality, we have used the fact that for any $E\in\mathcal{S}_1$, either $E\cap G_i=\varnothing$ for all $1\leqslant i\leqslant n$, in which case \[\sum_{k\in E} \Bigl|\Bigl\langle e^*_k,\frac{1}{n}\sum_{i=1}^n \frac{1}{|G_i|}\sum_{j\in G_i}e_j\Bigr\rangle\Bigr|=0,\] or there exists a minimal $i_0$ with $1\leqslant i_0\leqslant n$ such that $E\cap G_{i_0}\neq\varnothing$. In this case, since $E\in\mathcal{S}_1$, \[|E|\leqslant \min E\leqslant \max G_{i_0} < |G_{i_0+1}|/2 < |G_{i_0+2}|/2^2<\ldots < |G_n|/2^{n-i_0}.\] Then \begin{align*} \sum_{k\in E} \Bigl|\Bigl\langle e_k^*, \frac{1}{n}\sum_{i=1}^n \frac{1}{|G_i|}\sum_{j\in G_i}e_j\Bigr\rangle \Bigr| & \leqslant \frac{1}{n}\sum_{i=i_0}^n \frac{|E\cap G_i|}{|G_i|} \leqslant \frac{1}{n}\sum_{i=0}^{n-i_0}\frac{1}{2^i}\leqslant 2/n.\end{align*}

Since $n$ was arbitrary, it is not possible that there exists any $\ee_1>0$ such that $f(d(E), e(F))\geqslant \ee_1$ for all $E,F\in S$ with $E\preceq F\in MAX(S)$. In fact, this phenomenon is the purpose of Schreier's construction of his example of a space failing the weak Banach-Saks property. It was Argyros, Mercourakis, and Tsarpalias who introduced the repeated averages hierarchy to provide a transfinite quantification of the complexity of norm null convex blockings of a given weakly null sequence, as guaranteed by Mazur's lemma. The properties of the repeated averages hierarchy are essentially those we present in this section, although we follow the terminology from \cite{CN}. 

We let $\mathscr{P}$ denote the set of probability measures on $\nn$ such that for each $\mathbb{P}\in \mathscr{P}$, $\mathbb{P}(\{i\})$ is defined. For $\mathbb{P}\in \mathscr{P}$, we abuse notation and let $\mathbb{P}(i)=\mathbb{P}(\{i\})$ for $i\in\nn$.  We let $\supp(\mathbb{P})=\{i\in\nn:\mathbb{P}(i)>0\}$, and refer to this set as the \emph{support} of $\mathbb{P}$.  We say $\mathcal{P}\subset [\nn]^{<\omega}$ is \emph{nice} provided it is compact, hereditary, spreading, contains all singletons, and has the property that for each $E\in \mathcal{P}$, then there exists some $m\in \nn\cap (\max E,\infty)$ such that $E\smallfrown (m)\in \mathcal{P}$ if and only if for every $m\in \nn\cap (\max E,\infty)$, $E\smallfrown (m)\in \mathcal{P}$.  It is straightforward to see that the Schreier families are all nice. We note that for any $M\in[\nn]$, there exists a unique non-empty, finite initial segment of $M$ which is a member of $MAX(\mathcal{P})$. The proof is the same as in the case of the Schreier families. In this section and this section only, we use the notation $M|_\mathcal{P}$ to denote this initial segment. We will not need this notation in subsequent sections, since our only case of interest will be the Schreier families and the repeated averages hierarchy.  

A \emph{probability block} is a pair $(\mathfrak{P}, \mathcal{P})$, where \[\mathfrak{P}=\{\mathbb{P}_{M,n}:M\in[\nn], n\in\nn\}\subset \mathscr{P}\] and $\mathcal{P}$ is a nice family such that 

\begin{enumerate}[(i)]\item for any $M\in[\nn]$ and $r\in \nn$, if $N=M\setminus \cup_{i=1}^{r-1}\supp(\mathbb{P}_{M,i})$, then $\mathbb{P}_{N,1}=\mathbb{P}_{M,r}$, and 

\item for each $M\in[\nn]$, $\supp(\mathbb{P}_{M,1})=M|_{\mathcal{P},1}$

\end{enumerate}

For a countable ordinal $\xi$, we say a probability block $(\mathfrak{P}, \mathcal{P})$ is $\xi$-\emph{sufficient} if for any $\ee>0$, any regular family $\mathcal{G}$ with $CB(\mathcal{G})\leqslant \omega^\xi$, and $L\in[\nn]$, there exists $M\in[L]$ such that \[\sup \{\mathbb{P}_{N,1}(E):N\in [M], E\in \mathcal{G}\}\leqslant \ee.\] 

Next, for a probability block $(\mathfrak{P}, \mathcal{P})$, a function $f:\nn\times MAX(\mathcal{P})\to \rr$, and $\ee\in \rr$,  we let \[\mathfrak{E}(f,\mathfrak{P}, \mathcal{P}, \ee)=\Bigl\{M\in[\nn]:\sum_{i=1}^\infty \mathbb{P}_{M,1}(i) f(i,\supp(\mathbb{P}_{M,1}))\geqslant \ee\Bigr\}.\] For a probability block $(\mathfrak{P}, \mathcal{P})$, a function $f:\nn\times MAX(\mathcal{P})\to \rr$, $M\in[\nn]$, and $\ee\in \rr$, we let \[\mathfrak{G}(f,\mathcal{P},M,\ee) = \{E\in[\nn]^{<\omega}:(\exists F\in MAX(\mathcal{P})\cap [M]^{<\omega}:(\forall i \in E)(f(i,F)\geqslant \ee)\}.\]

We say $\mathfrak{G}(f,\mathcal{P},M,\ee)$ is $\xi$-\emph{full} provided that there exists $L\in[\nn]$ such that $\mathcal{S}_\xi(L)\subset \mathfrak{G}(f,\mathcal{P}, M,\ee)$.  We say that the probability block $(\mathfrak{P}, \mathcal{P})$ is $\xi$-\emph{regulatory} provided that for any bounded function $f:\nn\times MAX(\mathcal{P})\to \rr$, any $\delta,\ee\in \rr$ with $\delta<\ee$, if there exists $L\in[\nn]$ such that $[L]\subset \mathfrak{E}(f,\mathfrak{P}, \mathcal{P}, \ee)$, then $\mathfrak{G}(f,\mathcal{P}, L,\delta)$ is $\xi$-full.  We have the following. 

\begin{lemma}\cite[Lemma 4.5]{CN} Let $\xi$ be a countable ordinal and let $(\mathfrak{P}, \mathcal{P})$ be a probability block. If $(\mathfrak{P}, \mathcal{P})$ is $\xi$-sufficient, then it is $\xi$-regulatory.   
\end{lemma} \label{resultado1}

\section{Some convenient weights on Schreier families}\label{schreier families}

In this section  we introduce the weight $p^\xi$ and  the weight $q^{\xi,\zeta}.$  The weight $p^\xi$ is defined on  $\mathcal{S}_\xi\setminus\{\varnothing\}$ and the weight $q^{\xi,\zeta}$ is defined on  $\mathcal{S}_\zeta[\mathcal{S}_\xi]\setminus
\{\varnothing\}.$ Here we will only highlight the permanence and convexity properties  of these weights (Proposition \ref{perm2}). First we will establish the Lemma \ref{perm1}.

\subsection{The weight $p^\xi$ on $\mathcal{S}_\xi$}
For each ordinal $\xi<\omega_1$ we define  the weight $p^\xi$ on $\mathcal{S}_\xi\setminus\{\varnothing\}$ and after we extend  $p^\xi$ to $[\nn]^{<\omega}. $ We proceed by induction on $\xi.$

   We let $p^0_E=1$ for all $  E\in\mathcal{S}_0\setminus\{\varnothing\}$. Next, assume that $p^\xi_E$ has been defined for some $\xi<\omega_1$ and every $  E\in\mathcal{S}_\xi\setminus\{\varnothing\} $. For $ E\in\mathcal{S}_{\xi+1}\setminus\{\varnothing\}$, we can uniquely  write  $E=\cup_{n=1}^l F_n$, $F_1<\ldots<F_l$, where $F_n\in MAX(\mathcal{S}_{\xi})$ for each $1\leqslant n<l$ and $ F_l\in \mathcal{S}_{\xi}.$   
We define $$p^{\xi+1}_E= p^\xi_{F_l}/\min E.$$

 Assume $\xi<\omega_1$ is a limit ordinal and $p^\nu$ has been defined on $\ \mathcal{S}_\nu\setminus\{\varnothing\}$ for each $\nu<\xi.$  
    Let $\xi_n\uparrow \xi$ be the characteristic sequence of $\xi.$  For $E\in \mathcal{S}_\xi\setminus\{\varnothing\}$ we define $$p^\xi_E=p^{\xi_{\min E}+1}_{E}.$$

In order to define $p^\xi$  on $[\nn]^{<\omega}\setminus\{\varnothing\}$
 we use of Lemma~\ref{decomposition}. Let $(E_i)_{i=1}^n$ be the 
$\mathcal{S}_\xi$-decomposition of  $E\in[\nn]^{<\omega}\setminus\{\varnothing\}
,$   
then we let $$p^\xi_E = p^\xi_{E_n}.$$

\begin{lemma}
   For each  $E\in MAX(\mathcal{S}_\xi)$  we have  $1=\sum_{\varnothing \prec F\preceq  E} p^\xi_F.$  
\label{perm1}
\end{lemma}
\begin{proof}  It is sufficient to prove the result in the case $n=1$.  We prove this result by induction on $\xi$.  For the case $\xi=0$, $E $ is a singleton and $p^0_{E }=1$, so the sum is of a single term equal to $1$. For the successor case, 
let $E\in MAX(\mathcal{S}_{\xi+1})$ and $(F_k)_{k=1}^l$ be the  $ \mathcal{S}_{\xi}$-decomposition of $E.$ Then $l=\min E=\min F_1.$ 
Then by the inductive hypothesis,

$$\sum_{F\preceq E} p^{\xi+1}_F=\sum_{k=1}^l\sum_{E\vert_{\xi,k-1}\prec F\preceq E\vert_{\xi,k}}p^{\xi+1}_F=\sum_{k=1}^l\sum_{\varnothing\prec G\preceq F_k}p^{\xi+1}_{E\vert_{\xi,k-1}\cup G}=\sum_{k=1}^l\frac{1}{\min E}\sum_{\varnothing\prec G\preceq F_k}p^\xi_G=1.
$$

 If $\xi$ is a limit ordinal, then since $E\in MAX(\mathcal{S}_\xi)$, $E\in MAX(\mathcal{S}_{\xi_{\min E}+1})$, where $\xi_m\uparrow \xi$ is the characteristic sequence of $\xi$.    Then \[\sum_{F\preceq E}p^\xi_F = \sum_{F\preceq E} p^{\xi_{\min E}+1}_F = 1.\] 

\end{proof}

We now relate the weights $p^\xi$ to the notion of $\xi$-sufficient probability block. For $M\in[\nn]$ and $n\in\nn$, we let $\mathbb{S}^\xi_{M,n}$ be the probability measure supported on $M|_{\xi,n}$ given by \[\mathbb{S}^\xi_{M,n}(i) = \left\{\begin{array}{ll} p^\xi_{[1,i]\cap M|_{\xi,n}} & : i\in M|_{\xi,n} \\ 0 & : i\in\nn\setminus M|_{\xi,n}. \end{array}\right.\] We recall the following. We let \[\mathfrak{S}_\xi=\{\mathbb{S}^\xi_{M,n}:M\in[\nn], n\in\nn\}.\]

\begin{lemma}[Corollary 4.9]\cite{CN} For $\xi<\omega_1$, $(\mathfrak{S}_\xi,\mathcal{S}_\xi)$ is $\xi$-sufficient.
\label{emit}
\end{lemma} \label{resultado2}

\subsection{The weight $q^{\xi,\zeta}$ on $\mathcal{S}_\zeta[\mathcal{S}_\xi]$}

Next, for $\xi,\zeta<\omega_1$, we define the weight $q^{\xi,\zeta}$ on $\mathcal{S}_\zeta[\mathcal{S}_\xi]$ and then we extend $q^{\xi,\zeta}$ to $[\nn]^{<\omega}\setminus\{\varnothing\}$. The definition is by induction on $\zeta$ for $\xi$ held fixed.

       We let $q^{\xi,0}_E=1$ for all $ E\in\mathcal{S}_\xi\setminus\{\varnothing\}$.    Assume that $q^{\xi, \zeta}$ has been defined on $\mathcal{S}_\zeta[\mathcal{S}_\xi]\setminus\{\varnothing\}$ for some $\zeta<\omega_1.$ 
              Fix $ E\in \mathcal{S}_{\zeta+1}[\mathcal{S}_\xi]\setminus\{\varnothing\}$.    We recall that $\mathcal{S}_{\zeta+1}[\mathcal{S}_\xi]=\mathcal{S}_1[\mathcal{S}_{\zeta}[\mathcal{S}_\xi]].$

   We then express $E=\cup_{n=1}^l F_n$, $F_1<\ldots <F_l$, $F_n\in MAX(\mathcal{S}_\zeta[\mathcal{S}_\xi])$ for $1\leqslant n<l$, and $F_l\in \mathcal{S}_\zeta[\mathcal{S}_\xi]$.     We then define \[q^{\xi, \zeta+1}_E= q^{\xi,\zeta}_{F_l}/\sqrt{\min E}.\]

    Assume $\zeta<\omega_1$ is a limit ordinal and $q^{\xi,\nu}$ has been defined for each $\nu<\zeta$ on $\mathcal{S}_\nu[\mathcal{S}_\xi].$
    Let $\zeta_n\uparrow \zeta$ be the characteristic sequence of $\zeta$ and let $E\in \mathcal{S}_\zeta[ \mathcal{S}_\xi]\setminus\{\varnothing\}$ then $E\in\mathcal{S}_{\min E+1}[ \mathcal{S}_\xi]$
and we define
 $$q^{\xi, \zeta}_E = q^{\xi,\zeta_{\min E+1}}_{E}.$$

Once $q^{\xi, \zeta}$ define on $\mathcal{S}_\zeta[ \mathcal{S}_\xi]\setminus\{\varnothing\}$ we extend  $q^{\xi, \zeta}$ to $[\nn]^{<\omega}\setminus\{\varnothing\}.$ Fix 
$E\in [\nn]^{<\omega}\setminus\{\varnothing\},$  we then express $E=\bigcup_{n=1}^lF_n$ where $F_1<\ldots<F_l,$ $F_n\in MAX(\mathcal{S}_\zeta[ \mathcal{S}_\xi])$ for $1\leqslant n<l$ and $F_l\in \mathcal{S}_\zeta[ \mathcal{S}_\xi])$ we then define $$q^{\xi, \zeta}_E=q^{\xi, \zeta}_{F_l}.$$

\begin{proposition}
  If $\xi,\zeta<\omega_1, $    $E\in MAX(\mathcal{S}_\zeta[\mathcal{S}_\xi])$  and $E=\cup_{r=1}^s F_r$, where $F_1<\ldots <F_s$, $F_r\in MAX(\mathcal{S}_\xi)$, $1\leqslant r\leqslant s$, then \[1= \sum_{r=1}^s \Bigl[\sum_{ \cup_{t=1}^{r-1}F_t \prec F\preceq \cup_{t=1}^r F_t} q^{\xi,\zeta}_Fp^\xi_F\Bigr]^2.\] 
\label{perm2}
\end{proposition}

\begin{proof}

 We work by induction on $\zeta$. The $\zeta=0$ case is Lemma \ref{perm1}.   We next prove the successor case.    Let $E\in MAX(\mathcal{S}_{\zeta+1}[\mathcal{S}_\xi])$, and write $E=\bigcup_{r=1}^sF_r $  where $F_1, \ldots, F_s$ are successive, maximal members of $\mathcal{S}_\xi$.  We note that necessarily $$(\min F_r)_{r=1}^s\in MAX(\mathcal{S}_{\zeta+1})=
MAX(\mathcal{S}_1[\mathcal{S}_\zeta]).$$ Therefore there exist $0=t_0<t_1<\ldots < t_j$ such that for each $1\leqslant m\leqslant j$, $$(\min F_r)_{r=t_{m-1}+1}^{t_m}\in MAX(\mathcal{S}_\zeta),$$ and such that $j= \min E=\min F_1$.     Therefore with $$H_m=\cup_{r=t_{m-1}+1}^{t_m} F_r,$$ it holds that $H_m\in MAX(\mathcal{S}_\zeta[\mathcal{S}_\xi])$ for each $1\leqslant m\leqslant j$. Note that $\min H_1=\min E=j$.  

  Define $$G_m=\cup_{t=1}^{t_{m-1}}F_t,$$ for $1\leqslant m\leqslant j$.   By the definition of $q^{\xi,\zeta}_F$, it holds that \[q^{\xi,\zeta+1}_{G_m\cup F}=q^{\xi,\zeta}_F/\sqrt{j}\] for each $1\leqslant m\leqslant j$ and $F\preceq H_m$.  
By the inductive hypothesis,  \begin{align*} \sum_{r=1}^s \Bigl[\sum_{\cup_{t=1}^{r-1} F_t \prec F\preceq \cup_{t=1}^r F_t} q^{\xi,\zeta+1}_Fp^\xi_F \Bigr]^2 = \sum_{m=1}^j \frac{1}{j}\sum_{r=t_{m-1}+1}^{t_m}\Bigl[\sum_{\varnothing\prec F\preceq H_m} q^{\xi,\zeta}_Fp^\xi_F  \Bigr]^2 = \sum_{m=1}^j \frac{1}{j} =1. \end{align*}

For the limit ordinal case, $E\in MAX(\mathcal{S}_\zeta[\mathcal{S}_\xi])$ implies that $E\in MAX(\mathcal{S}_{\zeta_{\min E}+1}[\mathcal{S}_\xi])$. If $E=\cup_{r=1}^s F_r,$ where $F_1, \ldots, F_r$ are successive, maximal members of $\mathcal{S}_\xi$, then $$q^{\xi,\zeta}_F=q^{\xi,\zeta_{\min E}+1}_F,$$ for each $F\preceq E$. By the inductive hypothesis, \[1=\sum_{r=1}^s \Bigl[\sum_{ \cup_{t=1}^{r-1}F_t \prec F\preceq \cup_{t=1}^r F_t} q^{\xi,\zeta_{\min E}+1}_Fp^\xi_F \Bigr]^2= \sum_{r=1}^s \Bigl[\sum_{ \cup_{t=1}^{r-1}F_t \prec F\preceq \cup_{t=1}^r F_t}q^{\xi,\zeta}_Fp^\xi_F \Bigr]^2.\]

\end{proof}

\begin{rem}\label{4}  It follows from the proof of Proposition \ref{perm2} that for any $E\in MAX(\mathcal{S}_{\zeta}[\mathcal{S}_\xi])$ and if $(E_r)_{r=1}^s$ is the $S_\xi$-decomposition of $E,$ then the function $q^{\xi,\zeta}$ is constant on on each interval $\{F : \cup_{r=1}^{i-1}E_r\prec F\preceq\cup_{r=1}^{i}E_r\}$
If we denote $a_i$ this common value, then $a_i$ are positive and $\sum_{i=1}^ma_i^2=1.$  
\end{rem}

\section{On weakly null sequences and weakly null collections in quotient spaces} \label{Quotient}

Before giving the definitions of properties ${S}_\xi$ and ${G}_{\xi,\zeta}$ in the next section, here we establish two   lifting propositions on weakly null sequences and weakly null collections in quotient spaces.

\begin{proposition} If $X$ is a Banach space not containing $\ell_1$,  $Y$ a subspace of  $X$, $Q:X\to X/Y$ the quotient map and  $(z_n)_{n=1}^\infty$ a weakly null sequence in $ B_{X/Y}$, then for any $L\in[\nn]$, there exist $M\in [L] $ and a weakly null sequence  $(x_n)_{n=1}^\infty$ in $3B_X$ such that $Qx_n=z_n$ for all $n\in M$.  
\label{legere}
\end{proposition}

\begin{proof} . Fix a sequence  $(u_n)_{n=1}^\infty$ in $(4/3)B_X$  such that $Qu_n=z_n$ for all $n\in\nn$.   Fix $L\in [\mathbb{N}]$. Since $X$ does not contain a copy of $\ell_1$,   there exists $M\in[L]$ such that the sequence $(u_n)_{n\in M}$ is weakly Cauchy.    Since $(z_n)_{n=1}^\infty$ is weakly null, we may choose finite sets $I_1<I_2<\ldots$, $I_n\subset M$, and positive numbers $(w_i)_{i\in \cup_{n=1}^\infty I_n}$ such that $$1=\sum_{i\in I_n}w_i \ \  \hbox{and} \ \ \|\sum_{i\in I_n} w_i z_i\|< 2^{-n}/3.$$   We may choose for each $n\in\nn$ some $v_n\in ( 2^{-n}/3) B_X$ such that $$Qv_n=\sum_{i\in I_n}w_iz_i.$$   Let $$x_n=u_n-\sum_{i\in I_n} w_iu_i + v_n,$$ for $n\in M$ and let $x_n=0$ for $n\in \nn\setminus M$.  Since $(u_n)_{n\in M}$ is weakly Cauchy,  the sequence 
$$\left(u_n-\sum_{i\in I_n} w_iu_i\right)_{n\in M}$$ 
is weakly null  and since $(v_n)_{n\in M}$ is norm null, the sequence $(x_n)_{n\in M}$ is weakly null.  Since $x_n=0$ for all $n\in \nn\setminus M$, $(x_n)_{n=1}^\infty $ is weakly null.    Note that if $n\in M$, \[Qx_n= Qu_n-\sum_{i\in I_n} w_iQx_i + Qv_n = z_n-\sum_{i\in I_n} w_iz_i +\sum_{i\in I_n}w_iz_i=z_n.\]   For each $n\in\nn$, $\|x_n\|=0$ if $n\in\nn\setminus M$, and for $n\in M$, \[\|x_n\|\leqslant \|u_n\|+\sum_{i\in I_n}w_i \|u_i\|+ \|v_n\| < 4/3+4/3+1/3=3.\]

\end{proof}

\begin{proposition} Let $X$ be a Banach space  not containing $\ell_1$,  $Y$ a  subspace of $X$, $Q:X\to X/Y$ the quotient map and let  $(u_E)_{E\in[\nn]^{<\omega}\setminus
\{\varnothing\}}$    be a weakly null collection in $B_{X/Y}$.   Then for any $L_0\in [\nn]$, there exist  a weakly null collection  $(x_E)_{ E\in[\nn]^{<\omega}\setminus\{\varnothing\}}$ in $ 3B_X$  and $L\in[L_0]$ such that for any $M\in[L]$, there exists $N\in[M]$ such that $Qx_{N|n}=u_{N|n}$ for all $n\in\nn$. 
\label{sara}
\end{proposition}

\begin{proof} We recall the convention that $(k)>\varnothing$ for all $k\in \nn$.   Let $E_1, E_2, \ldots$ be an enumeration of $[\nn]^{<\omega}.$   By Proposition \ref{legere}, for each $n\in\nn$ and $L\in[\nn]$, there exist $M\in[L]$ and a weakly null collection  $(x_{E_n\smallfrown(k)})_{k>E_n}$ in $3B_X$  such that $$Qx_{E_n\smallfrown (k)}=u_{E_n\smallfrown (k)},$$ for all $k\in M, k>E_n$. 

By Proposition \ref{legere}, we can recursively select $L_0\supset L_1\supset L_2\supset \ldots$ and weakly null sequences $(x_{E_n\smallfrown(k)})_{k> E_n}$ in $3B_X$ such that for all $n\in\nn,$     for all   $k\in L_n, E_n<k$ $$Qx_{E_n\smallfrown(k)}=u_{E_n\smallfrown(k)}.$$  This defines the entire collection $(x_E)_{\varnothing\neq E\in [\nn]^{<\omega}}$, which is a weakly null collection, since $(x_{E_n\smallfrown (k)})_{k>E_n}$ is weakly null for each $n\in\nn$.     Choose $l_1<l_2<\ldots$ such that  $l_n\in L_n, E_n<l_n$  and let $L=(l_1, l_2, \ldots)$.   Note that  $\{l_n,l_{n+1},\dots\} \subset L_n $  for all $n\in\nn$, from which it follows that for any $M\in[L]$, $k,m\in \nn$, there exist $n\in (m,\infty)\cap M\cap L_k$.       Fix $k_1\in \nn$ such that $\varnothing=E_{k_1}$ and fix $n_1\in M\cap L_{k_1}$. Note that \[Qx_{(n_1)}= Qx_{E_{k_1}\smallfrown (n_1)} = u_{E_{k_1}\smallfrown (n_1)} = u_{(n_1)},\] since $E_{k_1}<n_1\in L_{k_1}$.     Now assuming $n_1<\ldots < n_t$ have been chosen, fix $k_{t+1}\in \nn$ such that $(n_1, \ldots, n_t)=E_{k_{t+1}}$.  Fix $n_{t+1}\in (n_t, \infty) \cap M\cap L_{k_{t+1}}$.    Then \[Qx_{(n_1, \ldots, n_{t+1})}= Qx_{E_{k_{t+1}}\smallfrown (n_{t+1})} = u_{E_{k_{t+1}}\smallfrown (n_{t+1})}=u_{(n_1, \ldots, n_{t+1})},\] since  $n_{t+1}\in L_{k_{t+1}}$ and $E_{k_{t+1}}<n_{t+1}$. This completes the recursive choice  and we are done with $N=(n_1, n_2, \ldots)$.

\end{proof}

\section{The properties ${S}_\xi$ and ${G}_{\xi,\zeta}$ for weakly null collections in Banach spaces} \label{Properties G}

In this section we  define  two
 properties  ${S}_\xi$ and ${G}_{\xi,\zeta}$  for weakly null collections in Banach spaces. (Definitions \ref{D2} and \ref{D5}). It is also  convenient to introduce two ordinals $\mathbf{s}_{\xi}(A)$ and  $\mathbf{g}_{\xi,\zeta}(A)$  related respectively to these properties for any operator $A$ (Definitions \ref{D3} and \ref{D6}).  Moreover, we will emphasize
 two  classes of operators $\mathbf{S}_\xi$ and $\mathbf{G}_{\xi,\zeta}$  defined in terms of these properties respectively (Definitions \ref{D4} and \ref{D7}). 

\begin{definition} Let $\mathbf{u}=(u_E)_{ E\in[\nn]^{<\omega}\setminus\{\varnothing\}}$  be a collection of vectors in a Banach space $U.$  For   ordinals  $\zeta, \xi<\omega_1$, $F\in MAX(S_\xi)$ or $E\in MAX(S_\zeta[S_\xi])$ we define the $\ell_1^+$ combination $\mathbb{E}_{\xi,F}\mathbf{u}$ of $\mathbf{u}$ on $F$ and the $\ell_2^+$ combination $\mathbb{E}^{(2)}_{\xi,\zeta,E}\mathbf{u}$ of $\mathbf{u}$ on $E$ by

$$
\mathbb{E}_{\xi,F}\mathbf{u}=
\sum_{\varnothing\prec G \preceq F} p^\xi_G u_G \ \ \hbox{and} \ \  
\mathbb{E}^{(2)}_{\xi,\zeta,E}\mathbf{u}=
\sum_{\varnothing\prec G\preceq E} q_G^{\xi,\zeta}p^\xi_G u_G.
$$

For any  $M\in[\nn]$ and any $n\in \nn$ we define
$$\mathbb{E}_{\xi,M}\mathbf{u}(n)=\sum_{M\vert_{\xi,n-1}\prec E\preceq M\vert_{\xi,n}}
p^\xi_Eu_E \ \ \ \hbox{and} \ \ \
\mathbb{E}_{\xi,M}\mathbf{u}=(\mathbb{E}_{\xi,M}\mathbf{u}(n))_{n=1}^\infty,$$ 
and  also
$$\mathbb{E}^{(2)}_{\xi,\zeta,M}\mathbf{u}(n)=\sum_{M\vert_{\xi,\zeta,n-1}\prec E\preceq M\vert_{\xi,\zeta,n}}
q^{\xi,\zeta}_Ep^\xi_Eu_E \ \ \ \hbox{and} \ \ \ 
\mathbb{E}^{(2)}_{\xi,\zeta,M}\mathbf{u}=(\mathbb{E}^{(2)}_{\xi,\zeta,M}\mathbf{u}(n))_{n=1}^\infty$$
\end{definition}

\begin{rem}\label{lemasperm} 
It follows from  Lemma \ref{perm1} that 
$$\mathbb{E}_{\xi,M}\mathbf{u}(n)\in \mathrm{co}\{u_E:M\vert_{\xi,n-1}\prec E\preceq M\vert_{\xi,n}\}.$$

An easy induction argument on $\zeta$ using  Remark \ref{4}  yields that, for
 any $n\in\nn$, 
 $$\mathbb{E}^{(2)}_{\xi,\zeta,M}\mathbf{u}(n)\in \Bigl\{\sum_{m=1}^\infty a_m \mathbb{E}_{\xi,M}\mathbf{u}(m): (a_m)_{n=1}^\infty\in c_{00}, \sum_{m=1}^\infty |a_m|^2=1\Bigr\}.$$
\end{rem}

\begin{definition} \label{D2} Given an ordinal $\xi<\omega_1$,  a constant $\gamma\geqslant 0$,  Banach spaces $U,V$ and an operator 
$
A:U\longrightarrow V,$ we say that $A$ has  the $\gamma$-$S_\xi$ \emph{Property} provided that, for any  weakly null  collection $\mathbf{u}=(u_E)_{ E\in[\nn]^{<\omega}\setminus\{\varnothing\}}$ in $ B_U$  and any $L\in[\nn]$, there exists $M\in[L]$ such that $\|A\mathbb{E}_{\xi,M}\mathbf{u}\|_2^w \leqslant \gamma$ where  $A\mathbb{E}_{\xi,M}\mathbf{u}=(A\mathbb{E}_{\xi,M}\mathbf{u}(n))_{n=1}^\infty.$

We say the Banach space $U$ has the $\gamma$-$\mathcal{S}_\xi$ \emph{Property} if $I_U$ has it. We say $A$ (resp.$U$) has the $\mathcal{S}_\xi$ \emph{Property} if it has the $\gamma$-$\mathcal{S}_\xi$ \emph{Property} for some $\gamma\geqslant0.$

\end{definition}

\begin{definition} \label{D3}

For an operator $A:U\longrightarrow V$ we define
\[\mathbf{s}_{\xi}(A)=\begin{cases}\inf \gamma&\text{if }A\text{ has the } \gamma\text{-}{S}_{\xi}\text{ property for some }  0\leq\gamma.\\
\infty&\text{otherwise.}
\end{cases}
\]
and for a Banach space $U$ we define  $\mathbf{s}_{\xi}(U)=\mathbf{s}_{\xi}(I_U)$
\end{definition}

\begin{definition} \label{D4}
For Banach spaces $U,V$ we let $\mathbf{S}_\xi(U,V)$    denote the class of all operators $A:U\longrightarrow V$ which have property $S_\xi.$
 We let $\mathbf{S}_\xi$ denote the class $$\bigcup_{U,V\in \mathtt{Ban}}\mathbf{S}_\xi(U,V),$$ where $\mathtt{Ban}$ is the class of all Banach spaces over $\mathbb{K}.$
 \end{definition}
 \begin{rem}\label{perm3}\upshape    If $M\in[\nn]$ and if $\mathbf{u}=(u_E)_{ E\in[\nn]^{<\omega}\setminus\{\varnothing\}}$ is a collection in $U,$ it follows from Remark \ref{lemasperm} that   $\|A\mathbb{E}_{\xi,M}\mathbf{u}\|_2^w \leqslant \gamma$  implies $\|A\mathbb{E}^{(2)}_{\xi,\zeta, M} \mathbf{u}  (m)\|\leqslant \gamma$ for all $\zeta<\omega_1$ and $m\in\nn$.     One reason we isolate the property $S_\xi$ is to guarantee that the $\ell_2$ averaging in the definition of $\mathbb{E}^{(2)}_{\xi, \zeta, M}\mathbf{u}$ results in sequences that remain bounded. 

\end{rem}
\begin{definition} \label{D5} Given  ordinals $\xi, \zeta<\omega_1$,  a constant $\gamma\geqslant 0$,  Banach spaces $U,V$  and an operator $A:U\longrightarrow V$ we say that $A$ has the $\gamma$-${G}_{\xi,\zeta}$ \emph{Property} provided that, for any  weakly null  collection $\mathbf{u}=(u_E)_{E\in[\nn]^{<\infty}\setminus\{\varnothing\}}$ in $ B_U$  and any $L\in[\nn]$, there exists $M\in[L]$ such that  $\|A\mathbb{E}^{(2)}_{\xi,\zeta,M}\mathbf{u}\|_1^w \leqslant \gamma$ where $A\mathbb{E}^{(2)}_{\xi,\zeta,M}\mathbf{u}=(A\mathbb{E}^{(2)}_{\xi,\zeta,M}\mathbf{u}(n))_{n=1}^\infty.$

We say the Banach space $U$ has the $\gamma$-${G}_{\xi,\zeta}$ \emph{Property} if $I_U$ has the $\gamma$-${G}_{\xi,\zeta}$ \emph{Property} . We say $A$ (resp. $U$) has the ${G}_{\xi,\zeta}$ \emph{Property} if it has $\gamma$-${G}_{\xi,\zeta}$ \emph{Property}  for some $\gamma\geqslant0.$
\end{definition}

\begin{definition} \label{D6}
 For an operator $A:U\longrightarrow V$ we denote
$$\mathbf{g}_{\xi,\zeta}(A)=\begin{cases}\inf \gamma&\text{if }A \text{ has the } \gamma\text{-}{G}_{\xi,\zeta}\text{ property for some }  0\leq\gamma.\\
\infty&\text{otherwise}
\end{cases}
$$
and for a Banach space $U$ we define $\mathbf{g}_{\xi,\zeta}(U)=\mathbf{g}_{\xi,\zeta}(I_U).$
\end{definition} 
\begin{definition} \label{D7}
For Banach spaces $U,V$ we let $\mathbf{G}_{\xi,\zeta}(U,V)$ denote the class of all operators $A:U\longrightarrow V$ which have property $G_{\xi,\zeta}.$
 We let $\mathbf{G}_{\xi,\zeta}$ denote the class $$\bigcup_{U,V\in \mathtt{Ban}}\mathbf{G}_{\xi,\zeta}(U,V).$$ 
\end{definition}

We recall that an operator $A:U\to V$ is \emph{completely continuous} provided that it maps weakly null sequences to norm null sequences.  

\begin{rem}\upshape \label{cc operator} If $A:U\to V$ is  a completely continuous operator, then $A$ has the $S_\xi$-Property and $\mathbf{s}_\xi(A)=0.$  In particular, $\mathbf{S}_\xi$ contains all compact operators, and in particular all finite rank operators. 
\end{rem}

\begin{proof}

Indeed, fix $\gamma>0.$ Let $\mathbf{u}=(u_E)_{ E\in[\mathbb{N}]^{<\omega}\setminus\{\varnothing\}}$ be a weakly null collection in $B_U$ and $L\in [\nn].$ Since $A$ sends weakly null sequences to norm null sequences,  we can select by induction a sequence $M=(l_1,l_2,\ldots)\in [L]$ such that $$\sum_{n=1}^\infty\Vert Au_{(l_1,...,l_n)}\Vert <\gamma.$$

It follows from Remark \ref{lemasperm} that, for all $m\in \mathbb{N},$
  $$\Vert A\mathbb{E}_{\xi,M}\mathbf{u}(m)\Vert\leqslant\sum_{M\vert_{\xi,m-1}\prec E\preceq M\vert_{\xi,m}}\Vert Au_E\Vert,$$ and then
$$\Vert A\mathbb{E}_{\xi,M}\mathbf{u}\Vert_2^w
\leqslant\sum_{m=1}^\infty
\sum_{M\vert_{\xi,m-1}\prec E\preceq M\vert_{\xi,m}}\Vert Au_E\Vert=\sum_{m=1}^\infty\Vert Au_{(l_1,...,l_m)}\Vert<\gamma.$$ Since $\gamma$ was arbitrary,  we have shown that $\mathbf{s}_\xi(A)=0.$ 
\end{proof}

\section{On  operators having $\gamma$-$S_\xi$ and $\gamma$-$G_{\xi, \zeta}$  properties }

Let us recall a special case of the infinite Ramsey theorem, the proof of which was achieved in steps by
Nash-Williams \cite{NW}, Galvin and Prikry \cite{GP}, Silver \cite{S}, and Ellentuck \cite{E}.  See Diestel [Corollary p.96] \cite{D}  for a complete proof.

\begin{theorem} If $\mathcal{V}\subset [\nn]$ is closed, then for any $L\in[\nn]$, there exists $M\in[L]$ such that either $[M]\subset \mathcal{V}$ or $[M]\cap \mathcal{V}=\emptyset$. 

\label{Ramsey}
\end{theorem}

We deduce the following easy consequence on  operators having $\gamma$-$S_\xi$ and $\gamma$-$G_{\xi, \zeta}$  properties.

\begin{proposition} Let $\zeta, \xi$ be countable ordinals, $\gamma\geqslant 0$ a constant, $U,V$ Banach spaces, and $A:U\to V$ an operator. 
\begin{enumerate}[(i)]\item If $A$ has the $\gamma$-$S_\xi$ Property, then for any weakly null collection  $\mathbf{u}=(u_E)_{E\in [\nn]^{<\omega}\setminus\{\varnothing\}}$ in $ B_U$  and any $L\in[\nn]$, there exists $M\in [L]$ such that for all $N\in[M],$ $\Vert A\mathbb{E}_{\xi,N}\mathbf{u}\Vert_2^w\leqslant\gamma.$
\item If $A$ has the $\gamma$-$G_{\xi, \zeta}$ Property, then  for any  weakly null collection  $\mathbf{u}=(u_E)_{E\in [\nn]^{<\omega}\setminus\{\varnothing\}}$ in $ B_U,$   any $L\in[\nn]$, there exists $M\in [L]$ such that for all $N\in[M],$ $\Vert A\mathbb{E}_{\xi,\zeta, N}^{(2)}\mathbf{u}\Vert_1^w\leqslant\gamma.$
 \end{enumerate}
\label{hereditary}
\end{proposition}

\begin{proof}$(i)$ Assume  that $A:U\to V$ has the $\gamma$-$S_\xi$ Property. Fix a weakly null collection $\mathbf{u}=(u_E)_{ E\in[\nn]^{<\omega}\setminus\{\varnothing\}}$  in $B_U$  and $L\in[\nn]$.    For $m\in\nn$, let $\mathcal{V}_m$ denote the set of all $M\in[\nn]$ such that $$\|(A\mathbb{E}_{\xi, M}\mathbf{
u}(n))_{n=1}^m\|_2^w \leqslant \gamma.$$   Let $\mathcal{V}=\cap_{m=1}^\infty \mathcal{V}_m$.  Note that the conclusion of $(i)$ is that there exists $M\in [L]$ such that $$[M] \subset \mathcal{V}.$$  The hypothesis that $A$  has the $\gamma$-$S_\xi$ Property  implies that for any $M\in [L]$, $[M]\cap \mathcal{V}\neq \varnothing$. Therefore in order to deduce the existence of the desired $M$, we need only to show that each set $\mathcal{V}_m$ is closed  and by Theorem \ref{Ramsey},  there must exist some  $M\in[L]$  such that either $[M]\subset \mathcal{V}$ or  $[M]\cap \mathcal{V}= \varnothing$.  But the $\gamma$-$S_\xi$ Property precludes the latter alternative.  Let us show that $\mathcal{V}_m$ is closed.   By the permanence property, if $M\in[\nn]$ is fixed and $E_1<\ldots  <E_m$ are successive, maximal members of $\mathcal{S}_\xi$ such that $\cup_{n=1}^m E_n$ is an initial segment of $M$, then $$\mathbb{E}_{\xi,M}\mathbf{u}(n) = \mathbb{E}_{\xi,N}\mathbf{u}(n),$$ for each $1\leqslant n\leqslant m$ and any $N\in[\nn]$ such that $\cup_{n=1}^m E_n$  is an initial segment of $ N$.   From this it follows that for a fixed $m$, the map $$M\mapsto \|(A\mathbb{E}_{\xi,M} \mathbf{u}(n))_{n=1}^m\|_2^w$$ is locally constant, and therefore continuous.      From this it follows that the preimage of $[0,\gamma]$ under this map, which is $\mathcal{V}_m$, is closed.

$(ii)$ The proof of $(ii)$ is an inessential modification of the proof of $(i)$.  

\end{proof}

\section{On   classes of operators $\mathbf{S}_\xi$ and  $\mathbf{G}_{\xi,\zeta}$ and the functionals $\mathbf{s}_\xi$ and $\mathbf{g}_{\xi,\zeta}$}

The purpose of this section is to establish some basic properties about the  classes of operators $\mathbf{S}_\xi$ and  $\mathbf{G}_{\xi,\zeta}$ and the functionals $\mathbf{s}_\xi$ and $\mathbf{g}_{\xi,\zeta}$  that will be used later, some of them without explicit mention.

\begin{rem}\upshape It is quite clear that if $Y$ is a subspace of $X$, then $\mathbf{s}_\xi(Y)\leqslant \mathbf{s}_\xi(X)$ and $\mathbf{g}_{\xi,\zeta}(Y)\leqslant \mathbf{g}_{\xi,\zeta}(X)$ for any countable $\xi, \zeta$.  Moreover, it is easy to see that for any countable $\xi, \zeta$, an operator $A:U\longrightarrow V$ has the $\mathcal{S}_\xi$ \textit{Property} (resp. $\mathcal{G}_{\xi,\zeta} $  \textit{Property}) if, and only if, it  has the $\mathcal{S}_\xi$  \textit{Property} (resp. $\mathcal{G}_{\xi,\zeta} $  \textit{Property}) for $U, V$ endowed with equivalent norms. Therefore, when considering membership of an operator in the classes $\mathbf{S}_\xi, \mathbf{G}_{\xi, \zeta}$, we can renorm the domain and codomain for convenience. 
\label{renorm}
\end{rem}

\begin{proposition} Fix $\xi, \zeta<\omega_1$. Let $U,V$ be Banach spaces.
 \begin{enumerate}[(i)]\item The functionals $\mathbf{s}_\xi$ and $\mathbf{g}_{\xi,\zeta}$ are seminorms on $\mathbf{S}_\xi(U,V)$ and $\mathbf{G}_{\xi, \zeta}(U,V)$, respectively.
 \item The classes $\mathbf{S}_\xi$, $\mathbf{G}_{\xi,\zeta}$ are operator ideals.     \item If $X$ is a Banach space which does not contain an isomorph of $\ell_1$ and $Q:X\to U$ is a surjection such that $AQ\in \mathbf{S}_\xi$ (resp. $AQ\in \mathbf{G}_{\xi,\zeta}$), then $A\in \mathbf{S}_\xi$ (resp. $\mathbf{G}_{\xi,\zeta}$).    \item If $X$ is a Banach space and $j:V\to X$ is an isomorphic embedding such that $jA\in \mathbf{S}_\xi$ (resp. $jA\in \mathbf{G}_{\xi,\zeta}$), then $A\in \mathbf{S}_\xi$ (resp. $A\in \mathbf{G}_{\xi,\zeta}$).

\end{enumerate}
\label{ideal}
\end{proposition}

\begin{proof}$(i)$ Fix $A,B\in \mathbf{S}_\xi(U,V)$.   Fix $\alpha > \mathbf{s}_\xi(A)$ and $\beta> \mathbf{s}_\xi(B)$.    By the definition of $\mathbf{s}_\xi$, $A$ has the $\alpha$-$S_\xi$ Property and $B$ has the $\beta$-$S_\xi$ Property. Fix $L\in[\nn]$. Let  $\mathbf{u}=(u_E)_{ E\in [\nn]^{<\omega}\setminus\{\varnothing\}}$  be a weakly null collection in $ B_U$.  By Proposition \ref{hereditary}, there exists $M_1\in [L]$ such that for all $N\in [M_1]$, $$\|A\mathbb{E}_{\xi,N}\mathbf{u}\|_2^w \leqslant \alpha.$$    Since $B$ has the $\beta$-$S_\xi$ Property, there exists $M\in [M_1]$ such that  $$\|B\mathbb{E}_{\xi,M}\mathbf{u}\|_2^w \leqslant \beta.$$   Therefore \[\|(A+B)\mathbb{E}_{\xi,M}\mathbf{u}\|_2^w \leqslant \|A\mathbb{E}_{\xi,M}\mathbf{u}\|_2^w + \|B\mathbb{E}_{\xi,M}\mathbf{u}\|_2^w \leqslant \alpha + \beta.\]   Since this holds for any such $\mathbf{u}$ and $L$, $\mathbf{s}_\xi(A+B)\leqslant \alpha + \beta$.    Its follows immediately that  $$\mathbf{s}_\xi(A+B) \leqslant \mathbf{s}_\xi(A)+\mathbf{s}_\xi(B)$$ 
 and   $A+B\in \mathbf{S}_\xi(U,V)$.  It is quite clear that $\mathbf{s}_\xi(cA) =|c| \mathbf{s}_\xi(A)$  any scalar $c$.    The proof  of the analogous statement for  $\mathbf{g}_{\xi, \zeta}$ is similar.

$(ii)$    
Next, we show the ideal property.  We recall  that $\mathbf{S}_\xi$ contains all compact operators and, in particular, all finite rank operators (Remark \ref{cc operator}).  Fix Banach spaces $T,U,V,W$ and operators $C:T\to U$, $B:U\to V$, $A:V\to W$ such that $B\in \mathbf{S}_\xi$, $\|A\|, \|C\|\leqslant 1$.   Then  for any weakly null collection $\mathbf{t}=(t_E)_{ E\in [\nn]^{<\omega}\setminus\{\varnothing\}}$ in $B_T$  and $L\in[\nn]$, since  $$\mathbf{u}:= (Ct_E)_{\varnothing \neq E\in[\nn]^{<\omega}}$$   is   a  weakly null collection in $B_U$,  for any $\ee>0$, there exists $M\in [L]$ such that $$\|BC\mathbb{E}_{\xi,M} \mathbf{t}\|_2^w = \|B\mathbb{E}_{\xi,M} \mathbf{u}\|_2^w \leqslant \mathbf{s}_\xi(B)+\ee.$$    Then $$\|ABC\mathbb{E}_{\xi,M} \mathbf{t}\|_2^w \leqslant \|A\|\|BC \mathbb{E}_{\xi,M} \mathbf{t}\|_2^w \leqslant \mathbf{s}_\xi(B)+\ee.$$  Since such $\mathbf{t}$, $L$, and $\ee>0$ were arbitrary,  $\mathbf{s}_\xi(ABC)\leqslant \mathbf{s}_\xi(B)$, and $ABC\in \mathbf{S}_\xi$.   If we remove the assumption that $\|A\|,\|C\|\leqslant 1$, then by homogeneity, we obtain the estimate $$\mathbf{s}_\xi(ABC)\leqslant \Vert A\Vert \;\Vert C\Vert \  \mathbf{s}_\xi(B).$$  This concludes $(ii)$ for $\mathbf{S}_\xi$.  

For $\mathbf{G}_{\xi,\zeta}$, the proof is an inessential modification. The only difference is that in the proof that $\mathbf{G}_{\xi,\zeta}$ contains all completely continuous operators, we use the membership \[A\mathbb{E}^{(2)}_{\xi,\zeta ,M}\mathbf{u}(n)\in \Bigl\{\sum_{m=1}^\infty a_mA u_{M|m}: (a_m)_{m=1}^\infty \in c_{00}\cap B_{\ell_\infty}\Bigr\} \] for all $n\in\nn$.

$(iii)$ Assume $AQ\in \mathbf{S}_\xi$.   If $Q:X\to U$ is a surjection, then $U$ is isomorphic to $X/\ker(Q)$. By renorming $U$, we can assume $Q:X\to U$ is a quotient map. We note that this does not affect the conclusion of $(iii)$, as noted in Remark \ref{renorm}.    

Let given $\mathbf{u}=(u_E)_{ E\in [\nn]^{<\omega}\setminus\{\varnothing\}}$  a weakly null collection in $ B_U$ and $L_0\in[\nn].$ By Proposition \ref{sara}, there exist a weakly null collection $\mathbf{x}=(x_E)_{E\in[\nn]^{<\omega}\setminus\{\varnothing\}}$ in $ 3B_X$ and $L\in[L_0]$ such that for any $M\in[L]$, there exists $N\in[M]$ such that $$Qx_{N|n}=u_{N|n},$$ for all $n\in\nn$.  The latter condition implies that for any $M\in [L]$, there exists $N\in [M]$ such that $$AQ\mathbb{E}_{\xi,N}\mathbf{x}= A\mathbb{E}_{\xi,N}\mathbf{u}.$$   
       
       Since $AQ\in \mathbf{S}_\xi$ and since $(x_E/3)_{E\in [\nn]^{<\omega}\setminus\{\varnothing\}}$  is a weakly null collection in $B_X$, by Proposition \ref{hereditary},
       for any $\gamma> \mathbf{s}_\xi(AQ)$, there exists $M\in[L]$ such that for all $N\in [M]$, $$\|AQ\mathbb{E}_{\xi,N}\mathbf{x}\|_2^w \leqslant 3\gamma.$$    In particular, if $N\in[M]$ is chosen such that $$AQ\mathbb{E}_{\xi,N}\mathbf{x}=A\mathbb{E}_{\xi,N}\mathbf{u},$$ then \[ \|A\mathbb{E}_{\xi,N}\mathbf{u}\|_2^w = \|AQ\mathbb{E}_{\xi,N}\mathbf{x}\|_2^w\leqslant 3\gamma.\]   From this it follows that $\mathbf{s}_\xi(A)\leqslant 3\mathbf{s}_\xi(AQ).$   The proof for $\mathbf{g}_{\xi,\zeta}$ is similar.

$(iv)$ As noted in Remark \ref{renorm}, by renorming $V$, we can assume $j$ is an isometric embedding.  Then since $$\|jA\mathbb{E}_{\xi,M}\mathbf{u}\|_2^w=\|A\mathbb{E}_{\xi,M}\mathbf{u}\|_2^w,$$ for  $M\in[\nn]$, and any weakly null collection $\mathbf{u}$ in $B_U$, it follows immediately that $\mathbf{s}_\xi(A)=\mathbf{s}_\xi(jA).$  The proof for $\mathbf{g}_{\xi,\zeta}$ is similar.

\end{proof}

\begin{corollary} Fix ordinals $\xi, \zeta<\omega_1$.  Let $X,X_1, Y,Y_1$ be Banach spaces and let $A:X\to X_1$, $B:Y\to Y_1$ be operators such that $B$ is finite rank. Let $A\otimes B:X\tim_\pi Y\to X_1\tim_\pi Y_1$ be the  tensor product of $A$ and $B$. If $\mathbf{g}_{\xi, \zeta}(A)<\infty$ (resp. $\mathbf{g}_{\xi, \zeta}(A)=0$) then $\mathbf{g}_{\xi, \zeta}(A\otimes B)<\infty$ (resp. $\mathbf{g}_{\xi ,\zeta}(A\otimes B)=0$). 

\label{finite}
\end{corollary}

\begin{proof} If $B$ is the zero operator, then so is $A\otimes B$, and $\mathbf{g}_{\xi,\zeta}(A\otimes B)=0<\infty$.  

If $B$ is rank $1$, then $B=y^*\otimes y_1$ for some $y^*\in Y^*$ and $y_1\in Y_1.$  Let $L: X\widehat{\otimes}_\pi Y\to  X$ to be the operator such that $L(x\otimes y)=y^*(y)x$ for each $x\in X$ and $y\in Y.$ Let  $R:X_1\to X_1\widehat{\otimes}_\pi Y_1$ to be the operator such that by  $R(x_1)=x_1\otimes y_1$ for each $x_1\in X_1.$ Then, for any $x\in X$ and $y\in Y,$   $$A(x)\otimes B(y)=RAL(x\otimes y),$$ so $A\otimes B=RAL$ , where $\|L\|\|R\|=\|y^*\|\|y_1\|$.  By Proposition \ref{ideal} $(ii)$, $$\mathbf{g}_{\xi, \zeta}(A\otimes B) \leqslant \mathbf{g}_{\xi, \zeta}(A)\|y^*\|\|y_1\|.$$  This inequality implies the result in the case that $B$ is rank $1$.

If the rank of $B$ is $n>1$, then we can write $B=\sum_{i=1}^n B_i$, where $B_i$ is rank $1$ for each $1\leqslant i\leqslant n$.    Then \[\mathbf{g}_{\xi, \zeta}(A\otimes B) = \mathbf{g}_{\xi,\zeta}\Bigl(A\otimes \sum_{i=1}^n B_i\Bigr) =\mathbf{g}_{\xi, \zeta}\Bigl(\sum_{i=1}^n A\otimes B_i\Bigr) \leqslant \sum_{i=1}^n \mathbf{g}_{\xi, \zeta}(A\otimes B_i).\]  We conclude by the previous paragraph.

\end{proof}

 \section{Sufficient conditions for operator $A$ to satisfy  $\mathbf{g}_{\xi, \zeta}(A)=0$}

We next show that once the prescribed linear combinations of branches of weakly null trees become weakly $1$-summing,  higher linear combinations can be made arbitrarily close to $0$ in norm (Proposition \ref{die}).   We start by proving the following lemma.

\begin{lemma}\label{diagonal}Suppose that $U,V$ are Banach spaces, $\xi, \zeta<\omega_1$ are ordinals, and $A:U\to V$ is an operator such that for each  weakly null collection $\mathbf{u}=(u_E)_{ E\in[\nn]^{<\omega}\setminus\{\varnothing\}}$ in $ B_U$,   each $\ee>0$, and each $L\in[\nn]$, there exists $E\in [L]^{<\omega}\cap MAX(S_\zeta[S_\xi])$ such that $\|A\mathbb{E}^{(2)}_{\xi,\zeta, E}\mathbf{u}\|<\ee$. Then
$\mathbf{g}_{\xi, \zeta}(A)=0$.
\end{lemma}
\begin{proof}Fix a weakly null collection $\mathbf{u}=(u_E)_{E\in[\nn]^{<\omega}\setminus\{\varnothing\}}$ in $ B_U$,  $\ee>0$, and  $L\in[\nn].$ We construct  by induction a  sequence $(E_n)_{n=1}^\infty $ of successive elements in $[L]^{<\omega}\cap MAX(S_\zeta[S_\xi])$ as follows. By hypothesis we can select $E_1\in[L]^{<\omega}\cap MAX(S_\zeta[S_\xi])$ such that $$\Vert A \mathbb{E}^{(2)}_{\xi,\zeta,E_1}\mathbf{u}\Vert<\ee/2.$$
Now assuming that $E_1< \ldots < E_n$ have been choosen,  we denote $\Delta_n=E_1\cup \ldots \cup E_n$ and
we define $\mathbf{v}=(v_E)_{ E\in[\nn]^{<\omega}\setminus\{\varnothing\}}$ in $B_U$ by letting  $v_E=u_{\Delta_n\cup E}$ if $E_n<E$ and $v_E=0$ otherwise. Then $\mathbf{v}$ is weakly null.
By hypothesis we can select $$E_{n+1}\in [(\max E_n,\infty)\cap L ]^{<\omega}\cap MAX(S_\zeta[S_\xi])$$ such that 
$$\Vert A \mathbb{E}^{(2)}_{\xi,\zeta,E_{n+1}}\mathbf{v}\Vert<\ee/2^{n+1}.$$

We denote $M=\bigcup_{n=1}^\infty E_n.$ The sequence $(E_n)_{n=1}^\infty$ is the $S_\zeta[S_\xi]$ decomposition of $M$ and, for every integer $n,$ $\Delta_n=M\vert_{\xi,\zeta,n}.$
It follows from the definition of the weights $p^\xi$ and $q^{\xi,\zeta}$ that for $\varnothing\prec F\preceq E_{n+1}$ and $G=\Delta_n\cup F$ then $p^\xi_F=p^\xi_G$ and $q^{\xi,\zeta}_F=q^{\xi,\zeta}_G$ so
\begin{align*}
\left\Vert A \mathbb{E}^{(2)}_{\xi,\zeta,E_{n+1}}\mathbf{v}\right\Vert&=\left\Vert\sum_{\varnothing\prec F\preceq E_{n+1}}q^{\xi,\zeta}_Fp^\xi_FAu_{\Delta_n\cup F}\right\Vert\\
&=\left\Vert\sum_{M\vert_{\xi,\zeta,n}\prec G\preceq M\vert_{\xi,\zeta,n+1}}q^{\xi,\zeta}_Gp^\xi_GAu_G\right\Vert=\left\Vert A \mathbb{E}^{(2)}_{\xi,\zeta,M}\mathbf{u}(n+1)\right\Vert<\varepsilon/2^{n+1}.\\
\end{align*}
Therefore
$$\Vert A\mathbb{E}^{(2)}_{\xi,\zeta,M}\mathbf{u}\Vert_1^w\leq \sum_{n=1}^\infty\Vert A\mathbb{E}^{(2)}_{\xi,\zeta,M}\mathbf{u}(n)\Vert<\varepsilon.
$$
Since $\varepsilon>0, \mathbf{u}\text{ weakly null collection in }B_U,$ and $L\in[\nn]$ were arbitrary, $\mathbf{g}_{\xi,\zeta}(A)=0.$
\end{proof}

\begin{proposition}\label{die} Let $U,V$ be Banach spaces and fix $\xi,\zeta<\omega_1$.    
If $A:U\to V$ is an operator such that $\mathbf{s}_\xi(A)<\infty$ and
 $\mathbf{g}_{\xi,\zeta}(A)<\infty$, then, for any $0<\eta<\omega_1,$ for any $\varepsilon>0,$ any weakly null collection $\mathbf{u}=(u_E)_{E\in [\nn]^{<\omega}\setminus\{\varnothing\}}$ in $B_U$ and any $L\in [\nn],$ there exists $M\in[ L]$  such that for any $N\in [M],$ $\Vert A\mathbb{E}^{(2)}_{\xi,\zeta+\eta,N}\mathbf{u}(1)\Vert<\varepsilon.$ In particular $\mathbf{g}_{\xi,\zeta+\eta}(A)=0.$
\end{proposition}
\begin{proof} 

We note that the last sentence of the proposition follows  from the preceding  one together with Lemma \ref{diagonal}. Therefore we prove only the first assertion.
By the infinite Ramsey theorem it is sufficient to prove that 
for any $0<\eta<\omega_1,$ for any $\varepsilon>0,$ any weakly null collection $\mathbf{u}=(u_E)_{E\in [\nn]^{<\omega}\setminus\{\varnothing\}}$ in $B_U$ and any $L\in [\nn],$ there exists $M\in[ L]$  such that $$\Vert A\mathbb{E}^{(2)}_{\xi,\zeta+\eta,M}\mathbf{u}(1)\Vert<\varepsilon.$$

It follows from the proof of Lemma \ref{diagonal} that it is sufficient to 
prove by induction on $\eta\in(0,\omega_1)$ that
 for any weakly null collection $\mathbf{u}=(u_E)_{ E\in[\nn]^{<\omega}\setminus\{\varnothing\}}$   in  $B_U,$ each
 $\varepsilon>0,$ each $L\in[\nn],$     there exists $E\in MAX(S_{\zeta+\eta}[S_\xi])\cap [L]^{<\omega}$ such that $$\Vert A\mathbb{E}^{(2)}_{\xi,\zeta+\eta,E}\mathbf{u}\Vert<\varepsilon.$$

We first prove the $\eta=1$ case. Fix $\varepsilon>0, $
 $\mathbf{u}=(u_E)_{ E\in[\nn]^{<\omega}\setminus\{\varnothing\}}$  a weakly null collection in  $B_U$ and  $L\in[\nn].$ Fix $\gamma>\mathbf{g}_{\xi,\zeta}(A).$ Replacing $L$ with a subset thereof and relabeling, we can assume that $\gamma/\sqrt{\min L}<\varepsilon.$
We can choose $M\in[L]$ such that  $$\Vert A\mathbb{E}^{(2)}_{\xi,\zeta,M}\mathbf{u}\Vert_1^w\leqslant\gamma.$$ If $(E_k)_{k=1}^\infty$ denotes the $S_\zeta[S_\xi]$ decomposition of $M$ (Lemma \ref{decomposition}), therefore
 $$E=\bigcup_{k=1}^{\min M}E_k=M\vert_{\xi,\zeta+1,1}.$$ 
 
Then,
\begin{align*}
A\mathbb{E}^{(2)}_{\xi,\zeta+1,E}\mathbf{u}&=\sum_{n=1}^{\min M}\sum_{M\vert_{\xi,\zeta,n-1}\prec F\preceq M\vert_{\xi,\zeta ,n}}q^{\xi,\zeta+1}_Fp^\xi_FAu_F \cr
&=\frac{1}{\sqrt{\min M}}\sum_{n=1}^{\min M}A\mathbb{E}^{(2)}_{\xi, \zeta, M}\mathbf{u}(n).\end{align*}

Consequently,  $$\Vert A\mathbb{E}^{(2)}_{\xi,\zeta+1,E}\mathbf{u}\Vert\leqslant \gamma/\sqrt{\min M}<\ee.$$
This completes the $\eta=1$ case.

Assume $2\leqslant\eta<\omega_1$ and that, for all $\nu<\eta,$ $\mathbf{g}_{\xi,\zeta+\nu}(A)=0.$ Applying the $\eta=1$ case to the ordinals $\xi, \zeta+\nu$ yields the result for $\eta=\nu+1$.\medskip

Assume $\eta$ is a limit ordinal.     Fix $\varepsilon>0,$ fix $\mathbf{u}=(u_E)_{ E\in[\nn]^{<\omega}\setminus\{\varnothing\}}$ a weakly null collection in $ B_U$  and fix $L\in[\nn]$.   Fix $s>\mathbf{s}_\xi(A)$ and note that with Remark \ref{perm3} and Proposition \ref{hereditary}, by replacing $L$ with a subset thereof and relabeling, we can assume  $\|A\mathbb{E}^{(2)}_{\xi, \nu, M}\mathbf{u}(n)\|<s,$  for all $M\in[L],$  all countable ordinals $\nu$ and all integer $n.$

Note that $\zeta+\eta$ is a limit ordinal,  let  $\eta_n\uparrow \zeta+\eta$ be the characteristic sequence in $\zeta+\eta$ which satisfies also $S_{\eta_n+1}\subset S_{\eta_{n+1}}.$ With this sequence 
 \[\mathcal{S}_{\zeta+\eta}=\{\varnothing\}\cup \{E\in[\nn]^{<\omega}\setminus\{\varnothing\}: E\in \mathcal{S}_{\eta_{\min E}+1}\}. \]  
 
By replacing $L$ with a subset, we can assume $$s/\sqrt{\min L}<\ee/2 \ \ \hbox{and}  \ \ \zeta+1<\eta_{\min L}.$$       

\medskip
 Let $F_1$ be the maximal initial segment of $L$ which lies in $S_{\eta_{\min L}}[S_\xi].$  We have $$ \Vert A\mathbb{
E}^{(2)}_{\xi,\eta_{\min L},F_1}\mathbf{u}\Vert<s.$$
Define $\mathbf{v}=(v_E)_{E\in[\nn]^{<\omega}\setminus
\{\varnothing\}}$
by letting $v_E=u_{F_1\cup E}$ if $F_1<E$ and $v_E=0$ otherwise. Note that $\mathbf{v}$ is a weakly null collection in $B_U.$ Since $g_{\xi,\eta_{\min L}}(A)=0,$ there exists $L_1\in[L], F_1<L_1$ such that for all  $n\in\nn$, $$\Vert A\mathbb{
E}^{(2)}_{\xi,\eta_{\min L},L_1}\mathbf{v}(n)\Vert<\varepsilon/2\min L.$$
Let  $(F_n)_{n=2}^\infty$ be the
$S_{\eta_{\min L}}[S_\xi]$-decomposition of $L_1$ and let $M=F_1\cup L_1.$ The sequence  $(F_n)_{n=1}^\infty$ is the
$S_{\eta_{\min L}}[S_\xi]$-decomposition of $M.$
We note that for every integer $n$
$$A\mathbb{E}^{(2)}_{\xi,\eta_{\min L},L_1}\mathbf{v}(n)=
A\mathbb{E}^{(2)}_{\xi,\eta_{\min L},M}\mathbf{u}(n+1)
$$

The sequence $(\min F_1,\ldots,\min F_{\min L})\in MAX(S_1)$ and so 
$$\bigcup_{n=1}^{\min L}F_n\in MAX(S_1[S_{\eta_{\min L}}[S_\xi]])=MAX(S_{\eta_{\min L}+1}[S_\xi])\subset MAX(S_{\zeta+\eta}[S_\xi]).$$

\medskip

 We deduce that the maximal initial segment of $M$ which lies in $\mathcal{S}_{\zeta+\eta}[\mathcal{S}_\xi]$ is $ \bigcup_{n=1}^{\min L} F_n$.    
 Therefore   
 \begin{align*} \|A\mathbb{E}^{(2)}_{\xi, \zeta+\eta, M}\mathbf{u}(1)\| & = \|A\mathbb{E}^{(2)}_{\xi, \eta_{\min L}+1, M}\mathbf{u}(1)\| = \frac{1}{\sqrt{\min L}}\Bigl\|A\sum_{m=1}^{\min L} \mathbb{E}^{(2)}_{\xi, \eta_{\min L}, M}\mathbf{u}(m)\Bigr\|\cr
 & \leqslant \frac{1}{\sqrt{\min L}}\|A\mathbb{E}^{(2)}_{\xi, \eta_{\min L}, M}\mathbf{u}(1)\| + \frac{1}{\sqrt{\min L}}\sum_{m=2}^{\min L}\|A\mathbb{E}^{(2)}_{\xi, \eta_{\min L}, M}\mathbf{u}(m)\| \cr
 &\leqslant \frac{1}{\sqrt{\min L}}\|A\mathbb{E}^{(2)}_{\xi, \eta_{\min L}, M}\mathbf{u}(1)\| +\frac{1}{\sqrt{\min L}} \sum_{m=1}^{\min L-1}\|A\mathbb{E}^{(2)}_{\xi, \eta_{\min L}, M}\mathbf{v}(m)\| \cr
    & < \frac{s}{\sqrt{\min L}} + \frac{1}{\sqrt{\min L}}\sum_{m=1}^{\min L-1}\ee/2\min L <\ee.
\end{align*}
\end{proof}

\section{On the values of $\mathbf{s}_\xi(U)$ and $\mathbf{g}_{\xi,\zeta}( U )$ when $Sz(U)>\omega^{\xi+1}$ for $\xi< \omega_1$}

The main objective of this section is to show that    $\mathbf{s}_\xi(U)=\mathbf{g}_{\xi,\zeta}( U )=\infty$ for certain Banach spaces $U$ satisfying $Sz(U) > \omega^{\xi+1}$ for $\xi< \omega_1.$

\begin{proposition} \label{s0} Fix $\xi<\omega_1$. Let $U$ be an infinite dimensional  separable  Banach space not containing an isomorphic copy of  $\ell_1$. 
If $Sz(U)>\omega^{\xi+1}$, then  $\mathbf{s}_\xi(U)=\mathbf{g}_{\xi,0}(U)=\infty$. 
\end{proposition} 
\begin{proof} Since $U$ is an infinite dimensional separable Banach space not  containing $\ell_1$,  by Theorem \ref{Sz=Iw} we have    $I_w(U)=Sz(U)> \omega^{\xi+1}$. So, according to Proposition \ref{OdellSZ}, there exist $\varepsilon>0$ and a collection  $\mathbf{u}=(u_E)_{ E\in \mathcal{S}_{\xi+1}\setminus\{\varnothing\}}$ in $B_U$ such that:
\begin{enumerate}[(a)]\item for any $E\in \mathcal{S}_{\xi+1}\setminus MAX(\mathcal{S}_{\xi+1})$, the sequence $(u_{E\smallfrown(n)})_{E<n}$ is  weakly null,  
\item for each $E\in \mathcal{S}_{\xi+1},$  $\varepsilon\leqslant\inf\{\Vert x\Vert : x\in \text{co}\{u_F: \varnothing \prec F\preceq E\}
\}$
    \end{enumerate} 

We extend $\mathbf{u}$ to $[\nn]^{<\infty}\setminus\{\varnothing\}$ by  $u_E=0$ for all $E\in [\nn]^{<\omega}\setminus \mathcal{S}_{\xi+1}.$ We continue to note by $\mathbf{u}$ this extension that remains   a weakly null collection in $B_U$.   Fix $n\in\nn$ and let $L\in[\nn]$  be such that $n= \min L$.  
Fix  $M\in[L]$,   thus $m:=\min M\geqslant n.$   Since $M\vert_{\xi,m}\in \mathcal{S}_{\xi+1}$, (the notation $M\vert_{\xi,k}$ is defined in  Remark \ref{notation}), it follows that 
$$\frac{1}{m}\sum_{i=1}^m \mathbb{E}_{\xi,M}\mathbf{u}(i)\in \text{co}\{u_F: \varnothing\prec F\preceq M\vert_{\xi,m} \}.$$
   Therefore $$\ee \leqslant \Bigl\|\frac{1}{m}\sum_{i=1}^m \mathbb{E}_{\xi,M}\mathbf{u}(i)\Bigr\|,$$ consequently $$\|\mathbb{E}_{\xi,M}\mathbf{u}\|_2^w \geqslant \ee m^{1/2}\geqslant \ee n^{1/2}.$$ Since $n$ was arbitrary, $\mathbf{s}_\xi(U)=\infty$.  
It also follows from this that $$\|\mathbb{E}^{(2)}_{\xi,0,M}\mathbf{u}\|_1^w  \geqslant \|\mathbb{E}_{\xi,M}\mathbf{u}\|_2^w \geqslant \ee m^{1/2}\geqslant \ee n^{1/2}.$$ This shows that $\mathbf{g}_{\xi,0}(U)=\infty$ and we are done.   

\end{proof}

In order to prove $\mathbf{g}_{\xi,\zeta}(U)=\infty$ for all $1\leqslant\zeta<\omega_1$ we need to establish three simple lemmas.
\begin{lemma}(i) For any $\xi<\omega_1, 1\leq\zeta<\omega_1,$  any $F\in MAX(S_\zeta[S_\xi]),$ there exist $E_1,\ldots,E_n\in MAX(S_1[S_\xi]),$ $E_1<\ldots<E_n$ such that $F=\cup_{i=1}^nE_i.$

(ii) For any $\xi<\omega_1, 1\leq\zeta<\omega_1,$  any $M\in [\nn],$ if $(F_n)_{n=1}^\infty$ is the $S_\zeta[S_\xi]$-decomposition of $M,$ then there exists a sequence $(E_i)_{i=1}^\infty$ of successive elements in $MAX(S_1[S_\xi])$ and a strictly increasing sequence $(k_n)_{n=1}^\infty$  of integers such that, for each $n,$ $F_n=\cup_{i=k_{n-1}+1}^{k_n}E_i,$ $(k_0=0).$
\end{lemma}\label{l0}
\begin{proof}(i) We work by induction on $\zeta$ with $\xi$ held fixed. The case $\zeta=1$ is tautological. Assume $2\leqslant\zeta<\omega_1$ and the result holds for all $\nu<\zeta.$ 

Assume $\zeta=\nu+1.$ If 
$F\in MAX(S_\zeta[S_\xi])=MAX(S_1[S_\nu[S_\xi]]),$
$m=\min F,$ we can write $F=\cup_{j=1}^mE_j,$ where $E_1,\ldots,E_m\in MAX(S_\nu[S_\xi]),$ $E_1<\ldots<E_m.$ By the inductive hypothesis, for each $1\leqslant j\leqslant m$ we can write $E_j=\cup_{i=1}^{n_j}E_{i,j},$ $E_{1,j},\ldots,E_{n_j,j}\in MAX(S_1[S_\xi])$ and $E_{1,j}<\ldots<E_{n_j,j}.$ Then $F=\cup_{j=1}^m\cup_{i=1}^{n_i}E_{i,j}.$

Assume $\zeta$ is a limit ordinal and $\zeta_n\uparrow\zeta$ is the sequence used to define $S_\zeta.$ For $F\in MAX(S_\zeta[S_\xi]),$ $F\in MAX(S_{\zeta_{\min E}+1})[S_\xi]$ and the existence follows from the inductive hypothesis.

(ii) follows directely from (i). It is obvious that $(E_n)_{n=1}^\infty$ is the $S_1[S_\xi]$-decomposition of $M.$
\end{proof}
\begin{lemma} \label{l1} For any $\xi<\omega_1, 1\leq\zeta<\omega_1,$ any Banach space $U,$ any collection $\mathbf{u}=(u_E)_{E\in [\nn]^{<\omega}\setminus\{\varnothing\}}$ in $U,$ any $M\in [\nn]
,$ there exists a strictly increasing sequence of integers $(k_l)_{l=1}^\infty$ such that, for  any $n\in \nn$
$$\mathbb{E}^{(2)}_{\xi,\zeta,M}\mathbf{u}(n)\in \mathrm{span}\{\mathbb{E}^{(2)}_{\xi,1,M}\mathbf{u}(m):k_{n-1}<m\leqslant k_n\}.
$$
\end{lemma}
\begin{proof}Fix a collection $\mathbf{u}=(u_E)_{E\in [\nn]^{<\omega}\setminus\{\varnothing\}}$ in $U$ and $M\in [\nn].$   We proceed by induction on $\zeta,$ $\xi$ held fixed. Let $(E_l)_{l=1}^\infty$ be the $S_1[S_\xi]$-decomposition of $M.$ We denote $\Delta_i=\cup_{l=1}^iE_l.$ 

The case $\zeta=2.$   Let $(k_l^2)_{l=1}^\infty$ by the sequence introduced in Lemma \ref{l0} correponding to $\zeta=2$ then, for each $n,$ $M\vert_{\xi,2,n}=\Delta_{k_n^2}.$ For any $l\in \nn$ we denote $m_i=\min E_{k_{l-1}+1}.$ Then
\begin{align*}
\mathbb{E}^{(2)}_{\xi,2,M}\mathbf{u}(l)&=\sum_{i=k_{l-1}^2+1}^{k_l^2}\sum_{\Delta_{i-1}\prec G\preceq \Delta_i}q^{\xi,2}_Gp^\xi_Gu_G.\cr\noalign{It follows from the definition of the weights $q^{\xi,2}$ that}
&=\frac{1}{\sqrt{m_l}}\sum_{i=k_{l-1}^2+1}^{k_l^2}\mathbb{E}^{(2)}_{\xi,1,M}\mathbf{u}(i)\cr
&\in \text{span}\{\mathbb{E}^{(2)}_{\xi,1,M}\mathbf{u}(i) : k_{l-1}^2< i\leq k_l^2\}.
\end{align*}

 If $\zeta$ is a limit ordinal, let $\zeta_n\uparrow\zeta$ be the sequence used to define $S_\zeta.$ Then, $$\mathbb{E}^{(2)}_{\xi,\zeta,M}\mathbf{u}(n)=\mathbb{E}^2_{\xi,\zeta_{\min M}+1,M}(n),$$ and the result follows from the $\zeta_{\min M}+1$ case.

For the successor case $\zeta=\nu+1,$ we proceed as in the $\zeta=2$ case. 
\end{proof}

\begin{lemma} \label{l2} If  $\mathbf{u}=(u_E)_{E\in [\nn]^{<\omega}\setminus\{\varnothing\}}$ is a weakly null seminormalized collection in a Banach space $U,$ then, 
 there exists $L\in[\nn]$ such that, 
  for every $M\in [L],$ the sequence 
$(u_{\{M(1),\ldots,M(n)\}})_{n=1}^\infty$ is a 2-basic sequence. For any $\xi<\omega_1$ and $M\in [L],$ the sequence $(\mathbb{E}^{(2)}_{\xi,1,M}\mathbf{u}(l))_{l=1}^\infty$ is a block basis of $(u_{\{M(1),\ldots,M(n)\}})_{n=1}^\infty.$
\end{lemma}
\begin{proof} We follows the proof of \cite[Theorem 1.a.5.]{LT}. Let $(\varepsilon_k)_{k=1}^\infty$ a sequence of positive numbers such that $$\prod_{k=1}^\infty(1+\varepsilon_k)\leqslant2.$$  We construct by induction a strictly increasing sequence $(n_k)_{k=1}^\infty$ of integers. We fix arbitrary $n_1\in \nn.$ Suppose that $n_1,\ldots,n_k$ have been constructed. Since the collection $\mathbf{u}$  is  weakly null 
there exists $n_{k+1}>n_k$ such that for every $\{\Lambda(1),\ldots,\Lambda(l)\}\subset \{n_1,\ldots,n_k\},$ for every $y\in\mathrm{span}\{u_{\{\Lambda(1),\ldots,\Lambda(i)\}}:1\leqslant i\leqslant l\}, $ and every scalar $\alpha,$ $$\Vert y\Vert\leqslant(1+\varepsilon_k)\Vert y+\alpha u_{\{\Lambda(1),\ldots,\Lambda(l),n_{k+1}\}}\Vert.$$  Thus, the set $L=\{n_k:1\leqslant k<\infty\}$ is such that for every $M\in [L]$,   $(u_{\{M(1),\ldots,M(n)\}})_{n=1}^\infty$ is a 2-basic sequence.
\end{proof}
\begin{lemma} \label{l3} Suppose that  $\mathbf{u}=(u_E)_{E\in [\nn]^{<\omega}\setminus\{\varnothing\}}$ is a weakly null seminormalized collection in a Banach space $U$ and  $L$ is as in Lemma \ref{l2}. Then, for any $\xi<\omega_1,$ $2\leqslant\zeta<\omega_1$ and $M\in [L]$ with $2\leqslant m=\min M,$ we have
$$\mathbb{E}^{(2)}_{\xi,\zeta,M}\mathbf{u}(1)=\sum_{l=1}^{k^\zeta_1}a_l(m,\zeta)\mathbb{E}^{(2)}_{\xi,1,M}\mathbf{u}(l).
$$  Moreover $a_2(m,\zeta)>0.$
\end{lemma}
\begin{proof} We work by induction on $\zeta.$ For the case $\zeta=2$ we have
$$\mathbb{E}^{(2)}_{\xi,2,M}\mathbf{u}(1)=\frac{1}{\sqrt{m}}\sum_{i=1}^{m}\mathbb{E}^{(2)}_{\xi,1,M}\mathbf{u}(i).
$$
Thus $a_2(2,m)=1/\sqrt{m}.$ Next,  assume $\zeta>2$ and the result holds for each $\nu<\zeta.$
Suppose, then, that  $\zeta$ is a limit ordinal and let $1\leqslant\zeta_n\uparrow\zeta$ be the sequence used to define $S_\zeta.$ Hence $$\mathbb{E}^{(2)}_{\xi,\zeta,M}\mathbf{u}(1)=\mathbb{E}^{(2)}_{\xi,\zeta_m+1,M}\mathbf{u}(1),$$ and $a_2(\zeta,m)=a_2(\zeta_m+1,m).$
Now suppose that $\zeta=\nu+1$ for some ordinal $\nu$.  Therefore
$$\mathbb{E}^{(2)}_{\xi,\zeta,M}\mathbf{u}(1)=\sum_{\varnothing\prec E\preceq M\vert_{\xi,\zeta,1}}q^{\xi,\zeta}_Ep^\xi_Eu_E=\frac{1}{\sqrt{m}}\sum_{i=1}^m\mathbb{E}^{(2)}_{\xi,\nu,M}\mathbf{u}(i)=\frac{1}{\sqrt{m}}\sum_{l=1}^{k^\nu_1}a_l(\nu,m)\mathbb{E}^{(2)}_{\xi,1,M}\mathbf{u}(l)+\cdots
$$ 
Consequently $a_2(\zeta,m)=a_2(\nu,m)/\sqrt{m}.$
 \end{proof}
\begin{proposition} \label{r0} Fix $\xi<\omega_1$. Let $U$ be an infinite dimensional  separable  Banach space not containing $\ell_1$. If 
$Sz(U)>\omega^{\xi+1}$,  then $\mathbf{g}_{\xi,\zeta}(U)=\infty$ for all $1\leqslant\zeta<\omega_1.$
\end{proposition}
\begin{proof}
There exit $\varepsilon>0$ and a weakly null collection 
$\mathbf{u}=(u_E)_{E\in S_{\xi+1}\setminus\{\varnothing\}}$ in $B_U$ such that
$$\varepsilon\leqslant\inf\{\Vert u\Vert : E\in \mathrm{co}\{u_F : \varnothing\prec F\preceq E\}\}.
$$
We fill out the collection $(u_E)_{E\in S_{\xi+1}\setminus\{\varnothing\}}$ to a weakly null colection $(u_E)_{E\in [\nn]^{<\omega}\setminus\{\varnothing\}}.$ For 
$E\in [\nn]^{<\omega}\setminus\{\varnothing\}$ there exits a unique decomposition $E=\bigcup_{i=1}^r$ for some $r\geqslant1$ and $E_1<\ldots<E_r,$ $E_i\in MAX(S_{\xi+1})$ for $1\leqslant i<r$ and $E_r\in S_{\xi+1}\setminus\{\varnothing\}$ and we define $u_E=u_{E_r}.$

Fix $2\leqslant\zeta<\omega_1.$ We prove that   $\mathbf{g}_{\xi,\zeta}(U)=\infty.$ 
To that end, by the infinite  Ramsey theorem, Proposition \ref{hereditary}, it is sufficient to show that for any real number $C$ there exists $L\in[\nn],$ such that for any $M\in [L]$ there exists $N\in [M]$  such that $\mathbb{E}^{(2)}_{\xi,\zeta,N}\mathbf{u}(1)>C.$ Let $L$  be as in Lemma \ref{l2}. We also assume $\min L\geqslant2.$ Fix $M\in[L]$  and let $E$ be the maximal initial subset of $M$ which member of $S_1[S_\xi],$  $m=\min M.$ Let $P\in[M]$ such that $E<P$ and $\varepsilon a(\zeta,m)\sqrt{\min P}>4C$ and let $N=E\cup P.$ $E$ is also the maximal initial segment of $N$ which is member of $S_1[S_\xi].$ Let $n=\min P$ and let $F$ be the maximal initial segment of $P$ which is member of $S_1[S_\xi].$  
 Then
$$
\mathbb{E}^{(2)}_{\xi,1,N}\mathbf{u}(1)=\frac{1}{\sqrt{ m}}\sum_{i=1}^{m}\mathbb{E}_{\xi,N}\mathbf{u}(i),$$
and
$$\mathbb{E}^{(2)}_{\xi,1,N}\mathbf{u}(2)=\frac{1}{\sqrt{ n}}\sum_{i=m+1}^{m+n}\mathbb{E}_{\xi,N}\mathbf{u}(i).$$
But since
$$\frac{1}{n}\sum_{i=m+1}^{m+n}\mathbb{E}_{\xi,N}\mathbf{u}(i)\in \mathrm{co}\{u_G : \varnothing\prec G\preceq F\}, 
$$
we deduce $$\Vert\mathbb{E}^{(2)}_{\xi,1,N}\mathbf{u}(2)\Vert\geqslant\varepsilon\sqrt{n}.$$
Next, since  $(\mathbb{E}^{(2)}_{\xi,1,N}\mathbf{u}(l))_{l=1}^\infty$ is a  2-basic sequence we conclude
$$\Vert\mathbb{E}^{(2)}_{\xi,\zeta,N}\mathbf{u}(1)\Vert\geqslant\Vert(a_2(\zeta,m)\mathbb{E}^{(2)}_{\xi,1,N}\mathbf{u}(2)\Vert/4\geqslant\varepsilon a_2(\zeta,m)\sqrt{n}/4>C.
$$

From this and from the proof of the Proposition \ref{die},  it will also follow that $\mathbf{g}_{\xi,1}(U)=\infty,$ and $\mathbf{g}_{\xi,0}(U)=\infty$ follows directly from the proof of the Proposition \ref{s0}. 
 \end{proof}

\section{On the values of $\mathbf{s}_\xi(U)$ and $\mathbf{g}_{\xi,\zeta}( U )$ when $Sz(U) \leq \omega^{\xi}$ for $\xi< \omega_1$}

The main objective of this section is to show that    $\mathbf{s}_\xi(U)=\mathbf{g}_{\xi,\zeta}( U )=0$ for certain Banach spaces $U$ satisfying $Sz(U) \leq \omega^{\xi}$ for $\xi< \omega_1$ (Proposition \ref{s}). Before proving  this we need to establish the Lemma \ref{strong}. This lemma shows that while weak nullity is a condition only on sequences of immediate successors, we can show that in Banach spaces with separable dual, these trees have subtrees which have a stronger property. 

\begin{lemma} Let $U$ be a Banach space such that $U^*$ is separable.  If $\mathbf{u}=(u_E)_{ E\in [\nn]^{<\omega}\setminus\{\varnothing\}}$  is a weakly null collection in   $B_U$, then for any $L\in[\nn]$, there exists $M\in [L]$ such that if $F\in [M]^{<\omega}$ and $F_1, F_2,\ldots \in [M]^{<\omega}\setminus\{\varnothing\}$ are such that $F<F_n$ for all $n\in\nn$ and  $\lim_n \max F_n=\infty$,  then the sequence $(u_{F\smallfrown F_n})_{n=1}^\infty$ is weakly null.     
\label{strong}
\end{lemma}

\begin{proof} Let $(u^*_n)_{n=1}^\infty$  be a dense sequence  in  $ U^*$ and define $d:U\to \rr$ by $$d(u)=\sum_{n=1}^\infty \frac{|\langle u^*_n, u\rangle|}{1+2^n\|u^*_n\|}.$$   Note that a bounded sequence $(u_n)_{n=1}^\infty$ in $U$ is weakly null if and only if $\lim_n d(u_n)=0$.    
Fix $L\in[\nn]$ and a weakly null collection $\mathbf{u}=(u_E)_{E \in [\nn]^{<\omega}\setminus\{\varnothing\}}$ in   $B_U$.    
 Since $(u_{(m)})_{m=1}^\infty$ is weakly null, there exists $m_1\in  L $ so large that $$d(u_{(m_1)})<1/2.$$
 Then we construct by induction on the integer $n$ a strictly increasing sequence $(m_n)_{n=2}^\infty$ in  $L$ such that for every integer $n$ and  every subset $E\subset\{m_1,\cdots,m_n\},$ $$d(u_{E\smallfrown(m_{n+1})})<1/2^{n+1}.$$ Let $M=\{m_n: n=1,2,\cdots\}.$ Let $F,F_1,F_2,\cdots$ be finite subsets of $M$ such that $F<F_n,$  $F_n\ne \varnothing$ for all $n\in\nn$ and $\lim_n \max F_n=\infty.$ Next, we show that $\lim_n d(u_{F\smallfrown F_n})=0.$
 Let  $\varepsilon>0.$ Fix an integer $N$ such that $1/2^N<\varepsilon$ and an integer $n_0$ such that for any $n\geqslant n_0,$ $\max F_n>m_N.$ 
 Fix $n\geqslant n_0,$ then $F\smallfrown F_n$ is the strictly increasing sequence $(m_{i_l})_{l=1}^k.$   We have $m_{i_k}=\max F_n>m_N$.  
 So $$d(u_{F\smallfrown F_n})=d(u_{(m_{i_1},\cdots,m_{i_{k-1}},m_{i_k})})<1/2^{i_k}<1/2^N,$$ and we are done.
\end{proof}
\begin{proposition} \label{s} Fix $\xi<\omega_1$. Let $U$ be an infinite dimensional  separable  Banach space not containing an isomorphic copy of $\ell_1$. 
 \begin{enumerate}[(i)] \item $Sz(U)\leqslant \omega^\xi$ if, and only, if for any $\ee>0$,  any weakly null collection $\mathbf{u}=(u_E)_{ E\in [\nn]^{<\omega}\setminus\{\varnothing\}}$ in $ B_U$, and any $L\in[\nn]$, there exists $E\in MAX(S_\xi)\cap[L]^{<\infty}$ such that $\|\mathbb{E}_{\xi,E}\mathbf{u}\|< \ee$.   

 \item  $Sz(U)\leqslant \omega^\xi$ if and only if $\mathbf{s}_\xi(U)=0$, if and only if $\mathbf{g}_{\xi,0}(U)=0$, if and only if for every $\zeta<\omega_1$,  $\mathbf{g}_{\xi,\zeta}(U)=0$.  \end{enumerate}
\label{szlenk}
\end{proposition}

\begin{proof} $(i)$ We suppose $Sz(U)>\omega^\xi.$ As noted in the proof of Proposition \ref{s0}, there exist $\ee>0$ and a collection $\mathbf{u}=(u_E)_{ E\in \mathcal{S}_{\xi}\setminus\{\varnothing\}}$ in $ B_U$ such that \begin{enumerate}[(a)]
\item for any $E\in \mathcal{S}_{\xi}\setminus MAX(\mathcal{S}_{\xi})$, $(u_{E\smallfrown(n)})_{E<n}$ is a weakly null sequences, 
\item for each $E\in \mathcal{S}_{\xi},$ $\varepsilon\leqslant\inf\{\Vert x\Vert : x\in \text{co}(u_F: \varnothing \prec F\preceq E)\}.$     \end{enumerate}  

We extend $\mathbf{u}$ to $[\nn]^{<\infty}\setminus\{\varnothing\}$ by  $u_E=0$ for all $E\in [\nn]^{<\omega}\setminus \mathcal{S}_{\xi}.$ We continue to denote by $\mathbf{u}$ this extension that remains   a weakly null collection in $B_U$.

 For any $L\in [\nn]$ and $E\in MAX(\mathcal{S}_\xi)\cap [L]^{<\omega},$ it follows from (b) that $$\mathbb{E}_{\xi,E}\mathbf{u}\in \text{co}(u_F:\varnothing\prec F\preceq M_{\xi,1}),$$ and so $\|\mathbb{E}_{\xi,E}\mathbf{u}\|\geqslant \ee$.   By contraposition, if we suppose that for every $\ee>0$, every weakly null collection $\mathbf{u}=(u_E)_{ E\in[\nn]^{<\omega}\setminus\{\varnothing\}}$ in $ B_U$, and every $L\in[\nn]$, there exists $E\in MAX(S_\xi)\cap[L]^{<\infty}$ such that $\|\mathbb{E}_{\xi,E}\mathbf{u}\|<\ee$, then it holds that $Sz(U)\leqslant \omega^\xi$.

For the converse, assume $Sz(U)\leqslant \omega^\xi$ and that there exist $\ee>0$, a weakly null collection $\mathbf{u}=(u_E)_{ E\in [\nn]^{<\omega}\setminus\{\varnothing\}}$ in $B_U$, and $L\in[\nn]$ such that for all $E\in MAX(S_\xi)\cap[L]^{<\infty}$, $$\|\mathbb{E}_{\xi,E}\mathbf{u}\|\geqslant \ee.$$  We will show that this assumption leads to a contradiction.

Since $U$ is separable and $Sz(U)$ is countable,  it follows by  Theorem \ref{Sz=Iw}.iii  that $U^*$ is separable. Therefore by Lemma \ref{strong}, by replacing $L$ with a subset thereof and relabeling, we can assume that if $F, F_1, F_2,\ldots \in [L]^{<\omega}$ are such that $F<F_n\neq \varnothing$ for all $n\in\nn$ and $\lim_n \max F_n=\infty$, then $(u_{F\smallfrown F_n})_{n=1}^\infty$ is weakly null. 
For each $E\in MAX(S_\xi)$, fix $u^*_E\in B_{U^*}$ such that $$\text{Re\ }\langle u^*_E, \mathbb{E}_{\xi,E}\mathbf{u}\rangle = \|\mathbb{E}_{\xi,E}\mathbf{u}\|.$$  

    Define $f:\nn\times MAX(\mathcal{S}_\xi)\to [-1,1]$ by \[f(n,E) = \left\{\begin{array}{ll} \text{Re\ }\langle u^*_{E}, u_{[1,n]\cap E}\rangle & : n\in E \\ 0 & : n\in \nn\setminus E. \end{array}\right.\] 
  Then for any $E\in MAX(S_\xi)\cap[L]^{<\infty}$,$$ \sum_{ \varnothing\prec F\preceq E} p_F^\xi f(\max F, E)  = \sum_{\varnothing\prec F\preceq E} \text{Re}\langle u^*_E, p^\xi_F u_F\rangle = \text{Re\ }\langle u^*_{E}, \mathbb{E}_{\xi, E}\mathbf{u}\rangle \geqslant \ee. $$

This yields that, in the notation of Section \ref{sufficient}, $[L]\subset \mathfrak{E}(f,\mathfrak{P}, \mathcal{P}, \ee)$. By Lemmas \ref{resultado1} and \ref{resultado2},  the probability block $(\mathfrak{S}_\xi, \mathcal{S}_\xi)$  is $\xi$-regulatory. Then,  for any $0<\ee_1<\ee$, $\mathfrak{G}(f,\mathcal{P}, L,\ee_1)$ is $\xi$-full. This means there exists $N\in [L]$ such that $\mathcal{S}_\xi(N)\subset \mathfrak{G}(f,\mathcal{P}, L,\ee_1)$. In other words, for any $F\in \mathcal{S}_\xi$, $N(F)\in \mathfrak{G}(f,\mathcal{P}, L,\ee_1)$. But this means exactly that  for each $F\in \mathcal{S}_\xi$, there exists  $E\in MAX(\mathcal{S}_\xi)\cap[L]^{<\omega}$ such that for all $i\in F$, 
$$ \text{Re\ }\langle u^*_{E} , u_{[1,N(i)]\cap E}\rangle=f(N(i), E) \geqslant \ee_1.$$

For a Banach space and $\varepsilon\in \mathbb{R},$ the family $(s_\varepsilon^{(\alpha)}(B_{U^*}))_{\alpha<\omega_1}$ is introduced in section \ref{szlenk index}.

 Fix $0<\ee_2<\ee_1$.    We  claim that for any $0\leqslant \zeta\leqslant \omega^\xi$ and for any $F\in MAX(\mathcal{S}_\xi^{(\zeta)})$, there exists  $u^*\in s_{\ee_2}^{(\zeta)}(B_{U^*})$ and, if $F\neq \varnothing$, there also exists $E\in \mathcal{S}_\xi$ such that for each $i\in F$, $$N(i)\in E \ \ \hbox{and} \ \ \ee_1 \leqslant \text{Re\ }\langle u^*, u_{[1, N(i)]\cap E}\rangle.$$    We prove this by induction on $\zeta$.

 The $\zeta=0$ case is precisely the last two sentences of the preceding paragraph with $u^*=u^*_{E}$, using the fact that
 $s_{\ee_2}^{(0)}(B_{U^*})=B_{U^*}$.

The case  $\zeta$ is a limit ordinal and assume the conclusion holds for each $\mu<\zeta$.   Fix $F\in MAX(\mathcal{S}_\xi^{(\zeta)})$.

If $\zeta=\omega^\xi$, then $F=\varnothing$, and we need to show that $s_{\ee_2}^{(\omega^\xi)}(B_{U^*})\neq \varnothing$. In this case, by the inductive hypothesis, $s_{\ee_2}^{(\mu)}(B_{U^*})\neq \varnothing$ for each  $\mu<\omega^\xi$. From this and weak$^*$-compactness,  $s_{\ee_2}^{(\omega^\xi)}(B_{U^*})\neq \varnothing$.

  If $\zeta<\omega^\xi$, then $F\neq \varnothing$. In this case, for each $\mu<\zeta$, there exists $F_\mu\in MAX(\mathcal{S}_\xi^{(\mu)})$ such that $F\preceq F_\mu$.   By the inductive hypothesis, there exist $u^*_\mu\in s_{\ee_2}^{(\mu)}(B_{U^*})$ and $E_\mu\in \mathcal{S}_\xi$ such that for each $i\in F$, $$N(i)\in E_\mu \ \ \hbox{and} \ \ \text{Re\ }\langle u^*_\mu, u_{[1, N(i)]\cap E_\mu}\rangle \geqslant \ee_1.$$  Fix $\mu_n\uparrow \zeta$ and note that since  $B_{U^*}\times \mathcal{S}_\xi$ is compact metrizable, then by passing to a subsequence and relabeling, we can assume $(u^*_{\mu_n}, E_{\mu_n})_{n=1}^\infty$ is convergent to some $(u^*, E)\in B_{U^*}\times \mathcal{S}_\xi$.   By passing to a subsequence again, we can assume that $$[1, N(\max F)]\cap E_{\mu_n} = [1, N(\max F)]\cap E,$$ for all $n\in\nn$.   For each $i\in F$ and for all $n\in\nn,$    $N(i)\in E_{\mu_n}$ so $N(i)\in E.$ It follows that $$u_{[1, N(i)]\cap E}=u_{[1, N(i)]\cap E_{\mu_n}},$$ for all $n\in\nn$. 
Since $(u^*_{\mu_n})_{n=1}^\infty$ is weak$^*$-convergent to $u^*$, it holds that,  for each $i\in F,$ \[\text{Re\ }\langle u^*, u_{[1, N(i)]\cap E}\rangle = \lim_n \text{Re\ }\langle u^*_{\mu_n},  u_{[1, N(i)]\cap E}\rangle=\lim_n \text{Re\ }\langle u^*_{\mu_n},  u_{[1, N(i)]\cap E_{\mu_n}}\rangle \geqslant \ee_1.\]    
Since $u^*_{\mu_n}\in s_{\ee_2}^{(\mu_n)}(B_{U^*})$, it follows that,   $$u^*\in \bigcap_{n=1}^\infty s_{\ee_2}^{(\mu_n)}(B_{U^*})=s^{(\zeta)}_{\ee_2}(B_{U^*}).$$   This completes the limit ordinal case of the proof.

Next, assume that the result holds for some $\zeta<\omega^\xi$ and assume $F\in MAX(\mathcal{S}_\xi^{(\zeta+1)})$.   In this case, $F\smallfrown (n)\in MAX(\mathcal{S}_\xi^{(\zeta)})$ for all $n\in\nn,$ $n>F.$ By the inductive hypothesis, for each $n>F$, there exists $u^*_n \in s_{\ee_2}^{(\zeta)}(B_{U^*})$ and $E_n\in \mathcal{S}_\xi$ such that for each $i\in F\smallfrown (n)$, $$N(i)\in E_n \ \ \hbox{and} \ \  \text{Re\ }\langle u^*_n, u_{[1, N(i)]\cap E_n}\rangle \geqslant \ee_1.$$    We can select an infinite subset $R$ of $(\max F, \infty)$ such that $(u^*_n)_{n\in R}$ is weak$^*$-convergent to some $u^*\in s_{\ee_2}^{(\zeta)}(B_{U^*})$ and $E\in \mathcal{S}_\xi$ such that $([1, N(n)]\cap E_n)_{n\in R}$ converges to $E$. By passing to a further subsequence and relabeling, we can assume that for each $i\in F$ and $n\in R$, $$[1, N(i)]\cap E_n= [1, N(i)]\cap E.$$  Since $N(i)\in E_n$ for all $i\in F$ and $n\in R$, $N(i)\in  E$. From this it follows that $N(i)\in E$ for all $i\in F$.   Moreover, since $u_{[1, N(i)]\cap E_n}= u_{[1,N(i)]\cap E}$ for all $i\in F$ and $n\in R$ and since $(u^*_n)_{n\in R}$ is weak$^*$-convergent to $u^*$, it holds that 
\[\text{Re\ }\langle u^*, u_{[1, N(i)]\cap E}\rangle=\lim_n \text{Re\ }\langle u^*_n, u_{[1, N(i)]\cap E}\rangle  = \lim_n \text{Re\ }\langle u^*_n, u_{[1, N(i)]\cap E_n}\rangle \geqslant \ee_1.\]  
  
It remains to show that $u^*\in s^{(\zeta+1)}_{\ee_2}$.  We note that since $([1, N(n)]\cap E_n)_{n\in R}$ is convergent to $E$ and since $n\leqslant N(n)\in [1,N(n)]\cap E_n$ for all $n\in R$, by replacing $R$ with a subset thereof and relabeling if necessary, we can assume that for each $n\in R$, $$[1, N(n)]\cap E_n = E\smallfrown F_n,$$ for some $E<F_n\neq \varnothing$ such that $\lim_{n\in R} \max F_n=\infty$.    By our choice of $L$, $$(u_{[1, N(n)]\cap E_n})_{n\in R} = (u_{E\smallfrown F_n})_{n\in R}$$ is weakly null.   From this it follows that \begin{align*} \underset{n\in R}{\lim\inf} \|u^*-u^*_n\| & \geqslant \underset{n\in R}{\lim\inf}\ \text{Re\ }\langle u^*_n-u^*, u_{[1, N(n)]\cap E_n}\rangle \geqslant \ee_1-0>\ee_2. \end{align*}  

Combining this with the fact that $(u^*_n)_{n\in R}$ is a sequence in $ s_{\ee_2}^{(\zeta)}(B_{U^*})$, we deduce that $$u^*\in s_{\ee_2}^{(\zeta+1)}(B_{U^*}).$$   This concludes the proof by induction. Applying the $\zeta=\omega^\xi$ case yields that $s^{(\omega^\xi)}_{\ee_2}\neq \varnothing$, and $Sz(U)>\omega^\xi$. This contradiction finishes $(ii)$.

$(ii)$ If $Sz(U)\leqslant \omega^\xi$, then by $(ii)$, for any $\ee>0$, any weakly null collection $\mathbf{u}=(u_E)_{ E\in [\nn]^{<\omega}\setminus\{\varnothing\}}$ in $ B_U$, and any $L\in[\nn]$, there exists $E\in MAX(S_\xi)\cap [L]^{<\omega}$ such that $\|\mathbb{E}_{\xi,E}\mathbf{u}\|<\ee$.    Since $$\mathbb{E}_{\xi,E}\mathbf{u}= \mathbb{E}^{(2)}_{\xi,0,E}\mathbf{u},$$ it follows from Lemma \ref{diagonal} that $\mathbf{g}_{\xi,0}(U)=0$.

Next, suppose that $\mathbf{g}_{\xi,0}(U)=0$.   By Proposition \ref{die}, $\mathbf{g}_{\xi,\zeta}(U)=0$ for all $0<\zeta<\omega_1$, and therefore for all $0\leqslant \zeta<\omega_1$.

Next, suppose that $\mathbf{g}_{\xi,\zeta}(U)=0$ for all $\zeta<\omega_1$.    Then since $\mathbf{g}_{\xi,0}(U)=0$, it holds that for any  $\gamma>0$, any weakly null collection $\mathbf{u}=(u_E)_{ E\in [\nn]^{<\omega}\setminus\{\varnothing\}}$ in $ B_U$, and any $L\in[\nn]$, there exists $M\in[L]$ such that $$\|\mathbb{E}^{(2)}_{\xi,0,M}\mathbf{u}\|_1^w \leqslant \gamma.$$   Since $$\mathbb{E}^{(2)}_{\xi,0,M}\mathbf{u}= \mathbb{E}_{\xi,M}\mathbf{u},$$ it holds that $$\|\mathbb{E}_{\xi,M}\mathbf{u}\|_2^w \leqslant \|\mathbb{E}^{(2)}_{\xi,0,M}\mathbf{u}\|_1^w \leqslant \gamma.$$ This yields that $\mathbf{s}_\xi(U)=0$.

Suppose that $\mathbf{s}_\xi(U)=0$.   Then for any $\ee>0$, any weakly null collection $\mathbf{u}=(u_E)_{ E\in [\nn]^{<\omega}\{\varnothing\}}$ in $ B_U$, and any $L\in[\nn]$, there exists $M\in[\nn]$ such that $$\|\mathbb{E}_{\xi,M}\mathbf{u}\|_2^w<\ee.$$  Since $$\|\mathbb{E}_{\xi,M}\mathbf{u}(1)\|\leqslant \|\mathbb{E}_{\xi,M}\mathbf{u}\|_2^w < \ee,$$ the item $(i)$ of the proposition yields that $Sz(U)\leqslant \omega^\xi$. 
\end{proof}

\section{A fundamental result on $\mathbf{s}_\xi$ and $\mathbf{g}_{\xi,1+\zeta}$ via the Grothendieck's constant $k_G$}

The main purpose of this and the next two sections is to obtain upper bounds for $\mathbf{s}_\xi(U)$ and $\mathbf{g}_{\xi, 1+\zeta}(U)$, where $U=C_0(\omega^{\omega^\xi})\tim_\pi C_0(\omega^{\omega^\zeta})$ with  $\zeta \leq \xi< \omega_1$. First of all notice that the space $U^*$ is isomorphic to $\ell_1\tim_\varepsilon \ell_1$, so $U$ does not contain an isomorphic  copy of $\ell_1$. We start by introducing a proposition $h$ that will help us with these goals. 

\begin{definition} Let $H$ denote the set of all  ordered pairs  $(\xi, \zeta)$ of countable ordinals such that $\zeta\leqslant \xi$.   Let $k_G$ be the  Grothendieck's constant mentioned in Proposition \ref{groth}.  For each $(\xi, \zeta)\in H$, let $h(\xi,\zeta)$ be the proposition 
\begin{displaymath}
   h(\xi, \zeta) = \left\{
     \begin{array}{ll}
       \max\{\mathbf{s}_\xi(C_0(\omega^{\omega^\xi})\tim_\pi C_0(\omega^{\omega^\zeta})), \mathbf{g}_{\xi, 1+\zeta}(C_0(\omega^{\omega^\xi})\tim_\pi C_0(\omega^{\omega^\zeta}))\}  \leqslant k_G & :  \zeta<\xi\\
			\max\{\mathbf{s}_\xi(C_0(\omega^{\omega^\xi})\tim_\pi C_0(\omega^{\omega^\xi})), \mathbf{g}_{\xi, 1+\zeta}(C_0(\omega^{\omega^\xi})\tim_\pi C_0(\omega^{\omega^\xi}))\} \leqslant 2k_G	&	:	 \zeta=\xi.
     \end{array}
   \right.
\end{displaymath} 

For $ \xi<\omega_1$,  let $h(\xi)$ be the proposition that $h(\xi, \zeta)$ holds for all $\zeta\leqslant \xi$.   Let $h$ be the proposition that $h(\xi)$ holds for all $\xi<\omega_1$.    
\end{definition}

So,  we will prove the following crucial result on the functionals  $\mathbf{s}_\xi$ and $\mathbf{g}_{\xi, 1+\zeta}$ with $\zeta \leq \xi< \omega_1$.

\begin{theorem} The proposition $h$ holds. 

\label{positive one}
\end{theorem}

The proof of Theorem \ref{positive one} in the case $\xi=0$   follows immediately from  Theorem \ref{JFA}  and Proposition \ref{s}. We will complete the proof of Theorem  \ref{positive one} by induction on $\xi < \omega_1$ and will do it in the next two sections. Here we will only prepare this proof by establishing three lemmas related to the definition of proposition $h$.

From now on, for a compact, Hausdorff space $K$ and $f\in C(K)$, we let \[\supp(f)=\overline{\{\varpi\in K: f(\varpi)\neq 0\}}.\]    For convenience, we define $(-1, \beta]=[0,\beta]$ for any ordinal $\beta$.

\begin{lemma}For $\xi<\omega_1$ and $\beta<\omega$, $\mathbf{g}_{\xi, 0}(C_0(\omega^{\omega^\xi})\tim_\pi C(\beta))<\infty$. 
\label{knight}\end{lemma}
\begin{proof}
The spaces $C_0(\omega^{\omega^\xi})\tim_\pi C(\beta)$ and $C_0(\omega^{\omega^\xi})$ are isomorphic, so   by Corollary \ref{finite}  it is sufficient to prove that $\mathbf{g}_{\xi,0}(C_0(\omega^{\omega^\xi}))<\infty.$ In fact, we will show that $\mathbf{g}_{\xi,0}(C_0(\omega^{\omega^\xi}))\leqslant 1.$

Let $\mathbf{u}=(u_E)_{ E\in[\nn]^{<\omega}\setminus\{\varnothing\}}$ be a weakly null collection in $ B_{C_0(\omega^{\omega^\xi})}.$  Fix $L\in[\nn] $ and $\ee>0$. We construct by induction a sequence
 $(E_n)_{n=1}^\infty$ of successive elements in $MAX(S_\xi)\cap [L]^{<\omega},$ 
a strictly increasing sequence $(\beta_n)_{n=0}^\infty$ in $(-1,\omega^{\omega^\xi}),$ ($\beta_0=-1$)
and a sequence $(g_n)_{n=1}^\infty$ in $B_{C_0(\omega^{\omega^\xi})}$ such that $\text{supp}(g_n)\subset (\beta_{n-1},\beta_n]$ and  if we denote $\Delta_n=\bigcup_{i=1}^nE_i$ then
$$\left\Vert g_n-\sum_{\Delta_{n-1}\prec E\preceq\Delta_n}p^\xi_Eu_E\right\Vert<\varepsilon/2^n.
$$

Let $E_1\in MAX(S_\xi)\cap[L]^{<\omega}.$ We can select $\beta_1<\omega^{\omega^\xi}$ such that, for each $\gamma>\beta_1,$
 $$ 
\vert \mathbb{E}_{\xi,E_1}\mathbf{u}(\gamma)\vert=\left\vert \sum_{\varnothing\prec E\preceq E_1}p^\xi_E
u_E(\gamma)\right\vert <\varepsilon/2.$$  We denote by $g_1$ the function   that coincides with $\mathbb{E}_{\xi,E_1}\mathbf{u}$ on $[0,\beta_1]$ and equals to 0 on $(\beta_1,\omega^{\omega^\xi}).$

Next, assume that $\beta_1<\ldots <\beta_n$, $E_1,\ldots,<E_n\in MAX(S_\xi)\bigcap [L]^\omega,$ $E_1<\ldots<E_n$ have been chosen.

We define $\mathbf{v}=(v_E)_{ E\in[\nn]^{<\omega}\setminus\{\varnothing\}}$ a weakly null collection in $B_{C(\beta_n)}$ by letting $ v_E= u_{\Delta_n\cup E}\vert_{[0, \beta_n]}$ if $\Delta_n<E$ and $v_E=0$ otherwise.   Let $L_{n+1}=L\cap (\max \Delta_n,\infty).$ Since $Sz(C(\beta_n)) \leqslant\omega^\xi$, $\mathbf{g}_{\xi,0}(C(\beta_n))=0$ so, by Proposition~\ref{szlenk},    we can select $E_{n+1}\in MAX(S_\xi)\cap [L_{n+1}]^{<\omega}$ such that $$\Vert\mathbb{E}_{\xi,E_{n+1}}\mathbf{v}\Vert=\left\Vert
\sum_{\Delta_n\prec E\preceq \Delta_{n+1}}p^\xi_E{u_E}_{\vert[0,\beta_n]}
\right\Vert
<\varepsilon/2^{n+1}.$$ Select $\beta_{n+1}\in (\beta_n,\omega^{\omega^\xi})$ such that  $$\left\vert
\sum_{\Delta_n\prec E\preceq \Delta_{n+1}}p^\xi_Eu_E(\gamma)
\right\vert<\varepsilon/2^{n+1},$$ for each $\gamma>\beta_{n+1}$.  We denote by    $g_{n+1}\in B_{C_0(\omega^{\omega^\xi})}$ the function such that $g_{n+1}=0$ on $[0,\beta_n]\cup (\beta_{n+1}, \omega^{\omega^\xi})$ and 
$g_{n+1}=\sum_{\Delta_n\prec E\preceq \Delta_{n+1}}p^\xi_Eu_E$ on $(\beta_n,\beta_{n+1}].$
This completes the recursive construction. We let $M=\cup_{n-1}^\infty E_n.$  We note that each integer $n,$ $\mathbb{E}_{\xi,M}\mathbf{u}(n)=\sum_{\Delta_{n-1}\prec E\preceq \Delta_{n}}p^\xi_Eu_E$ and so $\Vert g_n-\mathbb{E}_{\xi,M}\mathbf{u}(n)\Vert<\varepsilon/2^n.$  We note  also that for $n=1,2,\ldots$ the functions $g_n\in B_{C_0(\omega^{\omega^\xi})}$ have pairwise disjoint supports so $\Vert(g_n)_{n=1}^\infty\Vert_1^w\leqslant 1,$ it follows that
$$\Vert \mathbb{E}_{\xi,M}\mathbf{u}\Vert_1^w\leqslant \Vert(g_n)_{n=1}^\infty\Vert_1^w+\sum_{n=1}^\infty\Vert g_n-\mathbb{E}_{\xi,M}\mathbf{u}(n)\Vert<1+\varepsilon.
$$
From this, it follows that $\mathbf{g}_{\xi,0}(C_0(\omega^{\omega^\xi}))\leqslant1.$
\end{proof}

We now isolate the following consequence of the hypotheses $h(\xi)$, $\xi<\omega_1$.

\begin{lemma}\label{claim1} Suppose that $\xi<\omega_1$ is such that $h(\eta)$ holds for all $\eta<\xi$.    Then for any $\alpha, \beta<\omega^{\omega^\xi}$, $Sz(C(\alpha)\tim_\pi C(\beta))\leqslant \omega^\xi$.
\end{lemma}  

\begin{proof} Without loss of generality, assume $\beta\leqslant \alpha$. Since $C( \beta)\widehat{\otimes}_\pi C(\alpha)$ is isomorphic to a subspace of of $C(\alpha)\widehat{\otimes}_\pi C(\alpha)$ and 
$Sz(C( \beta)\widehat{\otimes}_\pi C(\alpha))\leq Sz(C(\alpha)\widehat{\otimes}_\pi C(\alpha)),$
it is sufficient to show that $Sz(C(\alpha)\widehat{\otimes}_\pi C(\alpha))\leqslant \omega^\xi$ for each $\alpha < \omega^{\omega^\xi}$.

If $\xi=0$, then $\alpha<\omega$ and $C(\alpha)\widehat{\otimes}_\pi C(\alpha)$ is finite dimensional and has Szlenk index $1=\omega^0\leqslant \omega^\xi$.   

 Now we suppose $0<\xi<\omega_1$  and $\omega\leqslant\alpha < \omega^{\omega^\xi},$ then for some $\eta<\xi$, $C(\alpha)$ is  isomorphic to $C_0(\omega^{\omega^\eta})$ \cite[Theorem 2.14]{Rosenthal}. 
Therefore $C(\alpha)\tim_\pi C(\alpha)$ is  isomorphic to   $C_0(\omega^{\omega^\eta})\tim_\pi C_0(\omega^{\omega^\eta})$. 
Since we have assumed   $h(\eta, \eta)$ holds, it follows that  $$\mathbf{s} _\eta (C_0(\omega^{\omega^\eta})\tim_\pi C_0(\omega^{\omega^\eta})) \leqslant 2k_G<\infty.$$     It follows also from the proof of Proposition \ref{die} that    
   $$\mathbf{s} _{\eta+1}(C_0(\omega^{\omega^\eta})\tim_\pi C_0(\omega^{\omega^\eta}))=0.$$   Since the Szlenk index is an isomorphic invariant,  and by  Proposition \ref{szlenk}, $$Sz(C(\beta)\tim_\pi C(\alpha))\leq  Sz(C_0(\omega^{\omega^\eta})\tim_\pi C_0(\omega^{\omega^\eta})) \leqslant \omega^{\eta+1}\leqslant \omega^\xi.$$      
   This yields the claim if $0<\xi<\omega_1$.  
\end{proof}

We next will prove a result analogous to  Lemma \ref{knight}  that will be used in section \ref{positive}  when  we will be working with the functionals $\mathbf{g}_{\xi,\zeta}$.

For an ordinal $\psi<\omega^{\omega^\lambda}$ and $f\in C_0(\omega^{\omega^\lambda})$ we let $R_\psi f=f\vert_{[0,\psi]}$ the restriction operator.  We let $P_\psi$ be the operator from $C_0(\omega^{\omega^\lambda})$ to itself given by letting $P_ \psi f(\theta)=f(\theta) $ for $\theta\leqslant\alpha$ and $P_ \psi f(\theta)=0$ for $\theta>\psi.$
\begin{lemma} \label{claim2} Assume that $\zeta\leqslant \xi<\omega_1$ are such that $h(\xi, \nu)$ holds for all $\nu<\zeta$ and  $h(\eta)$ holds for any $\eta<\xi$,  then    for any $\alpha<\omega^{\omega^\xi}$, any  $\beta<\omega^{\omega^\zeta}$, any $\ee>0$, any weakly null collection  $\mathbf{v}=(v_E)_{ E\in[\nn]^{<\omega}\setminus\{\varnothing\}}$ in $ B_{C_0(\omega^{\omega^\xi})\tim_\pi C_0(\omega^{\omega^\zeta})}$, and any $L\in[\nn]$, there exists $E\in MAX(S_{1+\zeta}[S_\xi])\bigcap[ L]^{<\omega}$ such that \[\|\mathbb{E}^{(2)}_{\xi,1+\zeta,E}\mathbf{v}-((I-P_{\alpha})\otimes (I-P_\beta))\mathbb{E}^{(2)}_{\xi,1+\zeta,E}\mathbf{v}\|<\ee.\] 
\end{lemma}

\begin{proof}   Fix $\alpha<\omega^{\omega^\xi}, \beta<\omega^{\omega^\zeta}.$ Fix $\varepsilon>0,$ a weakly null collection $\mathbf{v}=(v_E)_{ E\in[\nn]^{<\omega}\setminus\{\varnothing\}}$ in $ B_{C_0(\omega^{\omega^\xi})\tim_\pi C_0(\omega^{\omega^\zeta})}$, and   $L\in[\nn].$

Since $\beta< \omega^{\omega^\zeta},$ we claim that there exists $\mu<1+\zeta$ such that $\mathbf{g}_{\xi,\mu}(C_0(\omega^{\omega^\xi})\widehat{\otimes}_\pi C(\beta))<\infty$.

 If $\beta$ is finite, then by Lemma \ref{knight},       $\mathbf{g}_{\xi, 0}(C_0(\omega^{\omega^\xi})\tim_\pi C(\beta))<\infty$. In this case, let $\mu=0<1+\zeta$  is convenient.   If $\omega\leqslant\beta<\omega^{\omega^\zeta},$ then $C(\beta)$ is isomorphic to $C_0(\omega^{\omega^\nu})$ for some $\nu<\zeta$.    In this case, $C_0(\omega^{\omega^\xi})\tim_\pi C(\beta)$ is isomorphic to $C_0(\omega^{\omega^\xi})\tim_\pi C_0(\omega^{\omega^\nu})$. Since $h(\xi, \nu)$ is assumed to hold, it follows that $\mathbf{g}_{\xi,1+\nu}(C_0(\omega^{\omega^\xi})\tim_\pi C_0(\omega^{\omega^{\nu}}))<\infty$, and therefore $\mathbf{g}_{\xi,1+\nu}(C_0(\omega^{\omega^\xi})\tim_\pi C(\beta))<\infty$.     In this case,  $\mu=1+\nu$ is convenient.

 We claim that there exists $M\in [L]$ such that, for each $E\in MAX(S_{1+\zeta}[S_\xi])\cap [M]^{<\omega},$ 
\begin{equation}
\Vert  (R_\alpha\otimes I)\mathbb{E}^{(2)}_{\xi,1+\zeta,E}\mathbf{v}\Vert<\varepsilon/2\text{ and }\Vert(I\otimes R_\beta)\mathbb{E}^{(2)}_{\xi,1+\zeta,E}\mathbf{v}\Vert<\varepsilon/2
.
\end{equation}

$(I\otimes R_\beta)\mathbf{v} = ((I\otimes R_\beta)v_E)_{E\in[\nn]^{<\omega}\setminus\{\varnothing\}} $  is a weakly null collection in $B_{C_0(\omega^{\omega^\xi})\tim_\pi C(\beta)}$, so
it follows from  Proposition \ref{die}, $\mathbf{g}_{\xi,1+\zeta}=0,$  
and there exists $M_1\in[L]$ such that for all $E\in MAX(S_{1+\zeta}[S_\xi])\bigcap[M_1]^{<\omega}$, 
$$\|(I\otimes R_\beta)\mathbb{E}_{\xi,1+\zeta, E}^{(2)}\mathbf{v}\|<\ee/2.$$ 
  $(R_\alpha\otimes I)\mathbf{v}=( (R_\alpha\otimes I)v_E)_{E\in[\nn]^{<\omega}\setminus\{\varnothing\}}$ is a weakly null collection in $B_{C(\alpha)\widehat{\otimes}_\pi C(\omega^{\omega^\zeta})}$

If $\xi=\zeta$, by symmetry, we can apply the same argument to deduce    there exists $M\in[ M_1]$ such that for all $E\in MAX(S_{1+\zeta}[S_\xi])\bigcap[M]^{<\omega}$, $$\|(R_\alpha\otimes I)\mathbb{E}_{\xi,1+\zeta, E}^{(2)}\mathbf{v}\|<\ee/2.$$

If $\zeta<\xi$, then, by Lemma \ref{claim1}, $Sz(C(\alpha)\tim_\pi C_0(\omega^{\omega^\zeta}))\leqslant \omega^\xi$  which implies that $\mathbf{g}_{\xi,0}(C(\alpha)\tim_\pi C_0(\omega^{\omega^\zeta}))=0$.  By Proposition \ref{die}, there exists $M\in[M_1]$ such that for all $E\in MAX(S_{1+\zeta}[S_\xi])\bigcap[M]^{<\omega}$, $$\|(R_\alpha\otimes I)\mathbb{E}_{\xi,1+\zeta, E}^{(2)}\mathbf{v}\|<\ee/2.$$  and so $(1)$ is proven.
 
    Note that \[\| (I\otimes P_\beta)\mathbb{E}^{(2)}_{\xi, 1+\zeta, E}\mathbf{v}\|=\| (I\otimes R_\beta)\mathbb{E}^{(2)}_{\xi, 1+\zeta, E}\mathbf{v}\|<\ee/2\] and \[\|(P_\alpha\otimes I) \mathbb{E}^{(2)}_{\xi, 1+\zeta, E}\mathbf{v}\|=\|(R_\alpha\otimes I) \mathbb{E}^{(2)}_{\xi, 1+\zeta, E}\mathbf{v}\|<\ee/2.\]   
Therefore \begin{align*} \|\mathbb{E}_{\xi,1+\zeta,E}^{(2)}\mathbf{v}-((I-P_{\alpha})\otimes (I-P_\beta))\mathbb{E}^{(2)}_{\xi,1+\zeta,E}\mathbf{v}\| & \leqslant  \|(P_\alpha\otimes (I-P_\beta))\mathbb{E}_{\xi,1+\zeta,E}^{(2)}\mathbf{v}\| \\ & \phantom{\leqslant}+ \|(I\otimes P_\beta)\mathbb{E}_{\xi,1+\zeta,E}^{(2)}\mathbf{v} \| \\ & \leqslant  \|(I\otimes (I-P_\beta))(P_\alpha\otimes I)\mathbb{E}_{\xi,1+\zeta,E}^{(2)}\mathbf{v}\| \\ & \phantom{\leqslant} + \|(I\otimes P_\beta)\mathbb{E}_{\xi,1+\zeta,E}^{(2)}\mathbf{v} \| \\ & = \|(R_\alpha\otimes I)\mathbb{E}_{\xi,1+\zeta,E}^{(2)}\mathbf{v}\| \\ & \phantom{\leqslant} + \|(I\otimes R_\beta)\mathbb{E}_{\xi,1+\zeta,E}^{(2)}\mathbf{v} \| <\ee/2+\ee/2=\ee. \end{align*} 
This finishes the proof of the lemma. 
\end{proof}

\section{On upper bounds  of  $\mathbf{s}_\xi(U)$ where $U=C(\omega^{\omega^\xi})\tim_\pi C(\omega^{\omega^\zeta})$ and  $\zeta \leq \xi< \omega_1$}           \label{positiveS}

With the notations from the previous section, in this section we show that under the induction hypothesis on $\xi< \omega_1$, we conclude that $\mathbf{s}_\xi(U) \leq k_G$ if $\zeta< \xi$ (Proposition \ref{P1})  and that $\mathbf{s}_\xi(U) \leqslant 2k_G$ if $\zeta= \xi$ (Proposition \ref{P2}).

\begin{proposition} \label{P1} Suppose that $h(\eta)$ holds for all $\eta<\xi$ and $\zeta< \xi$. Then $\mathbf{s}_\xi(U) \leq k_G.$
\end{proposition}

\begin{proof} Fix $\varepsilon>0$   
and assume that $h(\eta)$ holds for all $\eta<\xi.$ We  will  show that $\mathbf{s}_{\xi}(U)\leqslant k_G+\varepsilon.$ To this end, fix a weakly null collection $\mathbf{u}=(u_E)_{ E\in[\nn]^{<\omega}\setminus\{\varnothing\}}$ in $B_U$  and $L\in[\nn].$
 For an ordinal $\alpha$, we let $R_\alpha f=f|_{[0, \alpha]}$ the restriction to $[0,\alpha]$ so $R_\alpha$ is an operator into $C(\alpha)$.    By an abuse of notation, we use this notation both for the restriction operators on both $C_0(\omega^{\omega^\xi})$ and $C_0(\omega^{\omega^\zeta})$. We let $P_\alpha$ be the operator (either from $C_0(\omega^{\omega^\xi})$ to itself or from $C_0(\omega^{\omega^\zeta})$ to itself) given by letting $P_\alpha f(\varpi)=f(\varpi)$ for $\varpi\leqslant \alpha$ and $P_\alpha f(\varpi)=0$ for $\varpi>\alpha$.

We first consider the case $\zeta<\xi$.          Let $\alpha_0=-1$.     Let   $E_1\in MAX(S_\xi)\cap [L]^{<\infty}.$        There exist  finite sets, $F_1\subset B_{C_0(\omega^{\omega^\xi})}$, $G_1\subset B_{C_0(\omega^{\omega^\zeta})}$ and \[v_1\in \text{co}\{f\otimes g: (f,g)\in F_1\times G_1\}\] \cite[Proposition 2.2.]{Ryan}, such that \[\|v_1-\mathbb{E}_{\xi,E_1}\mathbf{u}\|<\ee/2.\] By replacing the members of $F_1$ with projections thereof, we can assume there exists $\alpha_1<\omega^{\omega^\xi}$ such that $\supp(f)\subset (\alpha_0, \alpha_1]$ for all $f\in F_1$.    Next, assume $E_1<\ldots <E_n$, $$E_i\in MAX(S_\xi)\cap [L\cap (\max E_{i-1},\infty )]^{<\infty},$$ for $1< i\leq n,$ $\alpha_0<\ldots <\alpha_n<\omega^{\omega^\xi}$, $F_1, \ldots, F_n$, and $G_1, \ldots, G_n$ have been chosen.  We denote $\Delta_n=E_1\cup\ldots\cup E_n.$ We denote $w_n=\sum_{\Delta_{n-1}\prec F\preceq \Delta_n}p^\xi_Fu_F.$ Since $\alpha_n<\omega^{\omega^\xi}$, we deduce from Lemma \ref{claim1} that $C(\alpha_n)\widehat{\otimes}_\pi C_0(\omega^{\omega^\zeta})$ has Szlenk index not exceeding $\omega^\xi$.    Define  the collection  
$\mathbf{v}=(v_E)_{ E\in[\nn]^{<\omega}\setminus\{\varnothing\}} \ \ \hbox{in} \ \  B_{C(\alpha_n)\tim_\pi C_0(\omega^{\omega^\zeta})},$  by letting $v_E=(R_{\alpha_n}\otimes I )u_{\Delta_n\cup E}$ if $E_n<E$ and $v_E=0$ otherwise.   The  collection $\mathbf{v}$ is weakly null.  Since $C(\alpha_n)\tim_\pi C_0(\omega^{\omega^\zeta})$ has Szlenk index not exceeding $\omega^\xi$, there exists $$E_{n+1}\in MAX(S_\xi)\cap [L\cap (\max E_n,\infty )]^{<\infty},$$   such that  
$$\|\mathbb{E}_{\xi, E_{n+1}}\mathbf{v}\|=\Vert (R_{\alpha_n}\otimes I)w_{n+1}\Vert =\left\Vert (R_{\alpha_n}\otimes I)\sum_{\Delta_n\prec F\preceq \Delta_n\cup E_{n+1}}p^\xi_Fu_F \right\Vert< \ee/3\cdot 2^{n+1}.$$   Define \[\Delta_{n+1}=E_1\cup\ldots \cup E_n\cup E_{n+1}=\Delta_n\cup E_{n+1}\] and \[w_{n+1}=\sum_{\Delta_n\prec F\preceq \Delta_{n+1}} p^\xi_F u_F.\]

We can choose finite sets $F'\subset B_{C_0(\omega^{\omega^\xi})}$, 
 $G_{n+1}\subset B_{C_0(\omega^{}\omega^\zeta)}$ and \[v'\in\text{co}\{f\otimes g: (f,g)\in F'\times G_{n+1}\}\] such that 
$$\Vert v'-w_{n+1}\Vert<\varepsilon/3\cdot 2^{n+1}.
 $$
   By replacing the members of $F'$ with small perturbations thereof, we can assume there exists $\alpha_{n+1}\in (\alpha_n, \omega^{\omega^\xi})$ such that $$\supp(f)\subset [0, \alpha_{n+1}],$$ for all $f\in F'$.     Define \[F_{n+1}=\{(I-P_{\alpha_n})f: f\in F'\}\subset B_{C_0(\omega^{\omega^\xi})}\] and \[v_{n+1}=((I-P_{\alpha_n})\otimes I)v'.\]   Note that $$\supp(f)\subset (\alpha_n, \alpha_{n+1}],$$ for all $f\in F_{n+1}$.  Note also that \begin{align*}
  \|v_{n+1}-w_{n+1}\| & = \|v'-w_{n+1}- (P_{\alpha_n}\otimes I )v'\| \\
 & \leqslant \|v'-w_{n+1} - (P_{\alpha_n}\otimes I )(v'-w_{n+1}\|+\|(P_{\alpha_n}\otimes I)w_{n+1}\| \\ 
 & = \|v'-w_{n+1} - (P_{\alpha_n}\otimes I )(v'-w_{n+1})\|+\|(R_{\alpha_n}\otimes I)w_{n+1}\| \\ 
 & \leqslant 2\|v'-w_{n+1} \|+\ee/3\cdot 2^{n+1} \\ & <\ee/2^{n+1}.\end{align*}  
 
 This completes our recursive construction.  Let $M=\cup_{n=1}^\infty E_n.$   For each integer $n,$ $w_n=\mathbb{E}_{\xi,M}\mathbf{u}(n)$ and $\Vert v_n-\mathbb{E}_{\xi,M}\mathbf{u}(n)\Vert<\varepsilon/2^n.$  Therefore we can deduce that $$\|\mathbb{E}_{\xi,M}\mathbf{u}\|_2^w \leqslant k_G+\ee,$$ by showing that $$\|(v_n)_{n=1}^\infty\|_2^w\leqslant k_G.$$  Since $v_n\in \text{co}\{f\otimes g: (f,g)\in F_n\times G_n\}$ for all $n\in\nn$, by the triangle inequality, we only need to observe that $$\|(f_n\otimes g_n)_{n=1}^\infty \|_2^w \leqslant k_G, \ \ \hbox{for all} \ \  (f_n,g_n)_{n=1}^\infty\in \prod_{n=1}^\infty F_n\times G_n.$$ For this, we apply  Proposition \ref{groth}. Here we note that for any $(f_n)_{n=1}^\infty \in \prod_{n=1}^\infty F_n$, $$\|(f_n)_{n=1}^\infty\|_2^w \leqslant \|(f_n)_{n=1}^\infty\|_1^w \leqslant 1,$$ since $(f_n)_{n=1}^\infty$ have pairwise disjoint supports.    This shows that $\mathbf{s}_\xi(U) \leqslant k_G$.    
 \end{proof}

\begin{proposition} \label{P2} Suppose that $h(\eta)$ holds for all $\eta<\xi$ and $\zeta= \xi$. Then $\mathbf{s}_\xi(U) \leqslant 2k_G.$
\end{proposition}
\begin{proof}
Fix $\varepsilon>0$   
and assume that $h(\eta)$ holds for all $\eta<\xi.$ We  will  show that $\mathbf{s}_{\xi}(U)\leqslant 2k_G+\varepsilon.$   
Fix a weakly null collection $\mathbf{u}=(u_E)_{ E\in[\nn]^{<\omega}\setminus\{\varnothing\}}$  in $B_U$  and $L\in[\nn].$

 The proof proceeds as in the previous proposition, except that we recursively choose $E_1 < E_2< \ldots$,  $-1=\alpha_0<\alpha_1<\ldots$, $F_n, \Phi_n, G_n, \Gamma_n\subset B_{C_0(\omega^{\omega^\xi})}$, and \[v_n\in\text{co}\{f\otimes g:(f,g)\in F_n\times G_n\},\] \[w_n\in \text{co}\{f\otimes g:(f,g)\in \Phi_n\times \Gamma_n\}\] such that, with $$M=\bigcup_{n=1}^\infty E_n,$$ $$\|v_n+w_n- \mathbb{E}_{\xi,M}\mathbf{u}(n)\|<\ee/2^n,$$ for all $n\in\nn$, and such that $\supp(f)\subset (\alpha_{n-1}, \alpha_n],$ for all $n\in\nn$ and $f\in F_n\cup \Gamma_n$.    From here, we argue as in the previous proposition to deduce that $$\|(f_n\otimes g_n)_{n=1}^\infty\|_2^w \leqslant k_G, \ \ \hbox{for all} \ \ (f_n,g_n)_{n=1}^\infty \in \prod_{n=1}^\infty F_n\times G_n,$$ since in this case the functions $(f_n)_{n=1}^\infty$ have pairwise disjoint supports. Similarly, $$\|(f_n\otimes g_n)_{n=1}^\infty\|_2^w\leqslant k_G, \ \ \hbox {for all}  (f_n, g_n)_{n=1}^\infty\in \prod_{n=1}^\infty \Phi_n\times \Gamma_n,$$since $(g_n)_{n=1}^\infty$ have pairwise disjoint supports.   Again, by the triangle inequality, $$\|(v_n)_{n=1}^\infty\|_2^w, \|(w_n)_{n=1}^\infty\|_2^w \leqslant k_G,$$ and \[\|\mathbb{E}_{\xi,M}\mathbf{u}\| \leqslant \|(v_n)_{n=1}^\infty\|_2^w +\|(w_n)_{n=1}^\infty\|_2^w +\sum_{n=1}^\infty\|v_n+w_n-\mathbb{E}_{\xi,M}\mathbf{u}(n)\| \leqslant 2k_G+\ee.\]   

The base step of making this sequence of choices is arbitrary, as in the previous proposition.   We indicate how to complete the recursive step.   Assume $E_1< \ldots< E_n,$ $$E_i\in MAX (S_\xi)\bigcap [L\cap(\max E_{i-1},\infty)]^{<\omega},$$ for $1<i \leqslant n$ and $\alpha_0< \ldots<\alpha_n<\omega^{\omega^\xi}$ have been chosen.  Denote $\Delta_n=E_1\cup \ldots\cup E_n$ and $m_n=\sum_{\Delta_{n-1}\prec F\preceq\Delta_n}p^\xi_Fu_F.$
We now consider the collection $$\mathbf{v}=(v_E)_{ E\in [\nn]^{<\omega}\setminus\{\varnothing\}} \ \ \hbox{in} \ \ B_{C(\alpha_n)\tim_\pi C(\alpha_n)},$$ given by  $v_E=(R_{\alpha_n}\otimes R_{\alpha_n})u_{\Delta_n\cup E}$ if $E_n<E$ and $v_E=0$ otherwise.    By Lemma \ref{claim1}, $$Sz(C(\alpha_n)\tim_\pi C(\alpha_n))\leqslant\omega^\xi.$$ Thus,  as in the previous proposition, we can choose $$E_{n+1}\in MAX S_\xi\cap [L\cap(\max E_{n},\infty)]^{<\omega}$$ such that
$$\Vert\mathbb{E}_{\xi, E_{n+1}}\mathbf{v}\Vert=\Vert(R_{\alpha_n}\otimes R_{\alpha_n})m_{n+1}\Vert
<\ee/3\cdot 2^{n+1}.$$    

We then choose a finite set $F\subset B_{C_0(\omega^{\omega^\xi})}$ and $\alpha_{n+1}\in (\alpha_n, \omega^{\omega^\xi})$ such that $\supp(f)\subset [0, \alpha_{n+1}],$ for all $f\in F$ and \[v'\in \text{co}\{f\otimes g: (f,g)\in F\times F\}\] such that $$\|v'-\mathbb{E}_{\xi, E_{n+1}}\mathbf{v}\|<\ee/3\cdot 2^{n+1}.$$     
 We then let \[F_{n+1}=\Gamma_{n+1}=\{(I-P_{\alpha_n})f: f\in F\},\] \[G_{n+1}=F,\] \[\Phi_{n+1}= \{P_{\alpha_n}f :f\in F\},\] \[v_{n+1}=((I-P_{\alpha_n})\otimes I)v',\] and \[w_{n+1} = (P_{\alpha_n}\otimes (I-P_{\alpha_n}))v'.\] We compute \begin{align*} \|v_{n+1}-w_{n+1} - \mathbb{E}_{\xi,E_{n+1}}\mathbf{v}\| & = \|v'-\mathbb{E}_{\xi,E_{n+1}}\mathbf{v}+{(P_{\alpha_n}\otimes P_{\alpha_n})}v'\| \\ & \leqslant 2\|v'-\mathbb{E}_{\xi,E_{n+1}}\mathbf{v}\|+\|{(P_{\alpha_n}\otimes P_{\alpha_n})} \mathbb{E}_{\xi,E_{n+1}}\mathbf{v}s\| \\ & <\ee/2^{n+1}.  \end{align*}   This completes the proof that $\mathbf{s}_\xi(U)\leqslant 2k_G$.  
\end{proof}

\section{On upper bounds of  $\mathbf{g}_{\xi,1+\zeta}(U)$ where $U=C(\omega^{\omega^\xi})\tim_\pi C(\omega^{\omega^\zeta})$ and  $\zeta \leq \xi< \omega_1$}           \label{positive}

In section we conclude  the proof of Theorem \ref{positive one} by induction  on $\xi< \omega_1$ proving the following proposition.

\begin{proposition} Suppose that $h(\eta)$ holds for all $\eta<\xi$. If $\zeta<\xi$ then  $\mathbf{g}_{\xi,1+\zeta}(U) \leqslant k_G.$ Moreover, If $\zeta=\xi$ then $\mathbf{g}_{\xi,1+\zeta}(U)  \leqslant 2 k_G.$
\end{proposition}

\begin{proof} If $\zeta<\xi$, fix $\kappa>k_G$ and if $\zeta=\xi$, fix $\kappa>2k_G$. Fix $\varepsilon>0$ such that $ k_G+ \ee<\kappa$  if $\zeta<\xi$  or $2k_G+ \ee<\kappa $ if $\zeta=\xi.$ 
Fix  a weakly null collection  $\mathbf{u}=(u_E)_{ E\in[\nn]^{<\omega}\setminus\{\varnothing\}}$ in $B_U$  and $L\in[\nn]$.        Since $\mathbf{s}_\xi(U) <\kappa$,  we can assume, by replacing $L$ with a subset, and using  Remark \ref{perm3} and Proposition \ref{hereditary} that for all $E\in MAX(S_{1+\zeta}[S_\xi])\bigcap [ L]^{<\omega},$ $$\|\mathbb{E}^{(2)}_{\xi,  1+\zeta, E}\mathbf{u}\| \leqslant \kappa.$$ Let $\alpha_0=\beta_0=-1$.    Let  $E_1\in MAX(S_{1+\zeta}[S_\xi])\bigcap [ L]^{<\omega}.$     Since $\|\mathbb{E}^{(2)}_{\xi, 1+\zeta, E_1}\mathbf{u}\|\leqslant \kappa$, we can select finite sets $F_1\subset \kappa B_{C_0(\omega^{\omega^\xi})}$, $G_1\subset B_{C_0(\omega^{\omega^\zeta})}$, $\alpha_1<\omega^{\omega^\xi}$, $\beta_1<\omega^{\omega^\zeta}$, and \[v_1\in \text{co}\{f\otimes g: (f,g)\in F_1\times G_1\}\] such that $\supp(f)\subset (\alpha_0, \alpha_1]$ for all $f\in F_1$, $\supp(g)\subset (\beta_0, \beta_1]$ for all $g\in G_1$, and \[\|v_1-\mathbb{E}^{(2)}_{\xi, 1+\zeta,E_1}\mathbf{u}\|<\ee/2.\]   

Now assume that $E_1<\ldots <E_n$, $$E_i\in  MAX(S_{1+\zeta}[S_\xi])\bigcap [ L\cap(\max E_{i-1},\infty)]^{<\omega},$$ for $1<i\leqslant n,$ $\alpha_0<\ldots <\alpha_n<\omega^{\omega^\xi}$,  $\beta_0<\ldots < \beta_n<\omega^{\omega^\zeta}$, $F_1, \ldots, F_n$, $G_1, \ldots, G_n$, and $v_1, \ldots, v_n$ have been chosen. Denote $\Delta_n=E_1\cup\ldots\cup E_n$ Define $\mathbf{v}=(v_E)_{ E\in[\nn]^{<\omega}\setminus\{\varnothing\}}\subset B_U$ by letting $v_E=u_{ \Delta\cup E}$ if $E_n<E$ and $v_E=0$ otherwise.   Using Lemma \ref{claim2}, we can select $$E_{n+1}\in MAX(S_{1+\zeta}[S_\xi])\cap [(\max E_n, \infty)\cap L]^{<\omega}$$ such that  \[\|\mathbb{E}^{(2)}_{\xi, 1+\zeta, E_{n+1}}\mathbf{v}\|\leqslant \kappa\] and \[\|\mathbb{E}^{(2)}_{\xi, 1+\zeta, E_{n+1}}\mathbf{v}-((I-P_{\alpha_n})\otimes (I-P_{\beta_n}))\mathbb{E}^{(2)}_{\xi, 1+\zeta, E_{n+1}}\mathbf{v}\|<\ee/2^{n+2}.\]    We can choose finite sets $F\subset \kappa B_{C_0(\omega^{\omega^\xi})}$ and $G\subset B_{C_0(\omega^{\omega^\zeta})}$, ordinals $\alpha_{n+1}\in (\alpha_n, \omega^{\omega^\xi})$, $\beta_{n+1}\in (\beta_n, \omega^{\omega^\zeta})$, and \[v'\in \text{co}\{f\otimes g:(f,g)\in F\times G\}\] such that $\supp(f)\subset [0,\alpha_{n+1}]$ for all $f\in F$, $\supp(g)\subset [0, \beta_{n+1}]$ for all $g\in G$, and 
$$\|v'- \mathbb{E}^{(2)}_{\xi, 1+\zeta, E_{n+1}}\mathbf{v}\|<\ee/2^{n+2}.$$
   Define \[F_{n+1}=\{(I-P_{\alpha_n})f:f\in F\},\] \[G_{n+1}=\{(I-P_{\beta_n})g: g\in G\},\] and \[v_{n+1}=((I-P_{\alpha_n})\otimes(I-P_{\beta_n}))v'.\]    Note that $\supp(f)\subset (\alpha_n, \alpha_{n+1}]$ for each $f\in F_{n+1}$, $\supp(g)\subset (\beta_n, \beta_{n+1}]$ for each $g\in G_{n+1}$, and 
\begin{align*} 
\|v_{n+1} - \mathbb{E}^{(2)}_{\xi, 1+\zeta, E_{n+1}}\mathbf{v})\| & \leqslant \|((I-P_{\alpha_n})\otimes(I-P_{\beta_n}))(v'-\mathbb{E}^{(2)}_{\xi, 1+\zeta, E_{n+1}}\mathbf{v})\| \\ & + \|\mathbb{E}^{(2)}_{\xi, 1+\zeta, E_{n+1}}\mathbf{v}- ((I-P_{\alpha_n})\otimes(I-P_{\beta_n}))\mathbb{E}^{(2)}_{\xi, 1+\zeta, E_{n+1}}\mathbf{v} \| \\ & <\ee/2^{n+2}/2+\ee/2^{n+2}=\ee/2^{n+1}
.\end{align*} 
This completes the recursive construction. 

 Let $N=\cup_{n=1}^\infty E_n$.    We note that for each $(f_n, g_n)_{n=1}^\infty\in \prod_{n=1}^\infty F_n\times G_n$, $\|(f_n)_{n=1}^\infty\|_1^w \leqslant \kappa,$ since the functions $f_n\in  \kappa B_{C_0(\omega^{\omega^\xi})}$ have pairwise disjoint supports,  and $\|(g_n)_{n=1}^\infty\|_1^w \leqslant  1$, since  the functions $g_n\in B_{C_0(\omega^{\omega^\zeta})}$ have pairwise disjoint supports.     Therefore by Proposition \ref{haig}, $$\|(f_n\otimes g_n)_{n=1}^\infty\|_1^w \leqslant \kappa,$$ for each $(f_n, g_n)_{n=1}^\infty \in\prod_{n=1}^\infty F_n\times G_n$.     By the triangle inequality, $\|(v_n)_{n=1}^\infty\|_1^w \leqslant \kappa.$    Since  $\mathbb{E}^{(2)}_{\xi, 1+\zeta, N}\mathbf{u}(n+1)=\mathbb{E}^{(2)}_{\xi, 1+\zeta, E_{n+1}}\mathbf{v},$ it follows that  
$$ \|\mathbb{E}^{(2)}_{\xi, 1+\zeta, N}\mathbf{u}\|_w^1 \leqslant \|(v_n)_{n=1}^\infty\|_1^w + \sum_{n=1}^\infty \|v_n-\mathbb{E}^{(2)}_{\xi, 1+\zeta, N}\mathbf{u}(n)\| \leqslant \kappa+ \ee. 
$$    
 Since this holds for any $\varepsilon>0$, any weakly null collection  $\mathbf{u}=(u_E)_{E\in [\nn]^{<\omega}\setminus \{\varnothing\}}$ in $ B_U$, and $L\in[\nn]$, it follows that $\mathbf{g}_{\xi, 1+\zeta}(U) \leqslant \kappa$. Since $\kappa>k_G$ was arbitrary in the $\zeta<\xi$ case, and $\kappa>2k_G$ was arbitrary in the $\zeta=\xi$ case, we are done. 
\end{proof}

\section{An appropriate Cantor scheme  on the spaces $[1, \omega^{\gamma}]$, $\gamma< \omega_1$}

In order to obtain in the next section a lower bound for $\mathbf{g}_{\xi,1+\zeta}(C(\omega^{\omega^\xi})\tim_\pi C(\omega^{\omega^\zeta}))$ with  $\zeta \leq \xi< \omega_1$  we need the existence of a convenient  Cantor scheme on $[1, \omega^\gamma]$. The proof of it is similar to the computation of the Bourgain $\ell_1$-index of the spaces $C(\omega^{\omega^\zeta})$ found in the work of Alspach, Judd, and Odell \cite{AJO}.

\begin{lemma} For any   $1\leqslant\gamma<\omega_1$, there exists a tree $T_\gamma$  on $[0,\omega^\gamma)$ with $\text{\emph{rank}}(T_\gamma)=\omega\gamma$ and a collection $(f_t)_{t\in T_\gamma}$ in $ B_{C_0(\omega^\gamma)}$ such that for each $t\in MAX(T_\gamma)$,  $(f_s)_{s\preceq t}$ is compatibe with a Cantor scheme on $[1, \omega^\gamma]$.    
\label{tree}
\end{lemma}

\begin{proof} We work by induction.      We will also define the functions $(f_t)_{t\in T_{\gamma}}$ to satisfy $f_t(0)=0$. 

For the $\gamma=1$ case, we first define \[T_{0,n}=\{(n-1,n-2, \ldots, n-k): 1\leqslant k\leqslant n\}\] and define $T_{1}=\cup_{n=1}^\infty T_{0,n}$.  Note that $\text{rank}(T_{1})=\omega= \omega \cdot 1$, since $T_{1}=\cup_{n=1}^\infty T_{0,n}$ is a totally incomparable union and the rank of $T_{0,n}$ is $n$.  In order to define $(f_t)_{t\in T_1}$  it suffices to define $(f_s)_{s\in T_{0,n}}$ for each $n\in\nn$. To that end, fix $n\in\nn$.   We define $$f_{(n-1,...,n-k)}=\sum_{i=1}^{2^k}(-1)^{i-1}1_{(2^{n-k}(i-1),2^{n-k}i]}$$
where $1_{(2^{n-k}(i-1),2^{n-k}i]}$ is the indicator function of  $(2^{n-k}(i-1),2^{n-k}i]$ in $[0,\omega].$

The sequence $(f_{(n-1,...,n-k)})_{k=1}^n$ is compatible with the Cantor scheme defined by $A_\varnothing=(0,2^n]$ and for $d=(\varepsilon_1,...,\varepsilon_k)\in\Delta_k$ by
$$A_{(\varepsilon_1,...,\varepsilon_k)}=\bigcap_{l=1}^k
f^{-1}_{(n-1,...,n-l)}(\varepsilon_l).$$

Assume the result holds for some ordinal $\gamma$. Fix a tree $T_\gamma$ on $[0, \omega^\gamma)$ with $\text{rank}(T_\gamma) = \omega\gamma$ and a collection $(g_t)_{t\in T_\gamma}$ as in the conclusion, and also satisfying $g_t(0)=0$ for all $t\in T_\gamma$.   As in the  $\gamma=1$ case,  it suffices to produce pairwise incomparable trees $T_{\gamma,n}$ on $[0, \omega \gamma + n)\subset [0, \omega (\gamma+1))$ with $\text{rank}(T_{\gamma,n})=\omega \gamma +n $ and  a collection of functions $(f_t)_{t\in T_{\gamma,n}}$ in $B_{C_0(\omega^{\gamma+1})}$ satisfying the conclusions, and then let $T_{\gamma+1}=\cup_{n=1}^\infty T_{\gamma,n}$. To that end, fix $n\in\nn$.   Let \[R_n=\{(\omega \gamma+ n-1, \ldots, \omega \gamma + n-k): 1\leqslant k\leqslant n\},\] \[S_n = \{(\omega \gamma + n-1, \ldots, \omega \gamma)\smallfrown t: t\in T_\gamma\},\] and let \[T_{\gamma,n}=R_n\cup S_n.\]  It is clear that $\text{rank}(T_{\gamma,n})=\omega \gamma + n$. 

In order to define the collection $(f_t)_{t\in T_{\gamma,n}}, $ we first define $(f_t)_{t\in R_{\gamma,n}}$ by
$$f_{(\omega \gamma+ n-1, \ldots, \omega \gamma + n-k)}=\sum_{i=1}^{2^k}(-1)^{i-1}1_{(\omega^\gamma2^{n-k}(i-1),\omega^\gamma2^{n-k}i]}$$ 
where, for $k=1,...,n$ and $i=1,...,2^k,$ $1_{(\omega^\gamma2^{n-k}(i-1),\omega^\gamma2^{n-k}i]}$ denotes the indicator function of the interval $
(\omega^\gamma2^{n-k}(i-1),\omega^\gamma2^{n-k}i]$ in $[0,\omega^{\gamma+1}].$ 
 For $t=(\omega\gamma+n-1,\ldots,\omega\gamma)\smallfrown s$ with $s\in T_\gamma,$ we define $f_t$ on each interval $(\omega^\gamma(i-1),\omega^\gamma i]$ for $1\leqslant i\leqslant 2^{n}$ by letting $$f_t(\omega^\gamma(i-1)+\eta)=g_s(\eta), \ \  \eta\in [0,\omega^\gamma],$$ and we define $f_t$ on $(\omega^\gamma2^n, \omega^{\gamma+1}]$
 by letting $$f_t(\eta)=0 \ \ \hbox {for all} \ \ \eta\in(\omega^\gamma2^n, \omega^{\gamma+1}].$$  It is easy to see that each $f_t\in C_0(\omega^{\gamma+1}).$  Let $t=(\omega\gamma+n-1,...,\omega\gamma) \smallfrown s\in MAX(T_{\gamma,n}),$ by construction $(f_r)_{r\preceq t}$ is compatible with a Cantor scheme on $[0,\omega^{\gamma+1}].$

Now assume that $\gamma$ is a limit ordinal and the result holds for all $1\leqslant \zeta<{\gamma}$.  For each $1\leqslant \zeta<\gamma$, let $T_\zeta$ be a tree on $[0, \omega^\zeta)$ with $\text{rank}(T_\zeta)=\omega \zeta$ and let $(g_t)_{t\in T_\zeta}$ be a family in $ B_{C_0(\omega^\zeta)}$ be as in the statement of the lemma.     Let \[U_{\zeta+1} = \{(\omega^\zeta + \nu_i)_{i=1}^n: (\nu_i)_{i=1}^n\in T_{\zeta+1}\}.\] Note that the function $ \eta\in[0,\omega^{\zeta+1})\longrightarrow \omega^\zeta+\eta\in[\omega^\zeta,\omega^{\zeta+1})$ is bijective, therefore the  function  $\nu=(\nu_i)_{i=1}^n\in T_{\zeta+1}\mapsto \phi_\zeta(\nu)=(\omega^\zeta+\nu_i)_{i=1}^n\in U_{\zeta+1}$ is a tree isomorphism, and  $U_{\zeta+1}$ is a tree on $[\omega^\zeta, \omega^{\zeta+1})$ with rank $\omega(\zeta+1)$. This construction guarantees that the trees $(U_{\zeta+1})_{\zeta< \gamma}$ are pairwise incomparable. From this it follows that $T_\gamma=\cup_{\zeta<\gamma}U_{\zeta+1}$ has rank $\omega\gamma$.

 For $t\in U_{\zeta+1}$, define $f_t\in C_0(\omega^{\gamma+1})$ to be the function defined by \[f_t(\eta) = \left\{ \begin{array}{ll} g_{\phi_\zeta^{-1}(t)}(\eta) & : \eta\leqslant \omega^{\zeta+1} \\ 0 & : \omega^{\zeta+1}<\eta \leqslant \omega^\gamma. \end{array}\right.\] 
 
 It is easy to see that $(f_t)_{t\in T_\gamma}$ satisfies the conclusion. Indeed, for any $t\in MAX(T_\gamma)$,  $\phi_\zeta^{-1}(t)\in MAX(T_{\zeta+1})$ for some $\zeta<\gamma$.  
 By construction, there exists a Cantor scheme on $[1, \omega^{\zeta+1}]$ with respect to which $(g_s)_{s\preceq \phi_\zeta^{-1}(t)}=(g_{\phi_\zeta^{-1}(s)})_{s\preceq t}$ is compatible.    Since $f_s|_{[1, \omega^{\zeta+1}]}= g_{\phi_\zeta^{-1}(s)}$ for each $s\preceq t$, it follows that if $D$ is a Cantor scheme on $[1, \omega^{\zeta+1}]$ with respect to which $(g_s)_{s\preceq \phi_\zeta^{-1}(t)}$ is compatible, then $D$ is also a Cantor scheme on $[1, \omega^\gamma]$ with respect to which $(f_s)_{s\preceq t}$ is also compatible. 
\end{proof}

\section{A lower bound of  $\mathbf{g}_{\xi,1+\zeta}(U)$ where $U=C(\omega^{\omega^\xi})\tim_\pi C(\omega^{\omega^\zeta})$ and $\zeta \leq \xi< \omega_1$}\label{sharpness}

The main purpose of this section is to prove $\mathbf{g}_{\xi,1+\zeta}(C_0(\omega^{\omega^\xi})\widehat{\otimes}_\pi C_0(\omega^{\omega^\zeta}))\geqslant1.$ For this we need the following lemma.
\begin{lemma} For any $1\leqslant \xi,\zeta<\omega_1$ and any tree $T$ with $\mathop{\rm rank}(T)=\omega^{1+\zeta},$ there exists a monotone map $
\Phi:\mathcal{S}_{1+\zeta}[\mathcal{S}_\xi]\setminus\{\varnothing\}\to T$ such that for each $E\in MAX(\mathcal{S}_{1+\zeta}[\mathcal{S}_\xi])$ if $E=\bigcup_{n=1}^mE_n$ with $E_1<\cdots<E_m$ and $E_n\in MAX(\mathcal{S}_\xi) $ for each $ 1\leqslant n\leqslant m,$ then, for each $ 1\leqslant i\leqslant m,$   $ \Phi(\bigcup_{n=1}^iE_n)=t_i$ and $\Phi(F)=t_i$ for each
 $\bigcup_{n=1}^{i-1}E_n\prec F\preceq\bigcup_{n=1}^iE_n.$
\label{mon}
\end{lemma}
\begin{proof}In the proof, we use the following fact about rank functions on trees. For a well-founded tree $S$ and $s\in S$, there exists an unique ordinal $\varrho(s)$ such that $s\in S^{(\varrho(s))}$ and $s\notin S^{(\varrho(s)+1)}.$ The uniqueness of such an ordinal $\varrho(s)$ is clear. Let us prove the existence of such a $\varrho(s)$. Since $S^{(\text{rank}(S))}=\varnothing$ there exists 
$\lambda=\min\{\mu : s\notin S^{(\mu)}\}.
$ If $\lambda$ is a limit ordinal then, for all $\theta<\lambda,$ $s\in S^{(\theta)}$ and so $s\in \bigcap_{\theta<\lambda}S^{(\theta)}=S^{(\lambda)}$ in contradiction with the definition of $\lambda.$ We define $\varrho(s)$ by $\lambda=\varrho(s)+1,$ of course $\varrho(s)=\max\{\mu: s\in S^{(\mu)}\}.$   We also note that if $s\prec s'$, then $\varrho(s)>\varrho(s')$.   More precisely, \[\varrho(s) = \sup \{\varrho(s')+1:s\prec s'\in S\},\] with the convention that the supremum of the empty set is $0$.

 We note that $\mathop{\rm rank}(\mathcal{S}_{1+\zeta})=\omega^{1+\zeta}+1$ and 
$\mathop{\rm rank}(\mathcal{S}_{1+\zeta}\setminus\{\varnothing\})
=\omega^{1+\zeta}.$
Define $h:\mathcal{S}_{1+\zeta}\setminus\{\varnothing\}\to [0,\omega^{1+\zeta})$  by $$h(E)=\max\{\mu:E\in\mathcal{S}_{1+\zeta}^{ (\mu)} \},$$ and define $g:T\to [0,\omega^{1+\zeta})$ by $$g(t)=\max\{\mu : t \in T^{ (\mu)} \}.$$

We will define $\Phi(E)$ by induction on the cardinality of $E$ to have the property that for $E=\bigcup_{n=1}^mE_n\in \mathcal{S}_{1+\zeta}[\mathcal{S}_\xi]$ such that $E_1<\ldots<E_m$ and $E_n\in MAX(\mathcal{S}_\xi)$ for $1\leqslant n<m,$ $$h((\min E_n)_{n=1}^m)\leqslant g(\Phi(E)).$$ 

The case $\vert E\vert=1.$ For $(n)\in \mathcal{S}_{1+\zeta}[\mathcal{S}_\xi],$ $h((n))<\omega^{1+\zeta},$ so there exists $t\in T$ such that $h((n))\leqslant g(t).$ Define $\Phi((n))=t.$

Fix  $G\in\mathcal{S}_{1+\zeta}[\mathcal{S}_\xi]$ with $|G|>1$ and assume that $\Phi(E)$ has been defined for each $E$ with $|E|<|G|$.  Let $p=\max G$ and let $E=G\setminus \{p\}$.  Then $G=E\smallfrown(p)$ and $E\in \mathcal{S}_{1+\zeta}[\mathcal{S}_\xi]\setminus MAX(\mathcal{S}_{1+\zeta}[\mathcal{S}_\xi]).$ Write $E=\cup_{n=1}^mE_n$ with $E_1<\cdots<E_m,$   $E_n\in MAX(\mathcal{S}_\xi)$ for $1\leqslant n<m$  and $E_m\in \mathcal{S}_\xi.$ If
$E_m\notin MAX(\mathcal{S}_\xi)$ we define $\Phi(G)=\Phi(E).$ If $E_m\in MAX(\mathcal{S}_\xi) ,$  let $t_m=\Phi(E).$ Let $H=(\min E_n)_{n=1}^m$ and note that, by construction, $h(H)\leqslant g(t_m).$ Note that $h(H\smallfrown (p))<h(H)\leqslant g(t_m).$ Therefore there exists $s\in T$ such that $t\prec s$ and $g(s)\geqslant h(H\smallfrown(p)).$ Define $\Phi(G)=s.$ This complete the recursive construction. It is clear that the conclusions are satisfied in this case.
\end{proof}

\begin{theorem} For each $\zeta\leqslant \xi<\omega_1$, $\mathbf{g}_{\xi, 1+\zeta}(C_0(\omega^{\omega^\xi})\tim_\pi C_0(\omega^{\omega^\zeta})) \geqslant 1$. 
\label{negative one}
\end{theorem}
\begin{proof} Using the notations introduced in section   \ref{seccantor}, by Remark \ref{Cantor} we know that $C_0(\omega^{\omega^\xi})$ is isometrically isomorphic to $C_0(\mathcal{S}_\xi)$. We will use $C_0(\mathcal{S}_\xi)$ in place of $C_0(\omega^{\omega^\xi})$.    

For each integer $n$ we denote $e_n$ the function on $S_\xi$ by letting $e_n(E)=1_E(n)$ where $1_E$ is the indicator function of $E.$ It is obvious that $e_n$ is continuous. Since $e_n(\varnothing)=0$    then $e_n \in C_0(S_\xi).$ Since each $E\in \mathcal{S}_\xi$ is finite, $\lim_n e_n(E)=0$ for all $n\in\nn$  and since  the sequence $(e_n)_{n=1}^\infty$ is bounded and pointwise null, it is weakly null.

By Lemma \ref{tree}, there exist a tree $T$ with $\text{rank}(T)=\omega^{1+\zeta}$ and a collection $(f_t)_{t\in T}$ in $ B_{C_0(\omega^{\omega^\zeta})}$ such that for every $t\in MAX(T)$, $(f_s)_{s\preceq t}$ is compatible with a Cantor scheme on $[0, \omega^{\omega^\zeta}].$    

By Lemma \ref{mon}, there exists a map $\Phi:\mathcal{S}_{1+\zeta}[\mathcal{S}_\xi]\setminus \{\varnothing\}\to T$ such that for each $E\in MAX(\mathcal{S}_{1+\zeta}[\mathcal{S}_\xi])$, if $E=\cup_{n=1}^m E_n$ with $E_1<\ldots <E_m$ and $E_n\in MAX(\mathcal{S}_\xi)$ for each $1\leqslant n\leqslant m$, then there exist $s_1\prec \ldots \prec s_m$, $s_n\in T$, such that for each $1\leqslant i\leqslant m$, $\Phi(F)=s_i$ whenever $\cup_{n=1}^{i-1}E_n\prec F \preceq \cup_{n=1}^i E_n$.

For each $E\in[\nn]^{<\omega}$, define
\begin{equation*}
u_E=\begin{cases}e_{\max E}\otimes f_{\Phi(E)}&\text{if } E\in \mathcal{S}_{1+\zeta}[\mathcal{S}_\xi]\setminus \{\varnothing\}\\
0 &\text{otherwise}\\
\end{cases}
\end{equation*}

\begin{cl} The collection $\mathbf{u}=(u_E)_{ E\in[\nn]^{<\omega}\setminus\{\varnothing\}}$ in $B_{C_0(\mathcal{S}_\xi)\tim_\pi C_0(\omega^{\omega^\zeta})}$ is weakly null and $$\|\mathbb{E}_{\xi, 1+\zeta, M}^{(2)}\mathbf{u}(1)\|\geqslant 1,$$ for all $M\in[\nn]$.   

\end{cl}

It is obvious that the claim gives the theorem. In order to prove the claim we first show that $\mathbf{u}$ is weakly null.     Note that 

$$(C_{0}(\mathcal{S}_\xi)\tim_\pi C_{0}(\omega^{\omega^\zeta}))^*=\mathcal{B}(C_{0}(\mathcal{S}_\xi)\times C_{0}(\omega^{\omega^\zeta}))=\mathcal{L}(C_{0}(\mathcal{S}_\xi),\ell_1(\omega^{\omega^\zeta})).$$

    Fix $b\in\mathcal{B}(C_{0}(\mathcal{S}_\xi)\times C_{0}(\omega^{\omega^\zeta}))$ and $T\in  \mathcal{L}(C_{0}(\mathcal{S}_\xi),\ell_1(\omega^{\omega^\zeta}))$ the correponding operator given by $T(f)(g)= b(f,g)$.
Fix $E\in [\nn]^{<\omega}.$ By standard properties of $\mathcal{S}_{1+\zeta}[\mathcal{S}_\xi]$, if $E\smallfrown (m)\in \mathcal{S}_{1+\zeta}[\mathcal{S}_\xi]$ for some $m>E$, then $E\smallfrown (n)\in \mathcal{S}_{1+\zeta}[\mathcal{S}_\xi]$ for all $n>E$. In this case, for each $E<n$,  $$u_{E\smallfrown (n)}=e_n\otimes f_{\Phi(E\smallfrown (n))} \ \ \hbox{and} \ \ b(u_{E\smallfrown (n)})=T(e_n)(f_{\Phi(E\smallfrown (n))}).$$ The sequence $(T(e_n))_{n=1}^\infty$ is weakly null in the Schur space $\ell_1([0, \omega^{\omega^\zeta}])$, so $\lim_n\| T(e_n)\|=0.$  Since the sequence $(f_{\Phi(E\smallfrown (n))})_{n=1}^\infty$ is bounded, $\lim_n b(u_{E\smallfrown(n)})=0$.   Therefore $(u_{E\smallfrown(n)})_{E<n}$ is weakly null in the case that $E\smallfrown(n)\in \mathcal{S}_{1+\zeta}[\mathcal{S}_\xi]$ for some (equivalently, every) $E<n$.

 If $E\smallfrown (n)\in [\nn]^{<\omega}\setminus \mathcal{S}_{1+\zeta}[\mathcal{S}_\xi]$ for all $n>E$, then the sequence $(u_{E\smallfrown (n)})_{n>E}$ is identically zero. Therefore $\mathbf{u}$ is weakly null.

Next, fix $M\in[\nn]$ and let $E\in MAX(\mathcal{S}_{1+\zeta}[\mathcal{S}_\xi])$ be such that $E\prec M$.  Let $(E_n)_{n=1}^m$ be the $S_\xi$-decomposition of $E.$ It follows from the Remark \ref{4} that, for $1\leqslant i\leqslant m,$ the function $q^{\xi,1+\zeta}$ is constant on $\{F : \cup_{n=1}^{i-1}E_n\prec F\preceq\cup_{n=1}^{i}E_n\}$ and  if we denote $a_i$ this common value, then $a_i$ is positive and $\sum_{i=1}^ma_i^2=1.$

For $1\leqslant i\leqslant m$, let $x^*_i=\delta_{E_i}\in \mathcal{M}(\mathcal{S}_\xi)$. Note that for $\varnothing\prec F\preceq E,$ $\langle x_i^*, e_{\max F}\rangle=1$ if and only if $\cup_{n=1}^{i-1}E_n\prec F\preceq \cup_{n=1}^{i}E_n$,  and $\langle x_i^*, e_{\max F}\rangle=0$ otherwise. Therefore
$$\mathbb{E}^2_{\xi,1+\zeta,M}\mathbf{u}(1)=\sum_{i=1}^m\sum_{\cup_{n=1}^{i-1}E_n\prec F\preceq \cup_{n=1}^{i}E_n}q_F^{\xi,1+\zeta}p^\xi_Fe_{\max F}\otimes f_{\Phi(f)}=
\sum_{i=1}^ma_i\sum_{\cup_{n=1}^{i-1}E_n\prec F\preceq \cup_{n=1}^{i}E_n}p^\xi_Fe_{\max F}\otimes f_{s_i}
$$
For $\cup_{n=1}^{i-1}E_n\prec F\preceq \cup_{n=1}^{i}E_n,$ $p^\xi_F\geq0$ and by Lemma \ref{perm1},
$$\sum_{\cup_{n=1}^{i-1}E_n\prec F\preceq \cup_{n=1}^{i}E_n}p^\xi_F=1.$$ Then, by Proposition \ref{lower}, $\|\mathbb{E}^{(2)}_{\xi, 1+\zeta,M}\mathbf{u}(1)\|\geqslant 1.$ This finishes the claim and the theorem.
\end{proof}

\section{The isomorphism classes  of $C(\omega^{\omega^\xi})\tim_\pi C(\omega^{\omega^\zeta})$, $\xi, \zeta< \omega_1$} \label{last}

We conclude by describing the isomorphism classes of $C(K)\widehat{\otimes}_\pi C(L)$ for countable comapct metric  spaces  $K,L$. We  note that the spaces $C(K)$, $K$ countable compact metric  spaces fit into the classification as $C(K)\widehat{\otimes}_\pi C(\{0\})$.  Of course, the situation in which $K$ and $L$ are both finite and isomorphism is determined by the dimension of $C(K)\widehat{\otimes}_\pi C(L)$, which is $|K||L|$.  Therefore we omit the situation in which both $K$ and $L$ are finite, but include the case in which at most one of the spaces $K,L$ is finite in our classification.

\begin{theorem} Let $K,L,M,N$ be countable compact metric  spaces such that $K,M,N$ are infinite,  $CB(L)\leqslant CB(K)$, and $CB(N)\leqslant CB(M)$.   The following are equivalent. \begin{enumerate}[(i)]\item $C(K)\tim_\pi C(L)$ and $C(M)\tim_\pi C(N)$ are isomorphic to subspaces of quotients of each other. \item $C(K)\tim_\pi C(L)$ and $C(M)\tim_\pi C(N)$ are isomorphic. \item $C(K)$ is isomorphic to $C(M)$ and $C(L)$ is isomorphic to $C(N)$. \end{enumerate}
\label{finish_line_1}
\end{theorem}

\begin{proof} Of course, $(iii)\Rightarrow(ii)\Rightarrow(i)$. Assume $(i)$ holds.  Let $\xi<\omega_1$ be such that $C(K)$ is isomorphic to $C(\omega^{\omega^\xi})$ and let $\mu<\omega_1$ be such that $C(M)$ is isomorphic to $C(\omega^{\omega^\mu}).$ Note that $Sz(C(K)\tim_\pi C(L))=\omega^{\xi+1}$ and $Sz(C(M)\tim_\pi C(N))=\omega^{\mu+1}$, since $CB(L)\leqslant CB(K)$ and $CB(N)\leqslant CB(M)$.   By properties of the Szlenk index, since $C(K)\tim_\pi C(L)$ and $C(M)\tim_\pi C(N)$ are isomorphic to subspaces of each other, $\omega^{\xi+1}\leqslant \omega^{\mu+1}\leqslant \omega^{\xi+1}$, and $\xi=\mu$.  Therefore $C(K)$ and $C(M)$ are isomorphic to $C(\omega^{\omega^\xi})=C(\omega^{\omega^\mu})$, and are therefore also isomorphic to each other.

We first show that $L$ must be infinite. We work by contradiction. Assume that $L$ is finite, in which case $C(K)\tim_\pi C(L)$ is isomorphic to $C(K).$ It follows from the proof of  item (iii) of Proposition \ref{ideal} 
 that,  for any countable $\alpha, \beta$, $\mathbf{g}_{\alpha, \beta}(C(K)\tim_\pi C(L))=0$ if and only if $\mathbf{g}_{\alpha, \beta}(C(M)\tim_\pi C(N))=0$, and $\mathbf{g}_{\alpha, \beta}(C(K)\tim_\pi C(L))<\infty$ if and only if $\mathbf{g}_{\alpha,\beta}(C(M)\tim_\pi C(N))<\infty$.  By  Lemma \ref{knight} we have $\mathbf{g}_{\xi,0}(C(K)\tim_\pi C(L))<\infty$.    Therefore $\mathbf{g}_{\xi,0}(C(M)\tim_\pi C(N))<\infty$, which implies, by Proposition\ref{die} that $\mathbf{g}_{\xi,\beta}(C(M)\tim_\pi C(N))=0$ for all $\beta>0$.   Since  $N$ is infinite, $C(N)$ is isomorphic to $C(\omega^{\omega^\nu})$ for some $\nu<\omega_1$. Moreover, note that by our assumptions about the Cantor-Bendixson  indices it follows that  $\nu \leq \xi$. Thus,  by Theorem \ref{negative one} we see that  $\mathbf{g}_{\xi,1+\nu}(C(M)\tim_\pi C(N)) \geq 1$. But this is a contradiction, since $1+\nu>0$.  Therefore $L$ must be infinite. 

Let $\zeta<\omega_1$ be such that $C(L)$ is isomorphic to $C(\omega^{\omega^\zeta})$ and let $\nu<\omega_1$ be such that $C(N)$ is isomorphic to $C(\omega^{\omega^\nu})$. Once again by our assumptions about the Cantor-Bendixson  indices we have $\zeta \leq \xi$.     Then according to Theorem \ref{positive one} we deduce that  $\mathbf{g}_{\xi,1+\zeta}(C(K)\tim_\pi C(L))<\infty$, and therefore $\mathbf{g}_{\xi,1+\zeta}(C(M)\tim_\pi C(N))<\infty$. This  implies that $\mathbf{g}_{\xi,\beta}(C(M)\tim_\pi C(N))=0$ for all $\beta>1+\zeta$.  But we also know by  Theorem \ref{negative one}   that $\mathbf{g}_{\xi,1+\nu}(C(M)\tim_\pi C(N)) \geq 1$, so $1+\nu \leqslant 1+\zeta$, and $\nu\leqslant \zeta$. By symmetry, $\zeta\leqslant \nu$, and $C(L)$ is isomorphic to $C(N)$. So  $(iii)$ holds and we are done.

\end{proof}

\begin{theorem} Let $K,L$ be countable compact  metric spaces such that $K$ is infinite and $CB(L)\leqslant CB(K)$.     The following are equivalent. \begin{enumerate}[(i)]\item There exists a compact, Hausdorff space $M$ such that $C(K)\tim_\pi C(L)$ and $C(M)$  are isomorphic to subspaces of quotients of each other. \item $L$ is finite. \end{enumerate}
\end{theorem}

\begin{proof} We observe that $M$ is necessarily infinite and as  in the proof of Theorem \ref{finish_line_1}, we deduce that $C(K)$ is isomorphic to $C(M)$, and each is isomorphic to $C(\omega^{\omega^\xi})$ for some countable ordinal $\xi$.    If $L$ were infinite, then $C(\omega)$ is isomorphic to a subspace of $C(L)$ and so $\mathbf{g}_{\xi,1}(C(K)\tim_\pi C(L))>0$.  It follows from Lemma \ref{knight}  that $\mathbf{g}_{\xi,0}(C(M))<\infty$,  and so $\mathbf{g}_{\xi,1}(C(M))=0$.  Since $C(K)\tim_\pi C(L)$ and $C(M)$ are isomorphic to subspaces of quotients of each other, $\mathbf{g}_{\xi,1}(C(K)\tim_\pi C(L))=0$, a contradiction. 
\end{proof}

The following corollary includes the  isomorphic classification of $C(\omega^{\omega^\xi})\tim_\pi C(\omega^{\omega^\zeta})$, $\zeta\leqslant \xi<\omega_1$ mentioned in Theorem \ref{mm}. 

\begin{corollary} \label{coro} Let $C(\omega^{\omega^{-1}})=C(\{0\})$.   Then, for each countable compact metric  spaces $K,L$ such that $K\cup L$ is infinite, $C(K)\tim_\pi C(L)$ is isomorphic to exactly one of the spaces $C(\omega^{\omega^\xi})\tim_\pi C(\omega^{\omega^\zeta})$, $\xi<\omega_1$, $-1\leqslant \zeta\leqslant \xi$. 
\end{corollary}

\end{document}